\documentclass[12pt, a4paper]{amsart}

\usepackage[utf8]{inputenc}
\usepackage[T1]{fontenc}
\usepackage{lmodern} 
\usepackage[left=2.5cm,right=2.5cm,top=2.5cm,bottom=2.5cm]{geometry}
\usepackage{amsmath,amsthm, amssymb,amsfonts,nicefrac}
\usepackage{mathabx} 
\usepackage{graphicx}
\usepackage{textcomp}
\usepackage[usenames, dvipsnames]{xcolor}
\usepackage{enumerate}
\usepackage[shortlabels]{enumitem}
\usepackage[colorlinks, linktocpage, citecolor = blue, linkcolor = blue,backref]{hyperref}
\usepackage{framed}
\usepackage{caption}
\usepackage{subcaption}
\usepackage{setspace}
\usepackage{xspace}
\usepackage[normalem]{ulem}
\usepackage{array}
\usepackage{verbatim}
\usepackage{aliascnt}

\usepackage[cmtip]{xy}
\xyoption{pdf}
\xyoption{color}
\xyoption{all}

\usepackage{tikz}
\usetikzlibrary{cd,arrows,positioning, decorations.pathreplacing}
\tikzset{>=stealth}
\tikzcdset{arrow style=tikz}
\tikzset{link/.style={column sep=1.8cm,row sep=0.16cm}}
\tikzset{link2/.style={column sep=0.4cm,row sep=0.1cm}} 
\tikzset{map/.style={row sep=0em, column sep=0em}}
\tikzset{c/.style={every coordinate/.try}}

\renewcommand{\to}{\longrightarrow}
\newcommand{\rat}{\dashrightarrow}

\newcommand{\Z}{\ensuremath{\mathbb{Z}}}

\newcommand{\F}{\ensuremath{\mathbb{F}}}

\newcommand{\C}{\ensuremath{\mathbb{C}}}
\newcommand{\p}{\ensuremath{\mathbb{P}}}
\renewcommand{\P}{\ensuremath{\mathbb{P}}}
\newcommand{\A}{\ensuremath{\mathbb{A}}}
\newcommand{\G}{\ensuremath{\mathbb{G}}}

\newcommand{\Al}{\mathcal{A}}
\newcommand{\El}{\mathcal{E}}
\newcommand{\Fl}{\mathcal{F}}
\newcommand{\Gl}{\mathcal{G}}
\newcommand{\Ll}{\mathcal{L}}
\newcommand{\Ml}{\mathcal{M}}
\newcommand{\Nl}{\mathcal{N}}

\newcommand{\Ol}{\mathcal{O}}

\DeclareMathOperator{\PGL}{PGL}

\DeclareMathOperator{\Aut}{Aut}
\DeclareMathOperator{\Bir}{Bir}
\newcommand{\Pic}{{\mathrm{Pic}}}
\newcommand{\NS}{{\mathrm{NS}}}
\newcommand{\rk}{{\mathrm{rk}}}
\DeclareMathOperator{\GL}{GL}
\DeclareMathOperator{\Cr}{Cr}
\DeclareMathOperator{\Gal}{Gal}
\DeclareMathOperator{\Slpec}{Spec}

\DeclareMathOperator{\seg}{\mathfrak{S}}
\newcommand{\Autz}{\ensuremath{\mathrm{Aut}^{\circ}}}

\renewcommand{\k}{\mathbf{k}}
\newcommand{\kk}{\mathbf{K}}

\makeatletter
\@ifdefinable\equationname{\let\equationname\equationautorefname}
\def\equationautorefname~#1\@empty\@empty\null{(#1\@empty\@empty\null)}%
\@ifdefinable\AMSname{\let\AMSname\AMSautorefname}
\def\AMSautorefname~#1\@empty\@empty\null{(#1\@empty\@empty\null)}%
\@ifdefinable\itemname{\let\itemname\itemautorefname}
\def\itemautorefname~#1\@empty\@empty\null{(#1\@empty\@empty\null)%
}%
\makeatother

\RequirePackage{amsthm}
\RequirePackage{aliascnt}
\newcommand{\basetheorem}[3]{%
	\newtheorem{#1}{#2}[#3]
	\newtheorem*{#1*}{#2}
	\expandafter\def\csname #1autorefname\endcsname{#2}
}%
\newcommand{\maketheorem}[3]{%
	\newaliascnt{#1}{#3}
	\newtheorem{#1}[#1]{#2}
	\aliascntresetthe{#1}
	\expandafter\def\csname #1autorefname\endcsname{#2}
	\newtheorem{#1*}{#2}
}%
\theoremstyle{plain}   
\basetheorem{theorem}{Theorem}{section}
\newtheorem*{thm*}{Theorem}
\maketheorem{corollary}{Corollary}{theorem}
\maketheorem{lemma}{Lemma}{theorem}
\maketheorem{proposition}{Proposition}{theorem}
\theoremstyle{definition}
\maketheorem{definition}{Definition}{theorem}
\maketheorem{defrmk}{Definition and Remark}{theorem}
\maketheorem{remark}{Remark}{theorem}
\maketheorem{example}{Example}{theorem}

\newtheorem{theoremA}{Theorem}

\newaliascnt{propositionA}{theoremA}
\newtheorem{propositionA}[propositionA]{Proposition}
\aliascntresetthe{propositionA}

\newaliascnt{corollaryA}{theoremA}

\aliascntresetthe{corollaryA}

\address{Department of Mathematics and Computer Science, University of Basel, Spiegelgasse 1, 4051 Basel, Switzerland}
\email{susanna.zimmermann@unibas.ch}

\address{Leibniz Universität Hannover, Institut für Algebraische Geometrie, Welfengarten 1, 30167 Hannover, Germany}
\email{fong@math.uni-hannover.de}


\thanks{
During this project, P.F. was supported by the project ANR ERCS-0013-002 and ERC StG Saphidir. 
S.Z. was supported by the ANR Project FIBALGA ANR-18-CE40-0003-01, the project ANR ERCS-0013-002, the Institut universitaire de France, the ERC StG Saphidir and the project SERI REF-1131-552105.
}

\begin{document}
	\author{Pascal Fong}
	\author{Susanna Zimmermann}
	\title[Automorphisms of Mori Del Pezzo fibrations]{Automorphism groups of Mori Del Pezzo fibrations over an irrational curve}

    \begin{abstract}
    We study the automorphism groups of Mori Del Pezzo fibrations over a smooth projective curve $C$ of positive genus. From that, we obtain a classification of maximal connected algebraic subgroups of $\Bir(C\times \p^2)$. Our results hold over any algebraically closed field of characteristic zero.
    \end{abstract}
    
	\maketitle

\setcounter{tocdepth}{1}
\tableofcontents

\section{Introduction}

The \emph{Cremona group} $\Cr_n(\k)$ is the group of birational transformations of the projective space $\p^n$ over the field $\k$. This group is known to be very large, see e.g. \cite{Blanc_Schneider_Yasinsky}. In fact, if $n>1$, then $\Cr_n(\k)$ is not an ind-algebraic group, see \cite{BlancFurter}.

This motivates the study of \emph{algebraic subgroups} of $\Cr_n(\k)$. The connected algebraic subgroups of $\Cr_2(\C)$ were classified by Enriques in \cite{enriques1893sui}: any such subgroup is contained in a maximal connected algebraic subgroup, and every maximal connected algebraic subgroup is conjugate to either $\Aut(\p^2) \simeq \PGL_3(\C)$ or $\Autz(\F_n)$, where $n \neq 1$ and $\F_n$ denotes the $n$-th Hirzebruch surface. The classification of maximal connected algebraic subgroups of $\Cr_3(\C)$ was later stated by Enriques and Fano in \cite{Enriques_Fano}, and a first classification was given by Umemura in \cite{Umemura80,Umemura82a,Umemura82b,Umemura85}. 
In this case, it again follows from the classification that every connected algebraic subgroup of $\Cr_3(\C)$ is contained in a maximal one.

Recently, in \cite{Fanelli_Floris_Zimmermann}, Fanelli-Floris-Zimmermann proved that the analogous property does not hold for $\Cr_n(\C)$ when $n\geq 5$: there exist connected algebraic subgroups that are not contained in any maximal one. Their result was extended to the case $n=4$ by Kollár in \cite{kollar_bundle}.

\medskip

Umemura's proof has been revisited and generalised to any algebraically closed field of characteristic zero by Blanc-Fanelli-Terpereau in \cite{BFT22,BFT23}. In these two articles, they use only techniques from algebraic geometry, such as the theory of algebraic groups and the \emph{Minimal Model Program} (MMP), whereas Umemura used a result of Lie that classifies analytic actions on complex threefolds (\cite[Theorem 1.12]{Umemura80}). 

\medskip

Throughout this article, we work over an algebraically closed field $\k$ of characteristic zero. Following the strategy of Blanc-Fanelli-Terpereau, we pursue an alternative direction, namely the classification of \emph{maximal connected algebraic subgroups of} $\Bir(C\times \p^2)$, where $C$ denotes a smooth projective curve of genus $g(C)\geq 1$. Running an MMP from a variety birationally equivalent to $C\times \p^2$, we obtain two types of Mori fibre spaces $\pi\colon X\to B$, according to the dimension of $B$:

\begin{enumerate}
\item Mori conic bundles over a projective surface $B$ birationally equivalent to $C\times \p^1$, see \autoref{lem:irrat generic fibre square}\autoref{square:2},
\item \emph{Mori Del Pezzo fibrations} over $B=C$, see \autoref{lem:preserves fibration}.
\end{enumerate}

Since $C$ is not rational, the case where $B=\Slpec(\k)$ (or equivalently, $X$ is a Fano variety), does not occur. The first case was studied mainly in \cite{Fong23}, where the first author classified connected algebraic subgroups that are maximal in $\Bir(X/B)$, with $B$ being a ruled surface. Such a subgroup is called \emph{relatively maximal}. 

\medskip

In this article, we focus on the automorphism groups of Mori Del Pezzo fibrations $\pi \colon X\to C$ and determine whether $\Autz(X)$ is a maximal connected algebraic subgroup of $\Bir(X)$. This leads to our \autoref{thm:mainDP} and \autoref{thm:main_superrigidity}. From these results, we deduce the classification of maximal connected algebraic subgroups of $\Bir(C\times \p^2)$, which gives \autoref{thm:main}.

\subsection{Main results}

The classification of vector bundles over an elliptic curve, due to Atiyah in \cite{Ati57}, states that there exist exactly $r$ indecomposable\footnote{This means that the projective bundle is the projectivisation of an indecomposable vector bundle $\El$, i.e., $\El$ is not isomorphic to the direct sum of two proper subbundles.} $\p^{r-1}$-bundles over an elliptic curve $C$ (see \autoref{thm:Atiyah} for $r=2,3$). These bundles are denoted $\Al_{r,d}$, and they are the projectivizations of indecomposable vector bundles $\El_{r,d}$ of rank $r$ and degree $d\in \{0,\ldots, r-1\}$.

\begin{theoremA}\label{thm:mainDP}
Let $\pi\colon X\to C$ be a Mori Del Pezzo fibration above a smooth projective curve $C$ of genus $g(C)\geq1$ and let $d=K_{X_{\k(C)}}^2$. 
If $\Autz(X)$ is a maximal connected algebraic subgroup of $\Bir(X)$, then exactly one of the following holds:
    \begin{enumerate}
    \item\label{mainDP:1} $g(C)=1$ and
        \begin{enumerate}[ref=\theenumi{}.\alph*]
        \item\label{mainDP:1.1} $d\in\{1,2,3,4,5,6\}$ and $\Autz(X)$ is an elliptic curve acting non-trivially on $C$.
        \item\label{mainDP:1.2} $d=8$ and $\Autz(X)\simeq\PGL_2(\k)$ acts diagonally on a general fibre of $\pi$.
        \item\label{mainDP:1.3} $d=9$ and $\Autz(X)$ acts non-trivially on $C$ and
            \begin{enumerate}[ref=\theenumii{}.\roman*]
            \item\label{mainDP:1.3i} $X$ is isomorphic to one of the indecomposable $\p^2$-bundles $\Al_{3,1}$ and $\Al_{3,2}$. In both cases, there exists a short exact sequence
            \[
            1 \to (\Z/3\Z)^2 \to \Autz(X) \to \Autz(C) \to 1.
            \]
            \item\label{mainDP:1.3ii} $X\simeq \p(\El_{2,0}\oplus \Ml)$, where $\Ml\in \Pic^0(C)$ has infinite order. In this case, there exists a short exact sequence
            \[
            1\to \G_m \rtimes \G_a \to \Autz(X) \to \Autz(C) \to 1. 
            \]
             \item\label{mainDP:1.3iii} $X\simeq \p(\Ol_C \oplus \Ll \oplus \Ml)$, where $\Ll\in \Pic^0(C)$ or $\Ml \in \Pic^0(C)$ is trivial, or both non-trivial and non-isomorphic such that $\Ll^{\otimes n} \otimes \Ml^{\otimes m}$ is not trivial for every $n,m \in \mathbb{Z}$ coprime. Then there exists a short exact sequence
             \[
             1 \to \Autz(X)_C \to \Autz(X) \to \Autz(C) \to 1,
             \]
             where 
             \[
             \Autz(X)_C \simeq
             \begin{cases}
                \PGL_3(\k) \text{ if $\Ll\simeq \Ml\simeq \Ol_C$ (equivalently, $X\simeq C\times \p^2$)},\\
                \GL_2(\k) \text{ if exactly one of $\Ll$ and $\Ml$ is trivial}, \\
                \G_m^2 \text{ if both $\Ll$ and $\Ml$ are non-trivial and non-isomorphic.}
             \end{cases}
             \]
            \end{enumerate}
        \end{enumerate}
    \item\label{mainDP:2} $g(C)\geq2$ and
        \begin{enumerate}[ref=\theenumi{}.\alph*]
        \item\label{mainDP:2.1} $d=8$ and $\Autz(X)\simeq\PGL_2(\k)$ acts diagonally on a general fibre of $\pi$.
        
        \item\label{mainDP:II.2} $d=9$ and 
        \begin{enumerate}[ref=\theenumii{}.\roman*]
            \item\label{mainDP:II.2i} $X\simeq C\times\p^2$ and $\Autz(X)\simeq \PGL_3(\k)$.
            \item\label{mainDP:II.2ii} $X\simeq \p(\El\oplus\Ml)$, where $\El$ is a rank-$2$ stable vector bundle fitting into a short exact sequence
        \[
        0\to\Ol_C\to\El\to\Ll\to0
        \]
        where $\Ll,\Ml$ are line bundles with $\deg(\Ll)=2\deg(\Ml)>0$.
        \end{enumerate}
        \end{enumerate}
    \end{enumerate}
\end{theoremA}

If $g(C)\geq 2$, notice that there are much less maximal connected algebraic subgroups of $\Bir(C\times \p^2)$. This is expected, since in dimension two, $\Autz(C\times \p^1)\simeq \PGL_2(\k)$ is the unique maximal connected algebraic subgroup of $\Bir(C\times \p^1)$ (see \cite[Theorem B]{Fong20}). It is also worth mentioning that we do not know whether  \autoref{thm:mainDP}\autoref{mainDP:II.2ii} actually occurs (see \autoref{rem:stable+line_bundle}).

\medskip

In \autoref{thm:main_superrigidity}, we investigate whether the pairs $(X,\pi)$ in \autoref{thm:mainDP} are \emph{superrigid} when $X$ is a $\p^2$-bundle (see \autoref{def:superrigid}), i.e., $\Autz(X)$ is a maximal connected algebraic subgroup of $\Bir(X)$ and there exists no $\Autz(X)$-equivariant birational map to another Mori fiber space. In the cases where $(X,\pi)$ is not superrigid, we provide a complete list of all such $\Autz(X)$-equivariant birational maps, except for \autoref{thm:mainDP}\autoref{mainDP:II.2ii}.

\begin{theoremA}\label{thm:main_superrigidity}
Let $\pi\colon X\to C$ be a $\p^2$-bundle above a smooth projective curve $C$ of genus $g(C)\geq1$. 
    Using the numbering from \autoref{thm:mainDP}\autoref{mainDP:1.3} (Case $g(C)=1$) and \autoref{mainDP:II.2} (Case $g(C)\geq2$), the following hold:
    \begin{enumerate}
    \item[\autoref{mainDP:1.3}] The Case $g(C)=1$: 
    \begin{enumerate} 
        \item[\autoref{mainDP:1.3i}] $(X,\pi)$ is superrigid.
        \item[\autoref{mainDP:1.3ii}] $(X,\pi)$ is not superrigid. The $\Autz(X)$ is conjugate to $\Autz(X')$, where \newline $\pi' \colon X' \to B$ is a Mori fibre space, if and only if one of the following holds:
        \begin{enumerate}[label=(\alph*)]
        \item There exist $b\in \mathbb{Z}$ and a $\p^1$-bundle structure $\pi'\colon X'\to \Al_{2,0}$, so that $B=\Al_{2,0}$, and
        such that 
        \begin{align*}
        X'&\simeq \P(\Ol_{\Al_{2,0}} \oplus \Ol_{\Al_{2,0}}(b\sigma)\otimes \tau^*(\Ml)),
        \end{align*}
        where $\sigma$ denotes the class of the unique minimal section of the ruled surface $\Al_{2,0}$.
        \item There exist $b\geq 0$, $n\in \{0,\ldots,b\}$, and a $\p^1$-bundle structure $\pi'\colon X'\to \P(\Ol_C \oplus \Ml)$, so that $B = \P(\Ol_C \oplus \Ml)$, and such that
        \[
           X'\simeq\Al_{(\P(\Ol_C \oplus \Ml),b,\Ml^{\otimes n})} 
        \]
         constructed in \autoref{ex:A(S,b,M)}.
        \item There exists a structure of $\p^2$-bundle $\pi'\colon X'\to C$, so that $B=C$, and such that $X'\simeq X$ or $$X'\simeq \P(\El_{2,0}\oplus \Ml^\vee) = :X^\vee.$$ 
    \end{enumerate}   
    \item[\autoref{mainDP:1.3iii}] $(X,\pi)$ is superrigid if and only if $\Ll\simeq \Ml \simeq \Ol_C$. 
        \begin{enumerate}[label=(\alph*)]
        \item If $\Ll$ is trivial but $\Ml$ is not, then $\Autz(X)$ is conjugate to $\Aut(X')$, where $\pi' \colon X' \to B$ is a Mori fibre space, if and only if one of the following holds:
             \begin{enumerate}[label=(\roman*)]
             \item $X=X'$ is a $\p^2$-bundle over $C$.
    
             \item There exists a $\p^1$-bundle structure $\pi'\colon X' \to C\times \p^1$, so that $B=C\times \p^1$, and such that
            $$X' \simeq \P(\Ol_{C\times \P^1} \oplus \Ol_{C\times \P^1}(\sigma) \otimes \tau^*(\Ml)),$$ 
             where $\sigma$ denotes the class of a constant section of $\tau\colon C\times \p^1\to C$.
            \end{enumerate}
     
    \item  If $\Ll$ and $\Ml$ are not trivial, then $\Autz(X)$ is conjugate to $\Aut(X')$, where $\pi' \colon X' \to B$ is a Mori fibre space, if and only if there exists 
    \[
    \begin{pmatrix}
    \alpha & \beta \\ \gamma & \delta
    \end{pmatrix}\in \mathrm{SL}_2(\Z) 
    \]
    such that $X'$ is isomorphic to one of the following Mori fibre space:
        \begin{enumerate}[label=(\roman*)]
         \item There exist $b\in \mathbb{Z}$ and a $\p^1$-bundle structure $\pi'\colon X'\to S$ such that 
        \[
        X' \simeq \P\left(\Ol_S \oplus (\Ol_S(b \sigma) \otimes \tau^*(\Ll^{\otimes \alpha} \otimes \Ml^{\otimes \beta}))\right),
        \]
        where $S:=\P(\Ol_C \oplus (\Ll^{\otimes \gamma} \otimes \Ml^{\otimes \delta})) = B$, and $\tau\colon S\to C$ denotes the $\p^1$-bundle structure morphism, and $\sigma$ is the section of $\tau$ corresponding to the line subbundle $\Ll^{\otimes\gamma}\otimes\Ml^{\otimes\delta}$.
        
        \item There exists a $\p^2$-bundle structure $\pi'\colon X'\to C$, so that $B=C$, and such that 
        \[X'\simeq \P(\Ol_C \oplus (\Ll^{\otimes \alpha} \otimes \Ml^{\otimes \beta}) \oplus (\Ll^{\otimes \gamma} \otimes \Ml^{\otimes \delta})).\]
         \end{enumerate}
    \end{enumerate}
    \end{enumerate}
    \item[\autoref{mainDP:II.2}] The case $g(C)\geq 2$: \autoref{mainDP:II.2i} $(X,\pi)$ is superrigid.
    \end{enumerate}
\end{theoremA}

For the case \autoref{thm:mainDP}\autoref{mainDP:II.2ii}, we do not know whether the pair $(X,\pi)$ is superrigid. Moreover, in \autoref{thm:main_superrigidity}, when $g(C)=1$ and the pair $(X,\pi)$ is not superrigid, then there exists an $\Autz(X)$-equivariant birational map $X\dashrightarrow X'$, that conjugates $\Autz(X)$ and $\Autz(X')$, where $X'$ is a $\p^1$-bundle over a ruled surface birationally equivalent to $C\times \p^1$. This connects the automorphism groups of $\p^2$-bundles studied in this article, with the automorphism groups of $\p^1$-bundles classified in \cite{Fong23}.

\medskip

If $\pi\colon X\to C$ is a Mori Del Pezzo fibration of degree $d\leq 4$, then $X$ is not birational to $ C\times \p^2$ (see \autoref{prop:MdP5}). Combining \autoref{thm:mainDP} and the classification of relatively maximal automorphism groups of $\P^1$-bundles in \cite[Theorem A and C]{Fong23}, we obtain the maximal connected algebraic subgroups in $\Bir( C\times \p^2)$ when $g(C)\geq1$.

\begin{theoremA}\label{thm:main}
Let $G$ be a maximal connected algebraic subgroup of $\Bir( C\times \p^2)$, where $C$ is a smooth projective curve of genus $g(C)\geq1$. Then there exists a $G$-equivariant birational map $\varphi\colon  C \times \p^2\rat X$ to a Mori fibre space $\pi\colon X\to B$, such that $\varphi G\varphi^{-1}=\Autz(X)$, and exactly one of the following holds:
    \begin{enumerate}
    \item[(I)] $g(C)=1$ and  
        \begin{enumerate}
        \item[(1)] $B=C$ and $X$ is as in \autoref{thm:mainDP}\autoref{mainDP:1} with $d\in \{6,8,9\}$;
        \item[(2)] $X$ is isomorphic to one of the following $\p^1$-bundles: 
        \begin{enumerate}[(i)]
        \item $S'\times \p^1$, where $\tau'\colon S'\to C$ is one of the following ruled surfaces:
        \begin{enumerate}[(a)]
            \item $C\times \p^1$,
            \item $\Al_{2,0}$,
            \item $\Al_{2,1}$,
            \item $\p(\Ol_C\oplus \Ll)$, where $\Ll\in \Pic^0(C)$ is non-trivial.
        \end{enumerate}
        For each $S'\times \p^1$ above, there exist two $\p^1$-bundle structures, namely the trivial one over $S'$ and the morphism $\tau' \times id_{\p^1}$. We have $B= S'$ and $B=C\times \p^1$, respectively.

        \item $\Al_{2,1} \times_C \Al_{2,1}$, equipped with the projection onto either factor, so that $B=\Al_{2,1}$.

        \item $\p(\Ol_C\oplus \Ll) \times_C \Al_{2,1}$, 
        where $\Ll\in \Pic^0(C)$ is not two-torsion. There exist two $\p^1$-bundle structures, which are the projections onto either factor, so that $B=\p(\Ol_C\oplus \Ll)$ and $B=\Al_{2,1}$, respectively.

        \item $\p(\Ol_{C\times \p^1}\oplus \Ol_{C\times \p^1}(b\sigma +\tau^*(D)))$ is a decomposable $\p^1$-bundle over $C\times \p^1$, where $\sigma$ a constant section in $C\times \p^1$, $b\geq2$, and $D\in \Pic^0(C)$. In this case, $B=C\times \p^1$.
        
        \item $\p(\Ol_{\Al_{2,1}} \oplus \Ol_{\Al_{2,1}}(2\sigma +\tau^*(D)))$ is a decomposable $\p^1$-bundle over $\Al_{2,1}$, 
        where $\sigma\subset \Al_{2,1}$ is a minimal section and $D$ is a non-trivial $2$-divisor on $C$ (see \autoref{lem:2-divisor}). In this case, $B=\Al_{2,1}$.
    \end{enumerate}
        \end{enumerate}
    \item[(II)] $g(C)\geq2$ and 
        \begin{enumerate}
        \item[(1)] $X$ is as in \autoref{thm:mainDP}\autoref{mainDP:2}.
        \item[(2)] $X=C\times\p^1\times\p^1$ and $\pi$ is the trivial $\p^1$-bundle, so that $B=C\times \p^1$.
        \end{enumerate}
    \end{enumerate}
\end{theoremA}

\subsection{Plan of the article}

This article is structured as follows.

\medskip

In \autoref{s:preliminaries}, we recall generalities on Del Pezzo surfaces and Mori Del Pezzo fibrations. We also analyse the Sarkisov links starting from Mori Del Pezzo fibrations.
The \emph{Sarkisov Program} is an algorithm that decomposes every birational map between Mori fibre spaces into \emph{Sarkisov links}; it was established by Corti in dimension two and three, by Hacon and McKernan in higher dimensions and by Floris for the equivariant version \cite{corti1995factoring,HMcK,Flo20}. 
This program serves as our main tool to determine the maximality of the automorphism group of a Mori fibre space.

\medskip

Let $\pi\colon X\to C$ be a Mori Del Pezzo fibration, where $C$ is a smooth projective curve of genus $g(C)\geq 1$. The integer $d=K_{X_{\k(C)}}^2 \in \{1,\ldots,9\}$ is called the \emph{degree} of $\pi$. 

\medskip

Inspired by the ideas in \cite{BFT22}, we study the case $d=8$ in \autoref{s:MdP8}. If $C$ is an elliptic curve, this case provides the only examples of $\Autz(X)$ that are maximal connected algebraic subgroups of $\Bir(X)$, while acting trivially on $C$. Next, in \autoref{s:finite}, we focus on the cases where $d\leq 5$, which share the property that $\Autz(X)_C$ is finite (see \autoref{lem:auto of dp surface}). Among these, only the degree $d=5$ cases are birational to $C\times \p^2$ (see \autoref{prop:MdP5}). In this situation, however, $\Autz(X)$ is trivial (see \autoref{mainprop:MDP5}). We then consider the case $d=6$ in \autoref{s:MdP6}. 
The case $d=7$ does not occur, since $\rho(X/C)>1$ (see \autoref{lem:auto of dp surface} and \autoref{rmk: MdP}). 
Taken together, these analyses lead to the following statements:

\begin{propositionA}\label{thm:K<5}
Let $\pi\colon X\to C$ be a Mori Del Pezzo fibration above a curve $C$ of genus $g(C)\geq1$ such that $\Autz(X)$ is non-trivial and suppose that $d:=K_{X_{\k(C)}}^2\in\{1,2,3,4,6\}$. Then $g(C)=1$ and $\Autz(X)$ is an elliptic curve acting non-trivially on $C$ and it is a maximal connected algebraic subgroup of $\Bir(X)$. Moreover, the following hold: 
\begin{enumerate}
    \item $X\simeq\Autz(X)\times^{\Autz(X)_C}\pi^{-1}(0)$, where $\pi^{-1}(0)$ is a Del Pezzo surface of degree $d$.
    \item If $d\neq4$, then any $\Autz(X)$-equivariant birational map $\varphi\colon X\rat Y$ to another Mori fibre space $\pi'\colon Y\to B$ satisfies that $\pi'$ is a Mori Del Pezzo fibration of degree $d$.
    \item If $d=6$ and $\varphi\colon X\rat Y$ is an $\Autz(X)$-equivariant birational map to another Mori fibre space then $Y\simeq X$.
\end{enumerate}
\end{propositionA}

\begin{propositionA}\label{mainprop:MDP5}
Let $\pi\colon X\to C$ be a Mori Del Pezzo fibration of degree $5$ above a curve $C$ of genus $g(C)\geq 1$, such that $\Autz(X)$ is non-trivial. 
Then $g(C)=1$ and there is a short exact sequence 
\[
1\to \Z/5\Z\to\Autz(X)\to\Autz(C)\to 1.
\]
Moreover, $\Autz(X)$ is a maximal connected algebraic subgroup of $\Bir(X)$ if and only if $X$ does not contain any $\Autz(X)$-invariant section.
\end{propositionA}

The case $d=9$ is treated in two steps. First, in \autoref{s:MdP9}, we generalize a result from \cite{BFT22}, and show that one can always reduce to the cases of $\p^2$-bundles (\autoref{prop:Mdp to bundle}). If $g(C)=1$ and $\Aut(X)$ is maximal, we prove that $\Autz(X)$ acts transitively on $C$ (\autoref{prop:trivial_action_on_C}). In particular, $X$ arises as the projectivization of a semi-homogeneous rank-$3$ vector bundle over $C$ (see \autoref{lem:pi-star surjective}). If $\pi$ is an indecomposable $\p^2$-bundle, then $X\simeq \Al_{3,d}$ with $d\in \{0,1,2\}$, which we study in \autoref{prop:A31 and A32} and \autoref{pro:E30}. 

\medskip

In \autoref{s:higher genus}, we focus on the case where $\pi\colon X\to C$ is a $\p^2$-bundle and $g(C)\geq 2$. Unlike the case where $g(C)=1$, it is out of reach to study the automorphism groups of all indecomposable projective bundles. However, because $g(C)\geq2$, $\Autz(C)$ is trivial, so $\Autz(X)$ acts trivially on $C$. Using the theory of vector bundles, and particularly the existence and uniqueness of the \emph{maximal destabilizing subbundle} and of the \emph{socle} (resp. for unstable and semistable bundles, see \autoref{lem:destabilizing+socle}), we prove that there are only very few maximal connected algebraic subgroups of $\Bir(C\times \p^2)$ (see \autoref{lem:semistable_high_genus}).

\medskip

Gathering the results on the Mori Del Pezzo fibrations of each degree, we prove \autoref{thm:mainDP} and \autoref{thm:main_superrigidity} in \autoref{s:proofAB}.

\medskip

Finally, in \autoref{s:to_proj_bundles}, we show that every maximal connected algebraic subgroup of $\Bir(C\times \p^2)$ is conjugate to $\Autz(X)$, where $\pi\colon X\to C$ is either a $\p^1$-bundle over a ruled surface studied in \cite{Fong23}, or a Mori Del Pezzo fibration over $C$. This allows us to prove \autoref{thm:main}. As by-product of our studies, we also obtain the following statements: 

\medskip

\begin{propositionA}\label{prop:iscf main}
Let $X\to S$ be a Mori fibre space above a surface $S$ with a morphism $\tau\colon S\to C$ to a curve $C$ of genus $g(C)\geq1$ and $\tau_*\Ol_S=\Ol_C$. Suppose that the generic fibre $X_{\k(S)}$ is irrational and that $\Autz(X)$ is non-trivial. 
Then $g(C)=1$, and $\Autz(X)$ is an elliptic curve, and $\tau$ is a ruled surface, and the following hold:
    \begin{enumerate}
    \item If $K_{X_{\k(C)}}^2\leq4$, then $\Autz(X)$ a maximal connected algebraic subgroup of $\Bir(X)$. 
    \item If $K_{X_{\k(C)}}^2\geq5$, then $K_{X_{\k(C)}}^2=5$ and $\Bir(X)$ is not a maximal connected algebraic subgroup of $\Bir(X)$.
    \end{enumerate}
\end{propositionA}

\begin{propositionA}\label{thm:Mfs K(S)-pos generic curve irrational}
Let $X\to S$ be a Mori fibre space above a surface $S$ of Kodaira dimension $\kappa(S)$, such that $X_{\k(S)}$ is an irrational curve of genus zero. 
If $\kappa(S)\geq0$, then $X$ is not birational to a Mori Del Pezzo fibre space. Moreover, if $\Autz(X)$ is non-trivial, then $\Autz(X)$ is an elliptic curve and a maximal connected algebraic subgroup of $\Bir(X)$.
\end{propositionA}

We considered this paper too long to include a study of $\p^1$-bundles over surfaces $S$ with $\kappa(S)\geq0$.

\subsection*{Acknowledgements} The authors thank Fabio Bernasconi, Michel Brion, Daniele Faenzi, Andrea Fanelli, Enrica Floris, and Julia Schneider for helpful discussions related to this project.

\subsection{Notations and conventions}
Throughout this article, we use the following notations and conventions. 
\begin{itemize}
\item The base-field is denoted by $\k$ and is algebraically closed and of characteristic zero, unless stated otherwise. 
\item Unless otherwise stated, $C$ denotes a smooth projective curve.
\item For a projective variety $X$, we denote by $\Aut(X)$ the automorphism group of $X$, which is a group scheme \cite[Theorem 3.6]{Matsumura_Oort}. The connected component containing the identity, denoted $\Autz(X)$, is an algebraic group \cite[Theorem 3.7]{Matsumura_Oort}.
\item Let $\pi\colon X\to B$ be a morphism of projective varieties such that $\pi_*(\Ol_X) \simeq \Ol_B$. Then $\pi$ induces a morphism a algebraic groups $\pi_* \colon \Autz(X) \to \Autz(B)$ (see \autoref{lem:blanchard}). We denote by $\Autz(X)_C$ the kernel of $\pi$.
\item For varieties $X$ and $Y$, we say that $\Autz(X)$ is conjugate to $\Autz(Y)$ if there is a birational map $\varphi\colon X\rat Y$ such that $\varphi\Autz(X)\varphi^{-1}=\Autz(Y)$.
\end{itemize}

\section{Preliminaries}\label{s:preliminaries}

\subsection{Rationally connectedness and sections}

\begin{remark}\label{lem:not Fano}
Let $C$ be a smooth projective curve of genus $g(C)\geq1$ and let $X$ be birational to a Mori fibre space $\pi\colon Y\to C$. Then $X$ is not Fano. 
Indeed, any Fano variety is rationally connected \cite[Theorem 0.1]{KMM92}, and rational connectedness is invariant under birational maps. So, if $X$ is Fano, then $Y$ is rationally connected and hence $C$ is rationally connected, against $g(C)\geq1$.
\end{remark}

\begin{lemma}[{\cite[Corollary 1.3]{GHS02}}]\label{thm:rationally connected}
Let $\hat\pi\colon X\to B$ be a dominant morphism. If $B$ and the general fibre $f$ of $\hat\pi$ are rationally connected, then $X$ is rationally connected. 
\end{lemma}

\begin{lemma}[{\cite[Theorem 1.1]{GHS02}}]\label{thm:Tsen}
Let $X$ be a projective variety and $\pi\colon X\to B$ be a dominant morphism to a smooth projective curve $B$ such that the general fibre is rationally connected. Then $\pi$ has a rational section. 
\end{lemma}

\subsection{Repetition of Del Pezzo surfaces}

In this section, we recall classical facts on Del Pezzo surfaces.  Let $\kk$ be any field of characteristic zero and $\overline\kk$ its algebraic closure. (Later on, we will have $\kk=\k(C)$ for some curve $C$ over an algebraically closed field $\k$ of characteristic zero.)

Recall that a Del Pezzo surface is a smooth surface $S$ such that $-K_S$ is ample. We call $K_S^2$ its degree, and $1\leq K_S^2\leq9$. Then $S_{\overline\kk}\simeq\p^1_{\overline\kk}\times\p^1_{\overline\kk}$ or is the blow-up of $\p^2_{\overline\kk}$ in $0\leq r\leq 8$ points in general position (i.e. no infinitely near points, no three points on a line, no six points on a conic, no eight points on a singular cubic where one is the singular point of the cubic). In particular, $S_{\overline\kk}$ has only finitely many curves of negative self-intersection, and each such curve has self-intersection $-1$. Moreover, $\Aut(S)$ is an algebraic group, see for instance \cite[Lemma 2.10(1)]{RZ18}.

\begin{lemma}\label{lem:auto of dp surface}
Let $S$ be a Del Pezzo surface over $\kk$. 
\begin{enumerate}
\item\label{auto dp surface:1} If $K_S^2\leq5$, then $\Aut(S)$ is finite. 
\item\label{auto dp surface:2} If $K_S^2=6$, then $\Aut(S_{\bar\kk})\simeq \G_m^2\rtimes(\mathrm{Sym}_3\times\Z/2)$. 
\item\label{auto dp surface:3} If $K_S^2=7$, then there is a birational morphism $S\to \p^2$. 
\end{enumerate}
\end{lemma}
\begin{proof}
We have $\Aut(S)_{\overline\kk}\simeq\Aut(S_{\overline\kk})$ by \cite[Proposition 2.20]{TZ24}.
If $K_S^2\leq5$, then $S_{\overline\kk}$ is the blow-up $\eta\colon S\to\p^2$ of $r\geq4$ points $p_1,\dots,p_r$ in general position and it has only finitely many $(-1)$-curves. The group $\Autz(S_{\overline\kk})$ preserves each of them, in particular the exceptional divisors of $\eta$.
Then $\eta\Autz(S_{\overline\kk})\eta^{-1}\subset\Aut(\p^2_{\overline\kk})$ fixes $p_1,\dots,p_r$. Since these points are in general position and $r\geq4$, the group $\eta\Autz(S_{\overline\kk})\eta^{-1}$ is trivial. So, $\Autz(S)$ is trivial.
Suppose that $K_S^2=6$.  Then $S_{\overline\kk}$ is isomorphic to the blow-up $\eta\colon S_{\overline\kk}\to\p^2_{\overline\kk}$ of the points $[1:0:0],[0:1:0],[0:0:1]$. The six $(-1)$-curves on $S_{\overline\kk}$ are arranged in a hexagon on $S_{\overline\kk}$ and $\Aut(S_{\overline\kk})$ acts on it. 
This action induces a short exact sequence
\[
1\to G\to \Aut(S_{\overline\kk})\to\mathrm{Sym}_3\times\Z/2\to 1.
\]
and $\eta G\eta^{-1}\simeq\G_m^2$. Moreover, the sequence is split: the lift of the standard quadratic involution of $\p^2_{\overline\kk}$ generates $\Z/2$ and the automorphisms of $\p^2_{\overline\kk}$ with coefficients in $\{0,1\}$ permuting the standard coordinates lift to $\mathrm{Sym}_3$. 
For $K_S^2=7$, the claim is proven in \cite[Lemma 4.15]{BFSZ24}.
\end{proof}

\begin{proposition}[{\cite[Theorem 2.6]{iskovskikh_1996}}]\label{prop:dP<5 rational}
Let $S$ be a Del Pezzo surface over $\kk$ with $\rho(S)=1$. 
\begin{enumerate}
\item If $K_S^2\geq5$, then $S$ is $\kk$-rational if and only if $S(\kk)\neq\varnothing$. 
\item If $K_S^2\leq4$, then $S$ is not $\kk$-rational. 
\end{enumerate}
\end{proposition}

\begin{lemma}[{\cite[Theorem 2.6]{iskovskikh_1996}}]\label{lem:rho=1,K<4} 
Suppose that $S$ is a Del Pezzo surface over $\kk$, $G\subset\Aut(S)$ a finite subgroup with $\Pic(S)^G=1$, and $\varphi\colon S\dashrightarrow S'$ a $G$-equivariant birational map to a $G$-Mori fibre space $S'\to B$. Assume that $K_S^2\leqslant 4$. Then one has the following:
    \begin{enumerate}
    \item If $K_S^2=1$, then $\varphi$ is an isomorphism and $\rho(S')=1$.
    \item If $2\leq K_S^2\leq3$, then $\varphi$ is a composition of links of type II and $K_{S'}^2=K_S^2$.
    \item If $K_S^2=4$, then $\varphi=\varphi_n\circ\cdots\circ\varphi_1$, where $\varphi_i\colon S_i\rat S_{i+1}$ (where $S=S_1,S'=S_{n+1}$) is a link of type I with $K_{S_i}^2=4,K_{S_{i+1}}^2=3$, or $\varphi_i$ is a link of type III with $K_{S_i}^2=3,K_{S_{i+1}}^2=4$, or $\varphi_i$ is a link of type II between conic bundles and $K_{S_i}^2=K_{S_{i+1}}^2=3$, or $\varphi_i$ is a link of type II between Del Pezzo surface and 
    $K_{S_{i+1}}^2=K_{S_{i}}^2=4$.
    \end{enumerate}
\end{lemma}

\subsection{Generalities on Mori Del Pezzo fibrations}\label{s:generalities}

First, let us recall the following fundamental statement. 

\begin{lemma}[Blanchard's lemma {\cite[Proposition 4.2.1]{BSU13}}]\label{lem:blanchard}
Let $G$ be a connected algebraic group, $X$ a $G$-variety, and $f \colon X \to Y$ a
proper morphism such that $\Ol_Y=f_*\Ol_X$. Then there is a unique regular action of $G$ on $Y$ such that $f$ is
equivariant.
\end{lemma}

\begin{definition}
A Mori Del Pezzo fibration is a Mori fibration $\pi\colon X\to C$ above a curve $C$ whose generic fibre $X_{\k(C)}$ is a surface. If $d:=K_{X_{\k(C)}}^2$, we say that $\pi$ is a Mori Del Pezzo fibration of degree $d$.  
\end{definition}

Let $\pi\colon X\to B$ be a surjective morphism between normal varieties. 
We denote by $N_1(X/B)$ the subspace of $N_1(X)$ generated by the curves contracted by $\pi$. 
Its dimension $\rho(X/B)$ is called the relative Picard rank of $\pi$.

\begin{remark}\label{rmk: MdP}
Let $\pi\colon X\to C$ be a Mori Del Pezzo fibration above a curve $C$. 
\begin{enumerate}
\item\label{MdP:-1} Since the curve $C$ is normal, it is smooth.
\item\label{MdP:0} $X_{\k(C)}$ is smooth by \cite[Lemma 15.1]{FS20} and is hence a Del Pezzo surface. It has a rational point by \autoref{thm:Tsen}.
\item\label{MdP:4} A general fibre of $\pi$ is a (smooth) Del Pezzo surface of degree $K_{X_{\k(C)}}^2$.
\item\label{MdP:1} Since $\rho(X/C)=1$, we have $\rho(X_{\k(C)})=1$. In particular, $K_{X_{\k(C)}}^2\neq 7$ by \autoref{lem:auto of dp surface}\autoref{auto dp surface:3}. 
\item\label{MdP:3} By \autoref{lem:blanchard}, the action of $\Autz(X)$ on $C$ induces a morphism $\pi_*\colon \Autz(X)\to \Autz(C)$ of algebraic groups and a short exact sequence
\begin{equation}\tag{ses}\label{eq:seq}
1\to \Autz(X)_C\to \Autz(X) \overset{\pi_*}{\to} \Autz(C).
\end{equation}
\item\label{MdP:2} There is an injective group homomorphism $\Autz(X)_C\hookrightarrow\Aut(X_{\k(C)})$, see for instance \cite[Remark 2.1.6]{BFT22}. 
\end{enumerate}
\end{remark}

\begin{proposition}\label{prop:MdP5}
Let $\pi\colon X\to C$ be a Mori Del Pezzo fibration above a curve $C$. The following hold:
\begin{enumerate}
\item\label{MdP5:1} If $K_{X_{\k(C)}}^2\geq5$, $X$ is birational to $C\times\p^2$ over $C$. 
\item\label{MdP5:2} If $g(C)\geq1$ and there exists a birational map $\varphi\colon X\rat C\times\p^2$, then there is an isomorphism $\alpha\colon C\to C$ such that $\alpha\circ\pi=pr_C\circ\varphi$ and $K_{X_{\k(C)}}^2\geq5$. 
\end{enumerate}
In particular, if $g(C)\geq 1$, then $X$ is birational to $C\times \p^2$ if and only if $K_{X_{\k(C)}}^2\geq5$.
\end{proposition}
\begin{proof}
\autoref{MdP5:1} By \autoref{rmk: MdP}\autoref{MdP:0}, $X_{\k(C)}$ has a rational point. Therefore, if $K_{X_{\k(C)}}^2\geq5$, then there is a birational map $X_{\k(C)}\rat \p^2_{\k(C)}$ by \autoref{prop:dP<5 rational}. It induces a birational map $X\rat C\times\p^2$ over $C$. 

\autoref{MdP5:2} Let $S\subset X$ be a general fibre of $\pi$. It is rational and hence its image $\varphi(S)$ is a rational surface. Since $g(C)\geq1$, it follows that $pr_C(\varphi(S))$ is a point.
Therefore, there exists a morphism $\alpha\colon C\to C$ such that $pr_C\circ\varphi=\alpha\circ\pi$. Since $\pi$ and $pr_C$ have connected fibres, $\alpha$ is an isomorphism.
Then $\varphi$ induces a birational map $X_{\k(C)}\rat\p^2_{\k(C)}$. Since $\pi$ is a Mori fibre space, we have $\rho(X/C)=\rho(X_{\k(C)})=1$, and now \autoref{prop:dP<5 rational} implies that $K_{X_{\k(C)}}^2\geq5$.
\end{proof}

\subsection{Sarkisov links starting from Mori Del Pezzo fibrations}

We recall the follow statement: 

\begin{theorem}[{\cite[Theorem 1.2]{Flo20}}]\label{thm:Enrica}
Let $G$ be a connected algebraic group acting on Mori fibre spaces $\pi\colon X\to B$ and $\pi'\colon X'\to B'$. Then any $G$-equivariant birational map $f\colon X\rat X'$ is the composition of $G$-equivariant Sarkisov links. 
Moreover, the relative version of the statement holds as well.
\end{theorem}

\begin{proposition}\label{lem:links I}
Let $\pi\colon X\to C$ a Mori Del Pezzo fibration above a curve of genus $g(C)\geq0$. 
Suppose there is a link of type I $\varphi\colon X\rat Y$ to a Mori fibre space $\pi'\colon Y\to B$ giving the following Sarkisov diagram:
\[
\begin{tikzcd}
Y'\ar[d,"div",swap]\ar[rr,dotted,"\psi"] && Y\ar[d,"\pi'"]\\
X\ar[dr,"\pi"]\ar[urr,"\varphi",dashed]&&B\ar[dl,"\theta"']\\
&C&
\end{tikzcd}
\]
Then $Y_{\k(C)}$ is a Del Pezzo surface with $\rho(Y_{\k(C)})=2$ and the divisorial contraction $div$ contracts a divisor onto a curve $\ell$ that is transversal to a general fibre of $\pi$. 
Moreover, one of the following holds:
\begin{enumerate}
\item $X_{\k(C)}\simeq\p^2_{\k(C)}$ and $ \deg(\pi|_{\ell})=1$ or $\deg (\pi|_{\ell})=4$;
\item $K_{X_{\k(C)}}^2=8$ and $\deg (\pi|_{\ell})=2$;
\item $K_{X_{\k(C)}}^2=4$ and $\deg (\pi|_{\ell})=1$.
\end{enumerate} 
\end{proposition}
\begin{proof}
We have $\rho(Y/C)=\rho(Y'/C)=2$ and hence $\rho(Y_{\k(C)})\leq2$.
Since $\dim C=1$ and $\theta\colon B\to C$ is a surjective morphism of Picard rank $\rho(B/C)=1$ with connected fibres, we have $\dim B=2$. 
Since the fibres of $\pi'$ are connected as well, we have $\rho(Y_{\k(C)})>1$ and hence $\rho(Y_{\k(C)})=2$. 
The base-locus of $\psi$ is contained in the fibres of $\pi\circ div$, because $\varphi$ preserves the fibration above $C$. Thus $Y_{\k(C)}\simeq Y'_{\k(C)}$.
Since $1=\rho(X_{\k(C)})<\rho(Y'_{\k(C)})=2$, it follows that $\ell$ is transversal to a general fibre of $\pi$. In particular, $\varphi$ induces the blow-up $\hat\varphi\colon X_{\k(C)}\rat Y_{\k(C)}$ of the closed point $\ell_{\k(C)}$.
The morphism $\theta\colon B\to C$ is a klt-Mori fibre space \cite[Lemma 3.13]{BLZ21}. 
In particular, the general fibre $f$ of $\theta$ is isomorphic to $\p^1$. It follows that $B_{\k(C)}$ is a rational curve and $Y_{\k(C)}\to B_{\k(C)}$ is a Mori fibre space. Then $\hat\varphi$ is a Sarkisov link of type I and hence $Y_{\k(C)}$ is a Del Pezzo surface \cite[Lemma 3.10]{BFSZ24}. Moreover, \cite[Theorem 2.6(i)]{iskovskikh_1996} yields the possibilities for $K_{X_{\k(C)}}^2$ and $\deg(\pi|_{\ell})$.
\end{proof}

\begin{lemma}\label{lem:preserves fibration}
Let $\pi\colon X\to C$ a Mori Del Pezzo fibration above a curve $C$ of genus $g(C)\geq1$. Let $\psi\colon X\rat Y$ be a birational map  to a Mori Del Pezzo fibration $\hat\pi\colon Y\to B$ above a curve $B$. Then there exists an isomorphism $\psi'\colon C\to B$ such that $\hat\pi\circ\psi=\psi'\circ\pi$. 

In other words, up to isomorphisms of $X$ and $Y$, $\psi$ is a composition of Sarkisov links of type II.
\end{lemma}

\begin{proof}
First, if $B$ has genus $g(B)=0$, then $Y$ is rationally connected by \autoref{thm:rationally connected}, a contradiction to $X$ not being rationally connected. So, $B$ has genus $g(B)\geq1$. 
Let $F$ be a general fibre of $\pi$ and $\tilde F$ its image by the birational map $\psi\colon X\rat Y$. Since $F$ is a Del Pezzo surface, it is rationally connected, and so $\hat\pi(\tilde F)$ is rationally connected. Since $g(B)\geq1$, it follows that $\hat\pi(\tilde F)$ is a point. 
In particular, $\psi$ maps a general fibre $F$ of $\pi$ to a general fibre of $\hat\pi$ and $\psi|_F\colon F\rat \psi(F)$ is birational. 

Let $s\subset X$ be a general section of $\pi$ (it exists by \autoref{thm:Tsen}). 
Its image $\tilde s$ in $Y$ is a curve with $\tilde F\tilde s=1$. Then $\tilde s$ is a rational section of $\hat\pi$. Since $C,B$ are both smooth curves, the induced birational map $\psi':=\hat\pi\circ\psi|_C\colon C\rat \tilde s\to B$ is an isomorphism satisfying $\hat\pi\circ\psi=\psi'\circ\pi$. 
\end{proof}

\begin{proposition}\label{lem:type II}
Let $\pi\colon X\to C$ a Mori Del Pezzo fibration above a curve $C$ of $g(C)\geq0$ and $\varphi\colon X\rat Y$ a Sarkisov link of type II to a Mori Del Pezzo fibre space $\pi'\colon Y\to C$, giving the following Sarkisov diagram:
\[
\begin{tikzcd}
Z\ar[d,swap,"div"]  \ar[rr,dotted,"\chi"] && Z' \ar[d,"div'"] \\
X \ar[rr,"\varphi",dashed] \ar[dr,"\pi",swap] & & Y \ar[dl,"\pi'"] \\
& C &
\end{tikzcd}
\]
If the base-loci of $div$ and $div'$ is generically transversal to the fibres above $\pi$ and $\pi'$, respectively, then $\varphi$ induces a link of type II
\[
\begin{tikzcd}
&Z_{\k(C)}\stackrel{\hat\chi}\simeq Z'_{\k(C)}\ar[ld,"\widehat{div}"']\ar[rd,"\widehat{div'}"]&\\
X_{\k(C)}\ar[r]&\ast&Y_{\k(C)}\ar[l]
\end{tikzcd}
\]
between Del Pezzo surfaces, where $\widehat{div},\widehat{div'},\hat\chi$ are induced by $div,div',\chi$, respectively. In particular, $Z_{\k(C)}$ is a Del Pezzo surface.
\end{proposition}

\begin{proof}
Since $\varphi$ preserves the fibration above $C$, the base-locus of $\chi$ is contained in the fibres above $C$. 
Then $\chi\circ div$ induces the blow-up $\widehat{div}\colon Z_{\k(C)}\to X_{\k(C)}$ of the closed point $\ell_{\k(C)}$. 
The same holds for $\varphi^{-1}$. In particular, $\varphi$ induces a link of type II as claimed. 
\cite[Lemma 3.10]{BFSZ24} shows that $Z_{\k(C)}$ is a Del Pezzo surface. 
\end{proof}

\begin{proposition}\label{lem: links III and IV}
Let $\pi\colon X\to C$ a Mori Del Pezzo fibration above a curve $C$ of genus $g(C)\geq1$. Then there are no links of type III and no link of type IV starting from $X$. 
\end{proposition}
\begin{proof}
A link $\psi\colon X\rat Y$ of type III and a link of type IV satisfy respectively commutative diagrams
\[
\begin{tikzcd}
X \ar[ddrr,"\psi",dashed,swap] \ar[dd,"\pi",swap]  \ar[rr,dotted] && Z \ar[dd,"div"] \\ \\
C \ar[dr,"\tau"] & & Y \ar[dl,"\hat\pi"] \\
& B &
\end{tikzcd}
\qquad
\begin{tikzcd}
X\ar[d,"\pi"]\ar[rr,dotted,"\psi"]&&Y\ar[d,"\pi'"']\\
C\ar[rd,"\theta_1"]&& B\ar[ld,"\theta_2"']\\
&D&
\end{tikzcd}
\]
where $div$ is a divisorial contraction, the dotted arrow is a pseudo-isomorphism and $\tau,\theta_1$ and $\theta_2$ are morphisms of relative Picard rank one. In particular, in the left-hand side diagram, $B$ is a point, and in the right-hand side diagram, $D$ is a point. By \cite[Lemma 3.13]{BLZ21}, $\tau$ and $\theta_1$ are klt Mori fibre spaces. In particular, the general fibre of $\tau$ and of $\theta_1$ is a rational curve. This contradicts $g(C)\geq1$.
\end{proof}

\section{Mori Del Pezzo fibrations of degree $8$}\label{s:MdP8}

\begin{remark}\label{rmk: dP8 generic fibre}
Let $X\to C$ be a Mori Del Pezzo fibration of degree $8$. 
The generic fibre $X_{\k(C)}$ is a smooth Del Pezzo surface of degree $8$ and Picard rank one (\autoref{rmk: MdP}\autoref{MdP:0}) and hence is isomorphic to a smooth quadric in $\p^3_{\k(C)}$ of Picard rank one.
By \autoref{rmk: MdP}\autoref{MdP:0}, it has a $\k(C)$-rational point.
By \cite[Lemma 4.4.1]{BFT22}, up to isomorphism of $\p^3_{\k(C)}$, there exists $\mu\in\k(C)$ such that 
\[
X_{\k(C)}\simeq Q:=(x_0^2-x_1x_2-\mu x_3^2=0)\subset\p^3_{\k(C)}.
\]
Since $Q$ is smooth, $\mu\neq0$, and since $\rho(Q)=1$, $\mu$ is not a square in $\k(C)$ \cite[Lemma 4.4.1(2)]{BFT22}.
\end{remark}

\begin{lemma}\label{lem:dP8 Aut on C}
Let $X\to C$ be a Mori Del Pezzo fibration of degree $8$ over a curve $C$ of genus $g(C)\geq1$. Then $\Autz(X)$ acts trivially on $C$. 
\end{lemma}
\begin{proof} 
By Remark~\ref{rmk: dP8 generic fibre}, there exists $\mu\in\k(C)^*$ such that $X_{\k(C)}\simeq Q:=(x_0^2-x_1x_2-\mu x_3^2=0)$ and $\mu$ is not a square in $\k(C)^*$. Since $\k$ is algebraically closed, it follows that $\mu$ is non-constant.
Let $D\subset C$ be the finite non-empty set of points defined by $\mu=0$. 
Then the fibres of $\pi$ above $D$ are singular and hence are preserved by $\Autz(X)$. Since $g(C)\geq1$, it follows that $\Autz(X)$ acts trivially on $C$.
\end{proof}

\begin{remark}[{\cite[Lemmas 4.4.4(1) and 4.4.5]{BFT22}}]\label{rmk:Mfs8 map} 
Let $X\to C$ be a Mori Del Pezzo fibration of degree $8$ over a curve $C$ of genus $g(C)\geq0$. The following is shown in \cite[Lemma 4.4.5]{BFT22} and its proof, which works over any smooth projective curve (and not just $C=\p^1$). 

By Remark~\ref{rmk: dP8 generic fibre}, there exists $\mu\in\k(C)^*$ such that $X_{\k(C)}\simeq (x_0^2-x_1x_2-\mu x_3^2=0)$. 
Let $B$ be a smooth compactification of the curve $(y^2=\mu)\subset \A^1\times C$. The second projection $pr_C\colon \A^1\times C\to C$ induces a double cover $\theta\colon B\to C$. Recall that $\k(B)=\k(C)[y]$. 
Then $X_{\k(B)}\simeq (x_0^2-x_1x_2-y^2x_3^2)$ and
we have an isomorphism 
\[
\begin{array}{ccccc}
\chi\colon X_{\k(B)}  & \stackrel{\simeq}\to & \p^1_{\k(B)}\times\p^1_{\k(B)} & &\\
\quad [x_0:\dots:x_3] & \longmapsto & ([x_0+yx_3:x_2],[x_0-yx_3:x_2])&=&([x_1:x_0-yx_3],[x_1:x_0+yx_3]).
\end{array}
\]
Let $\sigma\in\Aut(\k(B))$ be the involution induced by the double cover $\theta$. 
It acts on $\p^1_{\k(B)}\times\p^1_{\k(B)}$ by exchanging the fibrations. 
For $A=(a_{ij})_{ij}\in\PGL_2(\k(B))$, we  write $A^{\sigma}=(\sigma(a_{ij}))_{ij}$. 
Then $\chi$ induces an isomorphism 
\[
\Autz(X_{\k(C)})=\Autz(X_{\k(B)})^{\sigma}\simeq \Autz(\p^1_{\k(B)}\times\p^1_{\k(B)})^{\sigma}\simeq\{(A,A^\sigma)\mid A\in \PGL_2(\k(B)) \}
\]
By \autoref{rmk: MdP}\autoref{MdP:2}, we have an injective group homomorphism
\begin{equation}\tag{iso}\label{eq:iso}
\Autz(X)_C\hookrightarrow\Autz(X_{\k(C)})\simeq\{(A,A^\sigma)\mid A\in \PGL_2(\k(B)) \}
\end{equation}
The morphism $\eta\colon X\times_CB\to B$ induces an isomorphism 
 \[
 (X\times_CB)_{\k(B)}\simeq X_{\k(B)}\stackrel{\chi}\simeq(\p^1_{\k(B)}\times\p^1_{\k(B)}).
 \]
Then $\eta$ and $\chi$ induce a birational map $\psi\colon X\times_CB\rat \p^1\times\p^1\times B$.
\end{remark}

\begin{lemma}\label{lem:Mfs8 aut}
Let $X\to C$ be a Mori Del Pezzo fibration of degree $8$ over a curve $C$ of genus $g(C)\geq0$. 
Let $\psi\colon X\times_CB\rat \p^1\times\p^1\times B$ be the birational map from Remark~\ref{rmk:Mfs8 map}. Then either $\psi$ induces an embedding of $\Autz(X)=\Autz(X)_C$ into the group
\[
\left\{\left(\left(\begin{matrix} a&P\\ 0&1\end{matrix}\right),\left(\begin{matrix} a&\sigma^*P\\ 0&1\end{matrix}\right)\right)\in\PGL_2(\k(B))\times\PGL_2(\k(B))\mid a\in\k^*, P\in  \k(B)\right\}
\]
 or $\psi$ induces an isomorphism 
 \[
 \Autz(X)=\Autz(X)_C\stackrel{\autoref{eq:iso}}\simeq\{(A,A^{\sigma})\in\PGL_2(\k(B))\times\PGL_2(\k(B))\mid A\in\PGL_2(\k)\}\simeq\PGL_2(\k)
 \] 
  and $\Autz(X)$ acts diagonally on a general fibre of $\pi$. 
\end{lemma}
\begin{proof}
By Remark~\ref{rmk: dP8 generic fibre}, there exists $\mu\in\k(C)^*$ such that $X_{\k(C)}\simeq Q:=(x_0^2-x_1x_2-\mu x_3^2=0)$. 
\autoref{lem:dP8 Aut on C} implies that $\Autz(X)=\Autz(X)_C$. 
The remaining claim follows from \cite[Proof of Theorem D]{BFT22} and Remark~\ref{rmk:Mfs8 map}.
\end{proof}

\begin{lemma}\label{lem:Mfs8 restriction}
Let $X\to C$ be a Mori Del Pezzo fibration of degree $8$ over a curve $C$ of genus $g(C)\geq0$.
Then either $\Autz(X)=\Autz(X)_C\simeq\PGL_2(\k)$ acts diagonally on any general fibre, or there is a Mori Del Pezzo fibration $Y\to C$ of degree $9$ and an $\Autz(X)$-equivariant birational morphism $X\rat Y$ over $C$. 
\end{lemma}
\begin{proof}
According to \autoref{lem:Mfs8 aut}, if $\Autz(X)$ is not isomorphic to $\PGL_2(\k)$, then it fixes a section. We blow up the corresponding point on $X_{\k(C)}$, and thusly obtain a birational map $X_{\k(C)}\rat \p^2_{\k(C)}$. It induces the desired birational map, which is $\Autz(X)=\Autz(X)_C$-equivariant by construction. 
\end{proof}

\begin{lemma}\label{lem:Mfs8 Aut=PGL2}
Let $\pi\colon X\to C$ be a Mori Del Pezzo fibration of degree $8$ over a curve $C$ of genus $g(C)\geq0$.
Suppose that $\Autz(X)=\Autz(X)_C\simeq\PGL_2(\k)$ acts diagonally on a general fibre of $\pi$ and let $\varphi\colon X\rat Y$ be an $\Autz(X)$-equivariant birational map to a Mori fibre space $\hat\pi\colon Y\to B$. 
Then the following hold:
\begin{enumerate}
\item There is an isomorphism $\iota\colon C\to B$ such that $\iota\circ\pi=\hat\pi\circ\varphi$ and $\varphi$ induces an isomorphism $X_{\k(C)}\simeq Y_{\k(C)}$, and the base-locus of $\varphi$ lies in the singular fibres of $\pi$.
\item $\Autz(X)$ is a maximal connected algebraic subgroup in $\Bir(X)$.
\end{enumerate}
\end{lemma}
\begin{proof}
Let $X_c\subset X$ be a smooth fibre of $\pi$. We have $X_c\simeq\p^1_\k\times\p^1_\k$ and $\Autz(X)=\Autz(X)_C$, which acts diagonally on $X_c$. Let $(x,y),(w,z)\in X_c$ be any two points. Since $\PGL_2(\k)$ acts $2$-transitively on $\p^1$, there exists $A\in\PGL_2(\k)$ such that $Ax=w$ and $Ay=z$. Thus, $(x,y)$ and $(w,z)$ belong to the same $\Autz(X)$-orbit. It follows that the base-locus of $\varphi$ is contained in the singular fibres of $\pi$ and that it induces an isomorphism $\hat\varphi\colon X_{\k(C)}\to Y_{\k(C)}$. 
In particular, $\hat\varphi\Autz(X)_C\hat\varphi^{-1}$ acts again diagonally on $Y_{\k(C)}$. 
By \autoref{lem:Mfs8 aut}, this induces an injective morphism of connected algebraic groups $\PGL_2(\k)\simeq\Autz(X)\hookrightarrow\Autz(Y)\simeq\PGL_2(\k)$. It is an isomorphism by dimension.
\end{proof}

\section{When the vertical automorphism group is finite}\label{s:finite}

Recall from \autoref{rmk: MdP}\autoref{MdP:2} that there is an injective group homomorphism $\Autz(X)_C\hookrightarrow\Aut(X_{\k(C)})$. In particular, if $X\to C$ is a Mori Del Pezzo fibration of degree $\leq5$, then $\Autz(X)_C$ is finite by \autoref{rmk: MdP}\autoref{MdP:0} and \autoref{lem:auto of dp surface}\autoref{auto dp surface:1}.

\subsection{Description of the automorphism group}

\begin{lemma}\label{lem: autm acting trivially on C}
Let $\pi\colon X\to C$ a Mori Del Pezzo fibration above a curve $C$ such that $\Autz(X)_C$ is finite. If $\Autz(X)$ acts trivially on $C$, then $\Autz(X)$ is trivial. 
\end{lemma}
\begin{proof}
By hypothesis, the group $\Autz(X)=\Autz(X)_C$ is finite, so $\Autz(X)$ is trivial. 
\end{proof}

\begin{lemma}\label{lem:dp<5-autom}
Let $C$ be a curve of genus $g(C)\geq1$ and $\pi\colon X\to C$ a morphism such that $\pi_*\Ol_X = \Ol_C$, $X_{\k(C)}$ is a Del Pezzo surface and $\Autz(X)_C$ is finite. 
If $\Autz(X)$ is non-trivial, then $g(C)=1$ and $\Autz(X)$ is an elliptic curve and there is an exact sequence
\begin{equation}\label{eq:seq-dp<6}
1\to \Autz(X)_C\to \Autz(X)\to \Autz(C)\to 1
\end{equation}
which is an unramified cover of the elliptic curve $\Autz(C)$. 
\end{lemma}
\begin{proof}
If $\Autz(X)$ acts trivially on $C$, then $\Autz(X)$ is trivial by \autoref{lem: autm acting trivially on C}. So, $\Autz(X)$ acts non-trivially on $C$. Since $C$ is an elliptic curve, this means that the induces homomorphism $\Autz(X)\to\Autz(C)$ is surjective. 
Since $\Autz(X)_C$ is finite, it follows that $\Autz(X)$ is a finite extension of an elliptic curve. In particular, $\Autz(X)$ is an elliptic curve. 
\end{proof}

Recall the following definition: 
let $G$ be an algebraic group, $H\subset G$ a closed subgroup and $S$ an $H$-variety. 
We define $G\times^{H}S:=(G\times S)/H$, where $H$ acts on $G\times S$ by $h\cdot(g,s)=(gh^{-1},hs)$. In literature, the notation $G\times_{H}S$ is used as well. 
The group $G$ acts on $G\times S$ by left-multiplication on the first factor and trivially on $S$, which induces a $G$-action on $G\times^{H}S$. We denote by $[g,s]$ the class of $(g,s)\in G\times S$ in $G\times^H S$. 

Since $G\times^{H}S$ is a principal bundle \cite[Lemme 6.1]{Res04}, it is smooth.

\begin{theorem}\label{thm:almost product}
Let $C$ be a curve of genus $g(C)=1$ and $\pi\colon X\to C$ with $\pi_*\Ol_X = \Ol_C$ and $X_{\k(C)}$ a Del Pezzo surface. Suppose that $H:=\Autz(X)_C$ is finite and the action of $G:=\Autz(X)$ on $C$ is non-trivial. 
Let $\iota\colon G/H\to C$ be the isomorphism of elliptic curves induced by the exact sequence \eqref{eq:seq-dp<6} and set $S:=\pi^{-1}(0)$. 
Then $S\subset X$ is an $\Autz(X)_C$-invariant surface, 
\[
\psi\colon G\times^{H} S\to X,\quad [g,s]\mapsto gs
\] 
is a $G$-equivariant isomorphism and $\pi\circ\psi=\iota\circ pr_{G/H}$.
\end{theorem}

\begin{proof}
By \autoref{lem:dp<5-autom}, $G$ is an elliptic curve, $H:=\Autz(X)_C$ is finite and there is an isomorphism $\iota\colon G/H\simeq \Autz(C)\simeq C$. 
So, we have a morphism $\phi:=\iota^{-1}\circ\pi\colon X\to  G/H$. 
Set $S:=\pi^{-1}(\iota(H/H))=\pi^{-1}(0)$. Since $G$ acts transivitely on $C$ by hypothesis and since the general fibre of $\pi$ is reduced, $S$ is a reduced subvariety of $X$.
By \cite[Lemma 6.1]{Res04}, the morphism $G\times S\to X$, $(g,s)\mapsto gs$ induces an isomorphism $\psi\colon G\times^{H}S\to X$, and $\psi$ is $G$-equivariant. 
Since $H$ is finite, we have $\dim S=2$. 
\end{proof}

\subsection{The case of Mori Del Pezzo fibrations of degree $1,2,3,5$}

We now determine the conjugacy classes of $\Autz(X)$ in $\Bir(X)$ when $X\to C$ is a Mori Del Pezzo fibration of degree $1,2,3$ or $5$.

\begin{theorem}\label{thm:K<3}
Let $\pi\colon X\to C$ be a Mori Del Pezzo fibration of degree $d\in\{1,2,3\}$ above a curve $C$ of genus $g(C)\geq1$. 
If $\Autz(X)$ acts non-trivially on $C$, then $g(C)=1$ and $\Autz(X)$ is an elliptic curve, it is a maximal connected algebraic subgroup of $\Bir(X)$ and any $\Autz(X)$-equivariant birational map $X\rat Y$
 to another Mori fibre space $\pi'\colon Y\to B$ satisfies that $\pi'$ is a Mori Del Pezzo fibration of degree $d$.
\end{theorem}
\begin{proof}
Suppose that $\Autz(X)$ acts non-trivially on $C$. Then $g(C)=1$ and by \autoref{thm:almost product}, $\Autz(X)$ is an elliptic curve.
Let $\varphi\colon X\rat Y$ be an $\Autz(X)$-equivariant birational map to a Mori fibre space $\pi'\colon Y\to B$. Then, by \autoref{thm:Enrica}, $\varphi$ is a composition $\varphi_n\circ\cdots\circ\varphi_1$ of $\Autz(X)$-equivariant Sarkisov links and isomorphisms $\varphi_1,\dots,\varphi_n$.
By \autoref{lem:links I}, \autoref{lem:preserves fibration} and \autoref{lem: links III and IV}, $\varphi_1\colon X\rat X_1$ is a link of type II over $C$ to a Mori Del Pezzo fibration $\pi_1\colon X_1\to C$. Since $\Autz(X)$ acts transitively on $C$, the base-locus of $\varphi_1$ and $\varphi_1^{-1}$ are curves that are transversal to a general fibre of $\pi$ and $\pi_1$, respectively.
By \autoref{lem:type II}, $\varphi_1$ induces a link of type II $\hat\varphi_1\colon X_{\k(C)}\rat (X_1)_{\k(C)}$. 
\cite[Theorem 2.6(ii)]{iskovskikh_1996} implies that $K_{(X_1)_{\k(C)}}^2=d$.  
Continuing this way, it follows that $\pi'$ is a Mori Del Pezzo fibration of degree $d$ and there is an isomorphism $\alpha\colon C\to B$ such that $\alpha\circ\pi=\pi'\circ\varphi$. 
Moreover, $\Autz(X)$ and $\Autz(Y)$ are elliptic curves by \autoref{lem: autm acting trivially on C} and \autoref{lem:dp<5-autom}, so $\varphi\Autz(X)\varphi^{-1}=\Autz(X')$.
\end{proof}

Recall that if $S_5$ is a Del Pezzo surface of degree $5$ over an arbitrary perfect field, then there is an injective group homomorphism $\Aut(S_5)\hookrightarrow\mathrm{Sym}_5$.

\begin{proposition}\label{pro:DP5}
Let $C$ be a smooth curve of $g(C)\geq1$ and $\pi\colon X\to C$ be a Mori Del Pezzo fibration of degree $5$. Suppose that $\Autz(X)$ is non-trivial.
Then $g(C)=1$, $\Autz(X)_C\simeq\Z/5$ and $\Autz(X)$ is an elliptic curve.
 \end{proposition}

\begin{proof}
Let $G:=\Autz(X)$ and $H:=\Autz(X)_C$. Since $G$ is non-trivial, it acts transitively on $C$ by \autoref{lem: autm acting trivially on C}. In particular, $g(C)=1$.
By \autoref{rmk: MdP}\autoref{MdP:2}, there is an injective homomorphism of groups $H \hookrightarrow\Aut(X_{\k(C)})\subset\mathrm{Sym}_5$, so  $H$ is finite. 
It follows from \autoref{lem:dp<5-autom} that $G$ is an elliptic curve. 
By \autoref{thm:almost product}, we have $X\simeq G\times^H F$, where $F$ is any fibre of $\pi$ (since $G$ acts transitively on $C$).   
Moreover, $\pi\colon X\to G/H$ is the projection onto $G/H\simeq C$. 

The surface $F$ is a Del Pezzo surface of degree $5$ by hypothesis. 
We have a morphism of algebraic groups $H\to\Aut(F)\simeq\mathrm{Sym}_5$.
Let $H'$ be its image. Since $G$ is abelian, also $H$ and hence $H'$ are abelian. Thus $\vert H'\vert\leq6$. 
In \cite[Figure 2]{Boi23} we find the list of actions of any subgroup of $\rm{Sym}_5$ on the set of $(-1)$-curves on $F$, up to conjugacy in $\rm{Sym}_5$.
Since $\rho(X/C)=1$ and $\Autz(X)$ acts transitively on $C$ and since $X\simeq G\times^HF$, we have $\NS(F)^H=\Z$. From \cite[Figure 2]{Boi23} we deduce that $H'\simeq\Z/5$. Since $X\simeq G\times^HF$, we have $H'\simeq H$.
\end{proof}

\subsection{Mori Del Pezzo fibrations of degree $4$}
This case is a little more involved than the case of Mori Del Pezzo fibrations of degree $d\in\{1,2,3,5\}$.

\begin{lemma}\label{lem:base is ruled}
Let $C$ be a smooth projective curve of genus $g(C)=1$.
Let $\pi\colon X\to B$ be a Mori fibre space with $\dim X=3=\dim B+1$ and suppose there is a surjective morphism $\tau\colon B\to C$ such that $\tau_*\Ol_B = \Ol_C$. Suppose moreover that either a general fibre of $\tau$ is isomorphic to $\p^1$ or that the general fibre of $\tau\circ\pi$ is a smooth rational surface.
If $\Autz(X)$ acts non-trivially on $C$, then $B$ is a ruled surface over $C$. 
\end{lemma}
\begin{proof}
By Blanchard's lemma (\autoref{lem:blanchard}), there is an $\Autz(X)$-action on $C$ making $\tau\circ\pi$ equivariant.
The surface $B$ is klt as base of a Mori fibre space \cite[Corollary 4.6]{Fuj99}. In particular, it has isolated singularities \cite[Proposition 9.3]{GKKP11}. The image by $\tau$ of these singularities is fixed by $\Autz(X)$. Since $\Autz(X)$ acts non-trivially on $C$, it acts transitively, so $B$ is smooth and moreover all fibres of $\tau$ are isomorphic. 

If a general fibre of $\tau$ is isomorphic to $\p^1$, then the transitivity of the $\Autz(X)$-action on $C$ implies that all fibres of $\tau$ are isomorphic to $\p^1$ and the claim follows from \cite[V.2 Definition]{MR0463157} and \autoref{thm:Tsen}.
Suppose that a general fibre $F$ of $\tau\circ\pi$ is a rational surface. Since $F$ is rational, the fibre $\pi(F)$ of $\tau$ is rational. 
Since $F$ is smooth and the general fibre of $\pi$ is isomorphic to $\p^1$, it follows that $\pi(F)$ is smooth. Thus $\pi(F)\simeq\p^1$. We conclude as above.
\end{proof}

\begin{lemma}\label{lem: DP4}
Let $C$ be an elliptic curve and let $\pi\colon X\to C$ be a Mori Del Pezzo fibration of degree $4$ and $\psi\colon X\rat Y$ an $\Autz(X)$-equivariant birational map to a Mori fibre space $\pi'\colon Y\to B$. Suppose that $\Autz(X)$ is non-trivial. 
Then one the following holds:
        \begin{enumerate}
        \item\label{proDP4:1} $B\simeq C$ and $\pi'$ is a Mori Del Pezzo fibration of degree $K_{Y_{\k(C)}}^2=4$,
        \item\label{proDP4:2} there is a ruled structure $\tau\colon B\to C$, the generic fibre $Y_{\k(C)}$ is a Del Pezzo surface of degree $3$ with $\rho(Y_{\k(C)})=2$.
        \end{enumerate}
\end{lemma}

\begin{proof}
By \autoref{thm:Enrica}, the $\Autz(X)$-equivariant birational map $\psi\colon X\rat Y$ is the composition $\psi=\varphi_n\circ\cdots\circ\varphi_1$ of Sarkisov links and isomorphisms $\varphi_{i+1}\colon X_i\rat X_{i+1}$, where $X=X_0$, $\pi=\pi_0$, $Y=X_n$, $\pi_n=\pi'$ and $\pi_i\colon X_i\to B_i$ is a Mori fibre space for $i=0,\dots,n$. We can assume that $\varphi_1$ is not an isomorphism and that $\varphi_{i+1}\varphi_i$ is not an isomorphism for all $1\leq i\leq n-1$.
    By \autoref{lem:dp<5-autom}, 
\begin{equation}\tag{$\diamondsuit$}\label{eq:transitive}
\text{$\Autz(X)$ is an elliptic curve acting transitively on $C$.}
\end{equation}

By \autoref{lem: links III and IV}, $\varphi_1$ is of type I or II. 
If $\varphi_i\circ\cdots\circ\varphi_1$ is a composition of links of type II and isomorphisms, then $\pi_{i}\colon X_{i}\to B_{i}$ is a Mori Del Pezzo fibration satisfying \autoref{proDP4:1} by \autoref{lem:type II} and \cite[Theorem 2.6(ii)($K_X^2=4$)]{iskovskikh_1996}. If $i=n$, we are finished. If $i\neq n$, we continue the argument with the composition $\varphi_n\circ\cdots\circ\varphi_{i+1}$. 
We can therefore assume that $\varphi_1$ is of type I. 
By \autoref{lem:links I}, a link $\varphi_1\colon X=X_0\rat X_1$ of type I is a birational map satisfying a commutative diagram
\[
\begin{tikzcd}
Y'\ar[d,"div_1",swap]\ar[rr,dotted] && X_1\ar[d,"\pi_1"]\\
X=X_0\ar[dr,"\pi=\pi_0"]\ar[urr,"\varphi_1",dashed]&&B_1\ar[dl,"\theta_1",swap]\\
&C&
\end{tikzcd}
\]
where the base-locus of $div_1^{-1}$ is an $\Autz(X)$-invariant section of $\pi_0$, $(X_1)_{\k(C)}\simeq Y'_{\k(C)}$ is a Del Pezzo surface of degree $3$, $\rho(X_1/C)=\rho((X_1)_{\k(C)})=2$ and $B_1$ is a surface. Then the general fibre of $\theta_1\circ\pi_1$ is a Del Pezzo surface, so 
\autoref{lem:base is ruled} implies that $\theta_1\colon B_1\to C$ is a ruled surface. It follows that $\pi_1\colon X_1\to B_1$ satisfies \autoref{proDP4:2}. 
If $\varphi_{i}\circ\cdots\circ\varphi_2$ is a (perhaps trivial) composition of links of type II and isomorphisms, then $\pi_{i}\colon X_{i}\to B_{i}$ satisfies \autoref{proDP4:2} by \autoref{lem:type II} and \cite[Theorem 2.6(ii)($K_X^2=3$)]{iskovskikh_1996}. If $i=n$, we are again finished. We can therefore suppose that $\varphi_{i+1}$ of type I, III or IV. 
Assume that $\varphi_{i+1}$ is of type I. 
\[
\begin{tikzcd}
Y'\ar[d,"div_1",swap]\ar[rr,dotted] && X_1\ar[d,"\pi_1"]\ar[r,dashed,"\varphi_2","II"'] &\cdots\ar[r,dashed,"\varphi_i","II"']& X_{i}\ar[r,dashed,"\varphi_{i+1}","I"']\ar[d,"\pi_i",swap] & X_{i+1}\ar[d,"\pi_{i+1}"]\\
X=X_0\ar[dr,"\pi=\pi_0"]\ar[urr,"\varphi_1",dashed,"I"']&&B_1\ar[dl,"\theta_1",swap]\ar[rr,"\simeq"]&&B_i&B_{i+1}\ar[l,"\theta_2",swap]\\
&C&&&&
\end{tikzcd}
\]
Since $\theta_1\colon B_1\to C$ is a ruled surface, $B_i$ is smooth. Since $B_{i+1}$ is a surface and $\rho(B_{i+1}/B_i)=1$, $\theta_2$ is the blow up of a point. Such a point is $\Autz(X)$-fixed, contradicting \autoref{eq:transitive}.

Assume that $\varphi_{i+1}$ is a link of type III. We will show that $\pi_{i+1}\colon X_{i+1}\to B_{i+1}$ satisfies \autoref{proDP4:1}. Then $\pi_{i+1}\colon X_{i+1}\to B_{i+1}$ is a Mori Del Pezzo fibration above the curve $B_{i+1}$ and by \autoref{lem:links I}, it is of degree $4$.
\[
\begin{tikzcd}
Y'\ar[d,"div_1",swap]\ar[rr,dotted] && X_1\ar[d,"\pi_1"]\ar[r,dashed,"\varphi_2","II"'] &\cdots\ar[r,dashed,"\varphi_i","II"']& X_{i}\ar[r,dashed,"\varphi_{i+1}","III"']\ar[d,"\pi_i",swap] & X_{i+1}\ar[d,"\pi_{i+1}"]\\
X=X_0\ar[dr,"\pi=\pi_0"]\ar[urr,"\varphi_1",dashed,"I"']&&B_1\ar[dl,"\theta_1",swap]\ar[rr,"\simeq"',"\alpha"]&&B_i\ar[r,"\theta_2"]&B_{i+1}\\
&C&&&&
\end{tikzcd}
\]
Suppose that $\theta_2\circ\alpha$ does not factor through $\theta_1$. 
Let $F$ be a general fibre of $\theta_1$. It is rational, because $\theta_1\circ\alpha^{-1}\colon B_i\to C$ is a ruled surface. Since $\theta_2\circ\alpha$ does not factor through $\theta_1$, the image $(\theta_2\circ\alpha)(F)$ is not a point, hence  $(\theta_2\circ\alpha)(F)=B_{i+1}$. So, $B_{i+1}$ is rational. Hence $X_{i+1}$ is rational, against our hypothesis that $g(C)=1$.
It follows that there is a morphism $\beta\colon C\to B_{i+1}$ such that $\theta_2\circ\alpha=\beta\circ\theta_1$. Since $\theta_2,\theta_1$ have connected fibres, $\beta$ is an isomorphism. \autoref{lem:links I} applied to $\varphi_{i+1}^{-1}$ implies that $\pi_{i+1}\colon X_{i+1}\to B_{i+1}$ satisfies \autoref{proDP4:1}.

Assume that $\varphi_{i+1}$ is a link of type IV. We will show that $\pi_{i+1}\colon X_{i+1}\to B_{i+1}$ satisfies \autoref{proDP4:2}. We have the following commutative diagram:
\[
\begin{tikzcd}
Y'\ar[d,"div_1",swap]\ar[rr,dotted] && X_1\ar[d,"\pi_1"]\ar[r,dashed,"\varphi_2","II"'] &\cdots\ar[r,dashed,"\varphi_i","II"']& X_{i}\ar[r,dotted,"\varphi_{i+1}","IV"']\ar[d,"\pi_i",swap] & X_{i+1}\ar[d,"\pi_{i+1}"]\\
X=X_0\ar[dr,"\pi=\pi_0"]\ar[urr,"\varphi_1",dashed,"I"']&&B_1\ar[dl,"\theta_1",swap]\ar[rr,"\alpha","\simeq"']&&B_i\ar[d,"\theta_2"']&B_{i+1}\ar[ld,"\theta_3"]\\
&C&&&D&
\end{tikzcd}
\]
\autoref{lem: links III and IV} implies that $B_{i+1}$ is a surface. 
Since $\theta_1\circ\alpha^{-1}\colon B_i\to C$ is a ruled surface, we have $\rho(B_i)=2$. Moreover, $\theta_2\colon B_i\to D$ is a klt Mori fibre space \cite[Lemma 3.13]{BLZ21}, so $D$ is a curve. 
Again, if $\theta_2\circ\alpha$ does not factor through $\theta_1$, we conclude that $X_{i+1}$ is rational, against the assumption $g(C)=1$. As above, it follows that there is an isomorphism $\beta\colon C\to D$ such that $\theta_2\circ\alpha=\beta\circ\theta_1$. In particular, $(X_{i+1})_{\k(C)}\simeq (X_{i})_{\k(C)}$ is a Del Pezzo surface of degree $3$. Then, by \autoref{lem:base is ruled}, $\theta_3$ is a ruled surface. It follows that $\pi_{i+1}\colon X_{i+1}\to B_{i+1}$ satisfies \autoref{proDP4:2}.
\end{proof}

\begin{proposition}\label{pro: DP4}
Let $\pi\colon X\to C$ be a Mori Del Pezzo fibration of degree $4$ above an elliptic curve $C$. Suppose that $\Autz(X)$ is nontrivial. Then $\Autz(X)$ is an elliptic curve and a maximal connected algebraic subgroup of $\Bir(X)$.
\end{proposition}

\begin{proof}
Let $\varphi\colon X\rat Y$ be an $\Autz(X)$-equivariant birational map to a Mori fibre space $\pi'\colon Y\to B$. By \autoref{lem: DP4}, either $\pi'$ is a Mori Del Pezzo fibration of degree $4$ and $B\simeq C$, or $B$ is a surface with a ruled structure $B\to C$ and $Y_{\k(C)}$ is a Del Pezzo surface of degree three with $\rho(Y_{\k(C)})=2$. 
Recall from \autoref{rmk: MdP}\autoref{MdP:2} that there are injective homomorphism of groups $\Autz(X)_C\hookrightarrow\Aut(X_{\k(C)})$ and $\Autz(Y)_C\hookrightarrow\Aut(Y_{\k(C)})$. 
Since $\Autz(X_{\k(C)})$ and $\Autz(X_{\k(C)})$ are finite by \autoref{lem:auto of dp surface}, it follows that $\Autz(X)_C$ and $\Autz(Y)_C$ are both finite. 
Then \autoref{lem:dp<5-autom} implies that $\Autz(X)$ and $\Autz(Y)$ are elliptic curves. 
It follows that $\varphi\Autz(X)\varphi^{-1}=\Autz(Y)$. 
\end{proof}


\section{Mori Del Pezzo fibrations of degree $6$}\label{s:MdP6}

Let $\pi\colon X\to C$ be a Mori Del Pezzo fibration of degree $6$. The case when $\Autz(X)_C$ is finite, is treated in Sections~\ref{s:finite}.

\begin{lemma}[see {\cite[Proposition 4.2.3]{BFT22} and \cite[Proposition 3.5]{Ron25}}]\label{lem:MdP 6 1}
Let $\pi\colon X\to C$ be a Mori Del Pezzo fibration of degree $6$ above a curve $C$ of genus $g(C)\geq0$.
Then the neutral component of $\Autz(X)_C$ is a torus. 
\end{lemma}
\begin{proof}
Let $G$ be the connected component of $\Autz(X)_C$. If it is trivial, we are done. Suppose that $G$ is non-trivial. 
Chevalley's structure theorem \cite[Theorem 1.1.1]{BSU13} states that there is a short exact sequence of connected algebraic groups
\[
1\to G_{aff}\to G\to A\to 1
\]
where $G_{aff}$ is the largest connected affine normal algebraic subgroup of $G$ and $A$ is an abelian variety. 
Let $F$ be a general fibre of $\pi$.
There is a morphism of algebraic groups $\varphi_F\colon G\to\Autz(F)\simeq\G_{m,\k}^2$. 
Consider the restriction $\varphi_F|_{G_{aff}} \colon G_{aff}\to\G_m^2$. Notice that $G$ and hence $G_{aff}$ act faithfully on $X$. 
Since there is no non-trivial morphism of algebraic groups $\G_a\to \G_m$ and since $F$ is general, it follows that $G_{aff}$ contains no subgroup isomorphic to $\G_a$. It follows that $G_{aff}$ is a torus. 

If $G_{aff}$ is trivial, then $G=A$ is abelian. Then $\varphi_F$ is trivial for all fibres $F$ of $\pi$. Hence $G$ is trivial, against our assumption. 

So, $G_{aff}$ is non-trivial. Let us show that $A$ is trivial.
 Consider the injection of abstract groups $G\hookrightarrow\Autz(X_{\k(S)})\simeq\G_{m,\k(S)}^2$ (see \cite[Remark 2.1.6]{BFT22} for connected groups). Let $K$ be an algebraic closure of $\k(S)$. We obtain an injective group homomorphism  $\iota\colon G(\k)\to (K^*)^2$. It induces injective group homomorphism
 \[
 A(\k)\simeq G(\k)/G_{aff}(\k)\hookrightarrow (K^*)^2/\iota(G_{aff}(\k)).
 \]
The $n$-torsion group of $A(\k)$ is of order $n^{2d}$, where $d=\dim A\geq 0$, and the $n$-torsion group of $(K^*)^2/\iota(G_{aff}(\k))$ is of order $1$ or $n$. Therefore, $A$ is trivial and so $G=G_{aff}$. 
\end{proof}

\begin{lemma}\label{lem:T finite}
Let $L$ be a field of characteristic zero, $O\subset\p^2$ a rational curve of degree $\deg(O)\geq3$ and $q_1,q_2,q_3\in\p^2$ three non-collinear points.
Suppose $O$ has equal multiplicity in all three points $q_1,q_2,q_3$.
Then the subgroup $G$ of $\Aut(\p^2)$ preserving the set $O\cup\{q_1,q_2,q_3\}$ is finite. 
\end{lemma}
\begin{proof}
First, let $H\subset\Aut(\p^2)$ be the connected component of the group of automorphisms of $\p^2$ preserving the curve $O$. 
Consider the exact sequence
\[
1\to A:=\{\varphi\in H\mid \varphi|_O=\rm{id}_O\}\to H\stackrel{res}\to\Autz(O).
\]
where $res\colon H\to\Autz(O)$, $h\mapsto h|_O$, is the restriction.
Since $O$ is rational and $\deg(O)\geq3$, it is singular. 
Any non-trivial element of $\Aut(\p^2)$ fixes pointwise only lines, hence $A$ is trivial, so that $H\simeq\Aut(O)$. 
Let $\eta\colon \widetilde O\to O$ be the normalization of $O$. There is an $H$-action on $\widetilde O$ that makes $\eta$ $H$-equivariant. Let $p$ be a singular point of $O$ and $\overline{L}$ the algebraic closure of $L$. Notice that $(\widetilde O)_{\overline L}\simeq\p^1_{\overline L}$. Since $\mathrm{char}(L)=0$, the fibre $(\eta^{-1}(p))_{\overline L}$ contains at least two points, and they are fixed by $H_{\overline L}$. Therefore, $H\simeq\Autz(O)$ is a torus.

Let $G^0$ be the connected component of $G$. 
Then 
\[
G^0=\{\varphi\in H\mid \varphi(q_i)=q_i,i=1,2,3\}\subset H\subset\Aut(\p^2),
\]
so $G^0$ is a subtorus of $H$.
If $q_1\in O$, then $G^0$ fixes three points on $O$ and is hence trivial. 
If $q_1\notin O$, then $G^0$ preserves the intersection of $O$ with the three lines passing through any two of $q_1,q_2,q_3$. There are three such intersection points, so again, it follows that $G^0$ is trivial. 
\end{proof}

\begin{lemma}\label{lem:MdP 6 2}
Let $\pi\colon X\to C$ be a Mori Del Pezzo fibration of degree $6$ above a curve $C$ of genus $g(C)\geq0$.
Then $\Autz(X)_C$ does not contain a $1$-dimensional torus. 
\end{lemma}
\begin{proof}
Suppose that $\Autz(X)_C$ contains a $1$-dimensional torus $T$.  
Recall from \autoref{rmk: MdP}\autoref{MdP:0} that $X_{\k(C)}$ has a rational point $p$. Let $L$ be an algebraic closure of $\k(C)$. Then $p_{L}$ is not on any of the six $(-1)$-curves of $X_{L}$.
By \autoref{rmk: MdP}\autoref{MdP:2}, there is an injective group homomorphism $T\subset\Autz(X)_C\hookrightarrow\Aut(X_{\k(C)})$. The image of $T$ in $\Aut(X_{\k(C)})$ is an abstract subgroup. 

Suppose that $p$ is $T$-invariant. 
The blow-up of $p$ induces a Sarkisov link of type II $\psi\colon X_{\k(C)}\rat Q$ to a Del Pezzo surface $Q$ of degree $8$ with $\rho(Q)=1$ \cite[Theorem 2.6]{iskovskikh_1996}, and $T':= \psi T\psi^{-1}\subset \Aut(Q)$.
It contracts an irreducible curve $B\subset X_{\k(C)}$ onto a point $p'\in Q$ of degree 3 in general position, which is $T'$-invariant (since $p$ and hence $B$ are $T$-invariant). 
We have $Q_L\simeq\p^1_L\times\p^1_L$ and $p'_L=\{p_1,p_2,p_3\}$, where the $p_1,p_2,p_3\in\p^1_L\times\p^1_L$ are $L$-rational points in general position.
Let $\Aut(\p^1_L\times\p^1_L,\{p_1,p_2,p_3\})$ be the subgroup of $\Aut(\p^1_L\times\p^1_L)$ preserving $\{p_1,p_2,p_3\}$ and $\Aut(Q,p')$ be the subgroup of $\Aut(Q)$ preserving $p'$. 
The group $\Aut(\p^1_L\times\p^1_L,\{p_1,p_2,p_3\})$ is finite, and $\Aut(Q,p')_L\simeq \Aut(\p^1_L\times\p^1_L,\{p_1,p_2,p_3\})$ (see for instance \cite[Proposition 2.20]{TZ24}). So, $\Aut(Q,p')$ is finite. Since $T'$ is an abstract subgroup of $\Aut(Q,p')$, it follows that $T'$ is finite. 
Since $T\simeq T'$ as abstract groups, this contradicts that $\dim T=1$. 

Therefore, $p$ is not $T$-invariant, i.e. the rational section $\ell\subset X$ corresponding to $p$ is not $T$-invariant. 
Then the closure of the $T$-orbit of $\ell$ in $X$ is an irreducible surface $S$, which corresponds to an irreducible curve $O$ in $X_{L}$. Notice that $S$ contains an open subset of $T\times\ell$ and hence $O$ is an $L$-rational curve. 
Furthermore, notice that $T\subset\Aut(X_L)$ is an abstract group (without any algebraic structure over $L$) and that $T\cdot O=O$.

Pick any three disjoint $(-1)$-curves of $X_L$ and let $\eta\colon X_L\to\p^2_L$ be their contraction onto points $q_1,q_2,q_3\in\p^2_L$. It sends the remaining three $(-1)$-curves onto the three lines $L_1,L_2,L_3$ through any two of the $q_1,q_2,q_3$. 
The abstract subgroup $T\subset\Aut(X_{\k(C)})$ preserves each $(-1)$-curve on $X_L$. Indeed, $T$ acts regularly on any general fibre $F$ of $\pi$ and hence preserves each $(-1)$-curve of $F$, so $T$ preserves each $(-1)$-curve of $X_L$. 
It follows that $T':=\eta \circ T\circ \eta^{-1}$ is contained in $\Aut(\p^2_L)$. 

Consider image $O':=\eta(O)$ of $O$ in $\p^2_L$, which satisfies $T'\cdot O'=O'$. 
If $\deg(O')\geq3$, then $T'$ is finite by \autoref{lem:T finite}.
Since $T\simeq T'$ as abstract groups, this contradicts that $\dim T=1$. It follows that $\deg(O')\leq2$. Since $O'$ is irreducible, it is smooth. 
It intersects each $L_i$, so $O$ intersects three of the $(-1)$-curves of $X_L$. Since $\rho(X_{\k(C)})=1$, the $(-1)$-curves of $X_L$ make up a $\Gal(L/\k(C))$-orbit. It follows that $O$ intersects each of the $(-1)$-curves of $X_L$. 
This contradicts $O'$ being a line or smooth conic. We have reached a contradiction and so $\Autz(X)_C$ does not contain a $1$-dimensional torus.
\end{proof}

\begin{proposition}\label{prop:Mdf 6 finite}
Let $\pi\colon X\to C$ be a Mori Del Pezzo fibration of degree $6$ above a curve $C$ of genus $g(C)\geq1$. If $\Autz(X)$ is non-trivial, then $g(C)=1$, $\Autz(X)_C$ is finite, $\Autz(X)$ is an elliptic curve acting transitively on $C$. 
\end{proposition}

\begin{proof}
By \autoref{lem:MdP 6 1}, $\Autz(X)_C$ contains a torus and by \autoref{lem:MdP 6 2} it is zero-dimensional, so $\Autz(X)_C$ is finite.
\autoref{lem: autm acting trivially on C} and \autoref{lem:dp<5-autom} imply that $g(C)=1$ and that $\Autz(X)$ is an elliptic curve.
\end{proof}

\begin{corollary}\label{cor:MdP6}
Let $\pi\colon X\to C$ be a Mori Del Pezzo fibration of degree $6$ above a curve $C$ of genus $g(C)\geq1$ such that $\Autz(X)$ is non-trivial. Let $\psi\colon X\rat Y$ be an $\Autz(X)$-equivariant birational map to a Mori fibration $\hat\pi\colon Y\to B$. 
Then $\hat\pi$ is a Mori Del Pezzo fibration and there is an isomorphism $\hat\psi\colon C\to B$ such that $\hat\pi\circ\psi=\hat\psi\circ\pi$. 
\end{corollary}
\begin{proof}
By \cite[Theorem 1.2]{Flo20}, $\psi$ is a composition of $\Autz(X)$-equivariant Sarkisov links. 
By \autoref{prop:Mdf 6 finite}, $g(C)=1$ and $\Autz(X)$ acts transitively on $C$. 
By \autoref{lem: links III and IV} and \autoref{lem:links I} there are no $\Autz(X)$-equivariant links of type I, III and IV starting from $X$, respectively. 
Therefore, $\psi$ is a composition of links of type II. This yields the claim.
\end{proof}

\begin{lemma}\label{lem: dP6 type IIa}
Let $\pi\colon X\to C$ be a Mori Del Pezzo fibration of degree $6$ above a curve $C$ of genus $g(C)\geq1$ such that $\Autz(X)$ is non-trivial. Let $\psi\colon X\rat Y$ be an $\Autz(X)$-equivariant Sarkisov link to a Mori Del Pezzo fibration $\hat\pi\colon Y\to C$. 
Then $K_{Y_{\k(C)}}^2\in\{8,6\}$.  
\end{lemma}
\begin{proof}
By \autoref{prop:Mdf 6 finite}, we have $g(C)=1$ and $\Autz(X)$ acts transitively on $C$. Then the base-locus of $\psi$ is a curve that is transversal to the fibres of $\pi$. 
The claim now follows from \autoref{lem:links I}, \autoref{lem: links III and IV}, \autoref{lem:type II} and the list of possible links of type II starting from a Del Pezzo surface of degree $6$ over $\k(C)$ in \cite[(A.4) Theorem]{corti1995factoring} or \cite[Theorem 2.6 (ii) $K_X^2=6$]{iskovskikh_1996}.
\end{proof}

\begin{proposition}\label{lem: dP6 type IIb}
Let $\pi\colon X\to C$ be a Mori Del Pezzo fibration of degree $6$ above a curve $C$ of genus $g(C)\geq1$ such that $\Autz(X)$ is non-trivial. Let $\psi\colon X\rat Y$ be an $\Autz(X)$-equivariant birational map to a Mori Del Pezzo fibration $\hat\pi\colon Y\to B$. Then, there exist isomorphisms $\alpha\colon X\to Y$ and $\alpha'\colon C\to B$ such that $\pi'\circ\alpha=\alpha'\circ\pi$. 
\end{proposition}
\begin{proof}
By \autoref{prop:Mdf 6 finite}, we have $g(C)=1$ and $\Autz(X)$ acts transitively on $C$. Thus $\Autz(Y)$ acts transitively on $C$.
By \cite[Theorem 1.2]{Flo20}, $\psi$ is a composition of isomorphisms and $\Autz(X)$-equivariant Sarkisov links. 
By \autoref{cor:MdP6}, there is an isomorphism $\hat\psi\colon C\to B$ such that $\hat\pi\circ\psi=\hat\psi\circ\pi$. In particular, by the relative version of \cite[Theorem 1.2]{Flo20}, $\psi$ is a composition $\psi=\psi_n\circ\cdots\circ\psi_1$ isomorphisms and $\Autz(X)$-equivariant links of type II.
We can suppose that $\psi_1\colon X\rat X_1$ is not an isomorphism. 
\autoref{lem: dP6 type IIa} states that $K_{(X_1)_{\k(C)}}^2\in\{6,8\}$.
\autoref{lem:dP8 Aut on C} implies that if $K_{(X_1)_{\k(C)}}^2=8$, then $\Autz(Y)$ acts trivially on $C$, so we have $K_{(X_1)_{\k(C)}}^2=6$. 
Then the inclusion $\psi_1\Autz(X)\psi_1^{-1}\subseteq\Autz(X_1)$ is an equality because both sides are elliptic curves by \autoref{prop:Mdf 6 finite}.
Continuing inductively this way shows that $K_{Y_{\k(C)}}^2=6$, $\Autz(Y)_C\simeq\Autz(X)_C$ and $\Autz(Y)\simeq\Autz(X)$. 
Then \autoref{thm:almost product} implies that $X\simeq\Autz(X)\times^{\Autz(X)_C} S_c$ and $Y\simeq \Autz(X)\times^{\Autz(X)_C} S_c'$, where $S_c,S_c'$ are fibres of $\pi,\pi'$, respectively, above any point $c\in C$. By hypothesis, $S_c,S_c'$ are Del Pezzo surfaces over $\k$ of degree $6$, so $S_c\simeq S_c'$. This yields the claim. 
\end{proof}

\begin{corollary}\label{cor:DP6}
Let $\pi\colon X\to C$ be a Mori Del Pezzo fibration of degree $6$ above a curve $C$ of genus $g(C)\geq1$ such that $\Autz(X)$ is non-trivial. Then $X$ is birational to $C\times\p^2$ and $\Autz(X)$ is maximal among the connected algebraic subgroups of $\Bir(C\times\p^2)$.
\end{corollary}
\begin{proof}
\autoref{prop:MdP5} states that $X$ is birational to $C\times\p^2$ over $C$. The remaining claim follows from \autoref{lem: dP6 type IIb}. 
\end{proof}

\begin{proof}[Proof of \autoref{thm:K<5}]
Recall that if $d\leq4$, then $\Autz(X)_C$ is finite by \autoref{lem:auto of dp surface}. We show in \autoref{prop:Mdf 6 finite} that $\Autz(X)_C$ is finite if $d=6$.
The rest the claim for $d\in\{1,2,3\}$ is \autoref{thm:K<3} and \autoref{thm:almost product}. 
If $d=4$, the claim is \autoref{thm:almost product} and \autoref{pro: DP4}.
If $d=6$, the claim is the summary of \autoref{thm:almost product}, \autoref{prop:Mdf 6 finite} and \autoref{lem: dP6 type IIb}.
\end{proof}

\section{Mori Del Pezzo fibrations of degree $9$}\label{s:MdP9}

\subsection{From Mori Del Pezzo fibrations to $\p^2$-bundles}

Recall from \autoref{thm:Tsen} that the generic fibre of a  Mori Del Pezzo fibration $\pi\colon X\to C$ of degree $9$ over a curve $C$ is isomorphic to $\p^2_{\k(C)}$.

\begin{lemma}[see {\cite[Lemma 4.3.3]{BFT22}} for linear case]\label{lem:morph to finite}
Let $G$ be a connected algebraic group and $n\geq1$ an integer. Then any (abstract) group homomorphism $G\to \Z/n$ is trivial. 
\end{lemma}
\begin{proof}
Let $\nu\colon G\to \Z/n$ be an abstract homomorphism of groups.
By Chevalley theorem \cite{BSU13}, there is a short exact sequence
\[
1\to G_{aff}\to G\to A\to 1
\]
where $A$ is abelian and $G_{aff}$ is the largest connected affine normal algebraic subgroup of $
G$. 
The restriction $\nu|_{G_aff}$ is trivial by \cite[4.3.3]{BFT22}. So, $\nu$ factors through an abstract group homomorphism $\overline{\nu}\colon A\simeq G/G_{aff}\to \Z/n$.
For any $a\in A$ there is some $a'\in A$ such that $na'=a$. Then $\overline{\nu}(a)=\overline{\nu}(na')=n\overline{\nu}(a')=0$. 
\end{proof}

The following statement extends \cite[Proposition 4.3.5]{BFT22} over any smooth curve.

\begin{proposition}\label{prop:Mdp to bundle}
Let $C$ be a smooth projective curve and $\pi \colon X \to C$ a morphism whose generic fibre is isomorphic to $\mathbb{P}^2_{\k(C)}$. 
Then there exists a regular action of $\mathrm{Aut}^{\circ}(X)$ on a $\mathbb{P}^2$-bundle $\tau \colon Y \to C$ and an $\Autz(X)$-equivariant birational map $X\rat Y$ over $C$.
\end{proposition}
\begin{proof}
If $C=\p^1$, it is \cite[Proposition 4.3.5]{BFT22}. We follow and adapt its proof for a smooth projective base curve $C$ of positive genus. There exists an open subset $U$ on which $\pi_{\vert U}$ is a $\p^2$-bundle. If $g(C)=1$ and $\Autz(X)$ acts transitively on $C$, then $\pi$ is already a $\p^2$-bundle, and the proof is over. Hence, we can assume in the rest of the proof that $\pi_* \colon \Autz(X) \to \Autz(C)$ is trivial (which is automatic if $g(C)\geq 2$).

Under this assumption, mutatis mutandis, the proof of \cite[Proposition 4.3.5]{BFT22} works over $C$. A key result used in their proof is \cite[Lemma 4.3.3]{BFT22}, applied to $G:=\Autz(X)$ which is connected and linear. In our case, we apply \autoref{lem:morph to finite}.
\end{proof}

In the previous proposition, we reduce the case of Mori Del Pezzo fibrations of degree $9$ to $\p^2$-bundles. In the next sections, we will use the following lemma to show that there is no $\Autz(X)$-equivariant birational map from $X$ to a Mori Del Pezzo fibre space of lower degree, if $X$ is $\p^2$-bundle such that $\Autz(X)$ is maximal.

\begin{lemma}\label{lem:no_map_to_lower_degree}
    Let $C$ be an elliptic curve and $\pi\colon X\to C$ a Mori Del Pezzo fibration such that $\Autz(X)$ acts transitively on $C$ and $\dim\Autz(X)_C>0$. Then $K_{X_{\k(C)}}^2=9$ and there is no $\Autz(X)$-equivariant birational map from $X$ to a Mori Del Pezzo fibre space of degree $\leq8$. 
\end{lemma} 

\begin{proof}
Let $d:=K_{X_{\k(C)}}^2$ be the degree of $\pi$.
Let $\pi'\colon X'\to D$ be a Mori Del Pezzo fibration of degree $e$ and $\varphi\colon X\rat X'$ an $\Autz(X)$-equivariant birational map. By \autoref{lem:preserves fibration}, there exists an isomorphism $\alpha\colon C\to D$ such that $\pi'\circ\varphi=\alpha\circ\pi$. Without loss of generality, we can assume that $\alpha$ is the identity map. Then $\varphi^{-1}\Autz(X)_C\varphi\subseteq\Autz(X')_C$ and hence $\dim\Autz(X')_C>0$. 
By \autoref{rmk: MdP}\autoref{MdP:2}, there is an injective group homomorphism $\Autz(X')_C\hookrightarrow\Autz(X_{\k(C)})$, so \autoref{lem:auto of dp surface}\autoref{auto dp surface:1} implies $d,e\geq6$. 
\autoref{prop:Mdf 6 finite} implies that $d,e\neq6$. 
\autoref{rmk: MdP}\eqref{MdP:1} implies that $d,e\neq7$. 
\autoref{lem:dP8 Aut on C} implies that $d,e\neq8$.
\end{proof}

\begin{lemma}\label{lem:section blow-up}
Let $C$ be a smooth projective curve of genus $g(C)\geq0$ and $\pi\colon X\to C$ a $\p^2$-bundle. Suppose there is an $\Autz(X)$-invariant section $\Sigma$ and let $div\colon \mathrm{Bl}_{\Sigma}(X)\to X$ be the blow-up of $X$ in $\Sigma$. Then $div^{-1}\colon X\rat \mathrm{Bl}_{\Sigma}(X)$ is an $\Autz(X)$-equivariant link of type I. 
\end{lemma}
\begin{proof}
We can choose coordinates of $\p^2$ such that $\Sigma=C\times\{(0:0:1)\}$. 
Let $(U_k)_k$ be a trivializing open cover of $C$. Then the blow-up $U_k\times\F_1\to U_k\times\p^2$, $(u,[y_0:y_1;z_0:z_1])\mapsto(u,[y_0z_0:y_1:y_0z_1])$ is the blow-up of $\Sigma$. Indeed, these are trivializing open sets of the blow-up $div\colon\mathrm{Bl}_{\Sigma}(X)\to X$ of $\Sigma$.
Moreover, the morphism $U_k\times\F_1\to U_k\times\p^1$, $(u,[y_0:y_1;z_0:z_1])\mapsto(u,[z_0:z_1])$ give trivializing open sets of a morphism $\pi'\colon\mathrm{Bl}_{\Sigma}(X)\to B$ to a ruled surface $\tau\colon B\to C$ such that $\tau\circ\pi'=\pi\circ div$. This yields a directed  Sarkisov diagram corresponding to a link of type I. All morphisms are $\Autz(X)$-equivariant, since $\Sigma$ is $\Autz(X)$-invariant.
\end{proof}

\subsection{Generalities on $\P^2$-bundles over elliptic curves}

\begin{definition}
    A vector bundle $\El$ is \emph{indecomposable} if it is not isomorphic to a direct sum $\El'\oplus \El''$, where $\El'$ and $\El''$ are non-zero vector bundles. Otherwise, $\El$ is \emph{decomposable}.
\end{definition}

Recall that the category of vector bundles is a Krull-Schmidt category, see \cite{Ati56}, i.e. every vector bundle decomposes as a finite direct sum of indecomposable vector bundles, and this decomposition is unique up to isomorphism and permutation of summands. 

\begin{remark}\label{rmk:indecomposable}
    \begin{enumerate}
        \item Recall that two projective bundles $\P(\El)$ and $\P(\El')$ are isomorphic if and only if there exists a line bundle $\Ll$ such that $\El \otimes \Ll \simeq \El'$. 
        \item A vector bundle $\El$ is indecomposable if and only if $\Ll \otimes \El$ is indecomposable for any line bundle $\mathcal L$. 
        \item If a vector bundle $\El$ has rank $3$, then $\El$ is decomposable if and only if $\P(\El)$ admits a section and a disjoint $\P^1$-subbundle.
    \end{enumerate}
\end{remark}

In virtue of the above remark, we say that $\p(\El)$ is \emph{indecomposable} (resp. \emph{decomposable}) if $\El$ is.

\medskip

By Atiyah's classification \cite{Ati57}, it follows that for every $r\geq 2$, there exist exactly $r$ indecomposable $\P^{r-1}$-bundles over an elliptic curve. These are the projectivization of indecomposable vector bundles $\El_{r,d}$, of rank $r$ and degree $d\in \{0,\ldots, r-1\}$, which are defined by successive non-trivial extensions. For self-containedness, we recall Atiyah's result for $r=2,3$:

\begin{theorem}\cite[Theorems 5, 6, 7 and 11]{Ati57} \label{thm:Atiyah}
	Let $C$ be an elliptic curve with a neutral point $p_0\in C$. The following hold:

    \begin{enumerate}
        \item Up to isomorphism, there exist exactly two indecomposable $\P^1$-bundles over $C$, denoted by $\mathcal{A}_{2,0}$ and $\mathcal{A}_{2,1}$. They are the projectivization of indecomposable rank-$2$ vector bundles which respectively fit into the following short exact sequences:
	\[
	\begin{array}{cccccccccc}
		0 & \to & \mathcal{O}_C & \to & \mathcal{E}_{2,0}  & \to & \mathcal{O}_C &\to& 0, \\
		0 & \to & \mathcal{O}_C & \to & \mathcal{E}_{2,1}  & \to & \mathcal{O}_C(p_0) &\to& 0.
	\end{array}
	\]
        \item Up to isomorphism, there exist exactly three indecomposable $\P^2$-bundles over $C$, denoted by $\Al_{3,0},\Al_{3,1}$ and $\Al_{3,2}$. They are the projectivization of indecomposable rank-$3$ vector bundles, which respectively fit into the following short exact sequences:
	\[
	\begin{array}{cccccccccc}
		0 & \to & \mathcal{O}_C & \to & \mathcal{E}_{3,0}  & \to & \mathcal{E}_{2,0} &\to& 0, \\
		0 & \to & \mathcal{O}_C & \to & \mathcal{E}_{3,1}  & \to & \mathcal{E}_{2,1} &\to& 0, \\
		0 & \to & \mathcal{O}_C \oplus \mathcal{O}_C & \to & \mathcal{E}_{3,2}  & \to & \mathcal{O}_C(2p_0) &\to& 0.
	\end{array}
	\]
    Moreover, the projective bundles $\Al_{3,1}$ and $\Al_{3,2}$ are dual to each other.
    \end{enumerate}
    Moreover, if $r\in \{2,3\}$, then $\El_{r,0}$ is the unique indecomposable vector bundle of rank $r$ and of degree $0$ admitting non-zero global sections, and $\Gamma(C,\El_{r,0})=\k$.
\end{theorem}

\begin{definition}\label{def:homogeneous}
    A vector bundle $\El$ over an elliptic curve $C$ is \emph{semi-homogeneous} if for every translation $t\in \Autz(C)$, there exists a line bundle $\Ll$ such that $t^*\El \simeq \El \otimes \Ll$. If moreover, $\Ll$ is trivial for every $t\in \Autz(C)$, then we say that $\El$ is \emph{homogeneous}.
\end{definition}

\begin{remark}\label{rmk:Atiyah}
\begin{enumerate}
    \item With \autoref{rmk:indecomposable}, Atiyah's theorems can be formulated as follows: for every indecomposable vector bundle $\El$ of rank $r\in \{2,3\}$ and degree $d\in \{0,\dots, r-1\}$, there exists a line bundle $\Ll$ of degree zero such that $\El \simeq \El_{r,d} \otimes \Ll$. As a direct consequence, the indecomposable vector bundles $\mathcal{E}_{r,d}$ are semi-homogeneous; or equivalently, the projective bundles $\Al_{r,d}$ are homogeneous, i.e., $t^*\Al_{r,d} \simeq \Al_{r,d}$ for every $t\in \Autz(C)$. 
    \item Let $r\in \{2,3\}$. 
    The line subbundle $\Ol_C\subset \El_{r,0}$ is unique since $\Gamma(C,\El_{r,0})=\k$.
    Notice that $\El_{r,0}$ is in fact homogeneous. 
    Indeed, for every $t\in \Autz(C)$, the vector bundle $t^*\El_{r,0}$ is indecomposable of rank $r$ and degree $0$, and it also admits $\Ol_C$ as a line subbundle. Hence, $\k\subset \Gamma(C,t^*\El_{r,0})$. Then it follows by uniqueness that $t^*\El_{r,0}\simeq \El_{r,0}$. 
\end{enumerate}
\end{remark}

\subsection{Repetition on ruled surfaces}

Let $b\in \mathbb{Z}$. We denote by $\F_b=\P(\Ol_{\P^1}\oplus \Ol_{\P^1}(b))$ the $b$-th Hirzebruch surface. The following definition is equivalent and is well-known to experts in toric geometry, see e.g. \cite{BFT23}.

\begin{definition}\label{def:Hirzebruch_surface}
    For every $b\in \mathbb{Z}$, the $b$-th Hirzebruch surface $\F_b$ is the quotient of $(\A^2\setminus \{0\})^2$ by the following action:
    \[
    \begin{array}{ccc}
    (\mathbb{G}_m)^2 \times (\A^2\setminus \{0\})^2 & \to & (\A^2\setminus \{0\})^2 \\
    (\mu,\rho), (y_0,y_1,z_0,z_1) & \longmapsto &  (\rho^{-b}\mu y_0, \mu y_1 , \rho z_0,\rho z_1).
    \end{array}
    \]
    The equivalence class of $(y_0,y_1,z_0,z_1)$ is denoted by $[y_0:y_1 \ ; z_0:z_1]$ and the $\P^1$-bundle structure morphism $\tau_b \colon \F_b\to \P^1$ is given by $ [y_0:y_1 \ ; z_0:z_1] \mapsto [z_0:z_1]$.
\end{definition}

\begin{remark}
    From the above definition, we notice that:
    \begin{itemize}
        \item If $b>0$, the unique section of self-intersection $(-b)$ is given by the equation $\{y_0 =0 \} \subset \F_b$.
        \item The ruled surfaces $\F_b$ and $\F_{-b}$ are isomorphic, via \[ [y_0:y_1 \ ; z_0:z_1] \mapsto [y_1:y_0 \ ; z_0:z_1].\]
    \end{itemize}
\end{remark}

The geometry of ruled surfaces and their automorphisms are well-known, see e.g. \cite{Fong20,Fong21,Maruyama}, and we recall some of their properties in the following:

\begin{theorem}\label{thm:dim_2}
    Let $C$ be a smooth projective curve. The following hold:
    \begin{enumerate}
        \item\label{dim_2:1} If $C$ is an elliptic curve, then the ruled surface $\Al_{2,0}$ has a unique minimal section of self-intersection zero and $\Autz(\Al_{2,0})$ is a maximal connected algebraic subgroup of $\Bir(\Al_{2,0})$ which fits into a short exact sequence:
        \[
        0 \to \mathbb{G}_a \to \Autz(\Al_{2,0}) \to \Autz(C) \to 0.
        \]
        \item\label{dim_2:2} If $C$ is an elliptic curve, then the ruled surface $\Al_{2,1}$ has infinitely many minimal sections of self-intersection one and $\Autz(\Al_{2,1})$ is a maximal connected algebraic subgroup of $\Bir(\Al_{2,1})$ which fits into a short exact sequence:
        \[
        0 \to (\mathbb{Z}/2\mathbb{Z})^2 \to \Autz(\Al_{2,1}) \to \Autz(C) \to 0.
        \]
        \item\label{dim_2:3} Let $\Ll \in \Pic^0(C)$ non-trivial and $S=\p(\Ol_C\oplus \Ll)$. Then $S$ has precisely two minimal sections, which are disjoint and of self-intersection zero, and $\Autz(S)$ is a maximal connected algebraic subgroup of $\Bir(S)$ which fits into a short exact sequence:
        \[
        0 \to \mathbb{G}_m \to \Autz(S) \to \Autz(C) \to 0.
        \]
        Moreover, the algebraic group $\Aut(S)$ fits into a short exact sequence:
        \[
        1 \to \mathbb{G}_m \rtimes \Omega \to \Aut(S) \to \Aut(C) ,
        \]
        where $\Omega = \mathbb{Z}/2\mathbb{Z}$ if and only if $\Ll$ is two-torsion, otherwise $\Omega$ is trivial.
    \end{enumerate}
\end{theorem}

\begin{remark}
    \begin{itemize}
    \item The algebraic group $G=\Autz(\Al_{2,1})$ is isomorphic to an elliptic curve. By \autoref{thm:almost product}, the ruled surface $\Al_{2,1}$ is isomorphic to the contracted product $G\times^H \p^1$, where $H=(\mathbb{Z}/2\mathbb{Z})^2$ and the $\p^1$-bundle structure morphism is identified to the projection on $G/H$. We will see in \autoref{ss:stable_proj} that this construction generalizes to $\Al_{3,1}$ and $\Al_{3,2}$.
    \item If $\Ll$ is two-torsion in \autoref{thm:dim_2} \autoref{dim_2:3}, the group $\Omega=\mathbb{Z}/2\mathbb{Z}$ comes from the two-torsion of $\Pic^0(C)$ and acts by permuting the two minimal sections of $S$. A similar phenomenon occurs for the $\p^2$-bundle $\p(\Ol_C\oplus \Ll \oplus \Ll^{\otimes 2})$, where $\Ll \in \Pic^0(C)$ is three-torsion, as we will see in \autoref{lem:sum_line_bundles}.
    \end{itemize}
\end{remark}

\subsection{Repetition on $\F_b$-bundles}

Blowing up a section of a $\P^2$-bundle over a curve gives rise to an $\F_1$-bundle. For self-containedness, we recall some results on $\F_b$-bundles (see \cite[Section 3]{BFT23} and \cite[Section 3]{Fong23} for details). 
In the following statement, we use the notation $U_{kl}:=U_k\cap U_l$ for open sets $U_k,U_l$.

\begin{lemma}[{\cite[Proposition 3.3.1]{BFT23}} and {\cite[Lemma 3.12]{Fong23}}]\label{lem:F_b-bundles}
    Let $b>0$ be an integer, $C$ a smooth projective curve and let $\tau\colon S\to C$ be a ruled surface with a section $\sigma$, and let $\pi \colon X\to S$ be a $\P^1$-bundle such that $\tau\pi$ is an $\F_b$-bundle. Then the following hold:
    \begin{enumerate}
    \item\label{F_b-bundles:1} There exist trivializations of $\tau \pi$ such that the transition maps are of the form
    \[
    \begin{array}{cccc}
        \phi_{kl}\colon & U_l \times \F_b & \dashrightarrow & U_k \times \F_b \\
        & (x,[y_0:y_1\ ; z_0:z_1]) & \longmapsto & (x,[y_0:\lambda_{kl}(x)y_1+ p_{kl}(x,z_0,z_1)y_0\ ; z_0:a_{kl}(x)z_0 + b_{kl}(x)z_1]),
    \end{array}
    \]
    where $\lambda_{kl}\in \Ol_C(U_{kl})^*$, $p_{kl}\in \Ol_C(U_{kl})[z_0,z_1]$ is homogeneous of degree $b$, and 
    \[
    \begin{pmatrix}
        1 & 0 \\ a_{kl} & b_{kl}
    \end{pmatrix} \in \GL_2(\Ol_C(U_{kl}))
    \]
    are the transition matrices of the $\P^1$-bundle $\tau$. Under this choice of coordinates, the section $\sigma$ is defined by the equation $\{z_0=0\}\subset S$. 

    \item\label{F_b-bundles:2} There exists a rank-$2$ vector bundle $\El$ such that $\P(\El)\simeq X$ and which fits into a short exact sequence
    \[
    0\to \Ol_S(b\sigma + \tau^*(D)) \to \El \to \Ol_S \to 0,
    \]
    where $D$ is a divisor on $C$ and the line subbundle $\Ol_S(b\sigma + \tau^*(D))$ corresponds to the section of $\pi$ spanned by the $(-b)$-sections along the fibres of $\tau\pi$. Moreover, the coefficients $\lambda_{kl}$ in \autoref{F_b-bundles:1} are the cocycles of the line bundle $\Ol_S(\tau^*(D))$.

    \item\label{F_b-bundles:3} We can choose $p_{kl}=0$ in \autoref{F_b-bundles:1} if and only if $\El$ is decomposable. In that case, we have $X\simeq \P(\Ol_S\oplus \Ol_S(b\sigma + \tau^*(D)))$.  
    \end{enumerate}
\end{lemma}

\begin{remark}\label{rmk:invariants}
    In \autoref{lem:F_b-bundles}, the linear class of the divisor $D$ is an invariant for the $\P^1$-bundle $\pi\colon X\to S$, and we say that $\pi$ has invariants $(S,b,D)$ (or equivalently, $(S,b,\mathcal{O}_C(D))$), see \cite[Remark 3.14, Definition 3.15]{Fong23}.
\end{remark}

The following lemma gives a necessary condition for the automorphism group of an $\F_b$-bundle to be a maximal connected algebraic subgroup of $\Bir(X)$ when $b>0$.

\begin{lemma}[{\cite[Lemma 3.7.1]{BFT23}} and {\cite[Proposition 4.5]{Fong23}}]\label{lem:F_b_trivial_action_on_C}
    Let $b>0$ be an integer and $C$ a smooth projective curve. Let $\pi \colon X\to S$ and $\tau\colon S\to C$ be $\P^1$-bundles such that $\tau\pi$ is an $\F_b$-bundle. 
    If $\Autz(X)$ has a fixed point on $C$, then $\Autz(X)$ is not a maximal connected algebraic subgroup of $\Bir(X)$.
\end{lemma}
\begin{proof}
If $g(C)=0$, the statement is \cite[Lemma 3.7.1]{BFT23}. If $g(C)\geq1$, then $\Autz(X)$ acting with a fixed point on $C$ implies that it acts trivially, i.e. $(\tau\pi)_*\colon \Autz(X)\to \Autz(C)$ is trivial, and the claim is proven in \cite[Proposition 4.5]{Fong23}.
\end{proof}

The case of $\F_0$-bundles is particularly explicit, as they correspond to fibre products of ruled surfaces.

\begin{lemma}[{\cite[Lemmas 3.6 and 3.8]{Fong23}}]\label{lem:F_0-bundles}  
    Let $C$ be a smooth projective curve, let $\tau\colon S\to C$ and $\pi \colon X\to S$ be $\P^1$-bundles such that $\tau\pi$ is an $\F_0$-bundle. Then there exists a ruled surface $\tau'\colon S'\to C$ such that 
    \[
    X \simeq S\times_C S',\text{ and }
    \Autz(X) \simeq \Autz(S) \times_{\Autz(C)} \Autz(S').
    \]
    Conversely, every fibre product of ruled surfaces yields an $\F_0$-bundle over $C$.
\end{lemma}

\subsection{Automorphisms of $\P^2$-bundles over a curve}

Let $\El$ be a vector bundle above a smooth projective curve $C$ and $X=\p(\El)$. Let $\pi\colon X\to C$ and $\pi'\colon\El\to C$ be the structure morphisms. We define
\[
\Aut(X)_C=\{\alpha\in\Aut(X)\mid \pi\circ\alpha=\pi\},
\qquad
\Aut(\El)_C=\{\alpha\in\Aut(\El)\mid  \pi'\circ\alpha=\pi'\}
\]

\begin{remark}\label{rmk:seq}
   In general, not every automorphism of $X$ over $C$ is induced by an automorphism of $\El$. In fact, there may exist a line bundle $\Nl$ such that $\El\simeq\El\otimes\Nl$, which induces an automorphism of $X$. In this case, notice that $\Nl$ is automatically torsion.
   
   By \cite[\S 5]{Grothendieck_geoform_geoalg}, there exists a short exact sequence
	\begin{equation}\label{SES:Grothendieck} \tag{$\dagger$}
	0 \to \Aut(\El)_C/\mathbb{G}_m \to \Aut(X)_C \to \Omega \to 0, 
	\end{equation}
	where $\Omega$ is a subgroup of the $\mathrm{rank}(\El)$-torsion subgroup of $\Pic^0(C)$ consisting of line bundles $\Nl$ such that $\El \simeq \El \otimes \Nl$. When $C$ is an elliptic curve and $\mathrm{rank}(\El)=3$, then $\Nl^{\otimes 3}\simeq \Ol_C$, so we can identify $\Omega$ with a subgroup of $(\mathbb{Z}/3\mathbb{Z})^2$.
    \end{remark}

	\begin{lemma}\label{lem:sum_line_bundles}
		Let $X$ be a decomposable $\p^2$-bundle over a smooth projective curve $C$. Then $\Omega\neq 0$ if and only if 
        \[X\simeq \p( \Ol_C\oplus \Ll \oplus \Ll^{\otimes 2})\] 
        where $\Ll$ is a non-trivial $3$-torsion line bundle. In this case, $\Omega = \mathbb{Z}/3\mathbb{Z}$, and the short exact sequence \autoref{SES:Grothendieck} splits.
	\end{lemma}
	
	\begin{proof}
		Let $\Nl\in \Omega$ and let $\El$ be a rank-$3$ vector bundle such that $X\simeq \p(\El)$. Since $\El$ is decomposable, there exists a rank-$2$ vector bundle $\El'$ and a line bundle $\Ml$ such that $\El \simeq \El' \oplus \Ml$. 
        
        If $\El'$ is indecomposable, then $\El \simeq \El \otimes \Nl$ implies that $\Ml \simeq \Ml \otimes \Nl$, so $\Nl \simeq \Ol_C$ and $\Omega =0$. 
        
        If $\El'$ is decomposable, then after tensoring with a line bundle we can assume that $\El \simeq \Ol_C \oplus \Ll \oplus \Ml$ for some line bundle $\Ll$. Then $\El\simeq \El \otimes \Nl$ gives
        \[
        \Ol_C \oplus \Ll \oplus \Ml \simeq \El\simeq \El\otimes\Nl\simeq \Nl \oplus (\Ll \otimes \Nl) \oplus (\Ml \otimes \Nl).
        \]
        If $\Nl$ is non-trivial, Krull-Schmidt theorem yields that
        $\Ll \otimes \Nl$ or $\Ml \otimes \Nl$ is trivial. If $\Ml \otimes \Nl$ is trivial, then $\Ll \simeq \Nl$ and $\Ll \otimes \Nl \simeq \Ml$. Hence, $\Ml\simeq \Ll^{\otimes 2}$ has order $3$. If $\Ll \otimes \Nl$ is trivial, then $\Ml \simeq \Nl$ and $\Ml^{\otimes 2}\simeq \Ll $ has order $3$. In both cases, we have $\Omega=\{\Ol_C,\Ll,\Ml\}$, which is a subgroup of $\Pic^0(C)[3]$ isomorphic to $\mathbb{Z}/3\mathbb{Z}$. 

        It remains to see that the short exact sequence splits for $\El \simeq \Ol_C\oplus \Ll \oplus \Ll^{\otimes 2}$, where $\Ll$ is a non-trivial $3$-torsion line bundle and $\Omega = \{\Ol_C ,\Ll,\Ll^{\otimes 2}\}$.
        There exist trivializations of $\El$ and $\El \otimes \Ll$, such that their respective transition matrices are
        \[ \Phi_{kl}=
	       \begin{pmatrix}
		      1 & 0 & 0 \\
		      0 & \lambda_{kl} & 0 \\
		      0 & 0& \lambda_{kl}^2
	       \end{pmatrix}, ~
               \Phi_{kl} \otimes \lambda_{kl} = 
	       \begin{pmatrix}
		      \lambda_{kl} & 0 & 0 \\
		      0 & \lambda_{kl}^2 & 0 \\
		      0 & 0 & 1
	       \end{pmatrix},
	       \]
           where $\lambda_{kl}$ are the cocycles of $\Ll$. Let $A$ be the invertible matrix
           \[
          A=  \begin{pmatrix}
		      0 & 0 & 1 \\
		      1& 0 & 0 \\
		      0 & 1 & 0
	       \end{pmatrix}.
           \]
           Then $ \Phi_{kl}\cdot A = A\cdot  \Phi_{kl}\otimes \lambda_{kl}$, which means that $A$ induces an isomorphism $\El \simeq \El\otimes \Ll$, which in turn gives an element $f_A\in \Aut(X)_C$ that is mapped to $\Ll\in\Omega$ in \autoref{SES:Grothendieck}. Since $A$ has order $3$, the subgroup generated by $f_A$ gives a section of $\Omega$ in $\Aut(X)_C$.           
	\end{proof}

    \begin{remark}
    Let $\El=\Ol_C\oplus \Ll \oplus \Ll^{\otimes 2}$ and $X=\p( \El)$, where $\Ll$ is a non-trivial $3$-torsion line bundle as in \autoref{lem:sum_line_bundles}. Then \autoref{SES:Grothendieck} splits, i.e., $$\Aut(X)_C \simeq (\Aut(\El)_C/\mathbb{G}_m )\rtimes \mathbb{Z}/3\mathbb{Z}.$$ Then, in \autoref{lem:sum_line_bundles} there exist three disjoint sections of the $\p^2$-bundle $X\to C$ corresponding to be line subbundles $\Ol_C$, $\Ll$ and $\Ll^{\otimes 2}$ of $\El$, that are permuted by $\Aut(X)_C$. The stabilizer in $\Aut(X)_C$ of any such section is the subgroup $\Aut(\El)_C/\mathbb G_m$.
    \end{remark}

    The following remark introduces a key matrix computation that will be used to determine the group $\Aut(\El)_C$ in the decomposable rank-$3$ cases.
    
	\begin{remark}\label{key_computation}
    Let $\El$ be a decomposable rank-$3$ vector bundle of the form $\El\simeq \El'\oplus\Ml$, where $\Ml$ is a line bundle and $\El'$ is a the rank-$2$ vector bundle containing a trivial line subbundle $\Ol_C \subset \El'$ and we set $\Ll:=\El'/\Ol_C$. Later in \autoref{lem:hom_proj_bundles} and \autoref{lem:non-hom_proj_bundles}, all decomposable vector bundles are of this form.
    
    Then, we can choose a trivializing open cover $(U_k)_k$ of $C$ such that the transition matrices of $\El$ equal
	\[
	T_{kl} = 
	\begin{pmatrix}
		1 & \alpha_{kl} & 0 \\
		0 & \lambda_{kl} & 0 \\
		0 & 0& \mu_{kl}
	\end{pmatrix},
	\]
	where $\lambda_{kl}, \mu_{kl}\in \Ol_C(U_{kl})^*$ are respectively the cocycles of the line bundles $\Ll$ and $\Ml$, $\alpha_{kl}\in \Ol_C(U_{kl})$, and $U_{kl} = U_k \cap U_l$.   
	
	On $U_k$ and $U_l$ an element of $\Aut(\El)_C$ is, respectively, given by invertible matrices
	\[
	M_{k} = 
	\begin{pmatrix}
		a_k & b_k & c_k \\
		d_k & e_k & f_k \\
		g_k & h_k& i_k
	\end{pmatrix}, ~
	M_{l} = 
	\begin{pmatrix}
		a_l & b_l & c_l \\
		d_l & e_l & f_l \\
		g_l & h_l& i_l
	\end{pmatrix},
	\]	
    with coefficients in $\Ol_C(U_k)$ and $\Ol_C(U_l)$.
	Developing the equality $M_k T_{kl} =  T_{kl} M_l$, we obtain the following equality in $\GL_3(\Ol_C(U_{kl}))$:
    \[
	\begin{pmatrix}
		a_k & \alpha_{kl}a_k + \lambda_{kl}b_k & \mu_{kl}c_k \\
		d_k & \alpha_{kl}d_k+ \lambda_{kl}e_k & \mu_{kl}f_k \\
		g_k & \alpha_{kl}g_k+ \lambda_{kl}h_k& \mu_{kl}i_k 
	\end{pmatrix}
    = 
	\begin{pmatrix}
		a_l +\alpha_{kl}d_l & b_l+\alpha_{kl}e_l & c_l+\alpha_{kl}f_l \\
		\lambda_{kl} d_l & \lambda_{kl}e_l & \lambda_{kl}f_l \\
		\mu_{kl}g_l & \mu_{kl}h_l& \mu_{kl}i_l
	\end{pmatrix}.
    \]
    Identifying the coefficients, we obtain the following conditions:
	\[
	\left\{ 
	\begin{aligned}
		d_k & = \lambda_{kl} d_l, \\
		g_k & =  \mu_{kl} g_l, \\ 
		\mu_{kl} f_k & =  \lambda_{kl} f_l,\\
        \mu_{kl} i_k & = \mu_{kl} i_l.
	\end{aligned}
	\right.
	\]
	This yields global sections $d\in \Gamma(C,\Ll)$, $g\in \Gamma(C,\Ml)$, $f\in \Gamma(C,\Ml^{-1}\otimes \Ll)$, and $i\in \k^*$. 
    \end{remark}

For the rest of this section, we focus on the case where $C$ is an elliptic curve. We will treat the case $g(C)\geq2$ in \autoref{s:higher genus}.

\subsection{The homogeneous $\P^2$-bundles over an elliptic curve}

\begin{lemma}\label{lem:pi-star surjective}
    Let $C$ be an elliptic curve. Let $\El$ be a vector bundle on $C$ and denote by $\pi\colon \P(\El)\to C$ the projective bundle structure morphism. Then the morphism $\pi_*\colon \Autz(\P(\El))\to \Autz(C)$ is surjective if and only if $\El$ is semi-homogeneous.
\end{lemma}

\begin{proof}
Let $X:=\P(\El)$.
    Assume that $\pi_*$ is surjective. Then for every translation $t\in \Autz(C)$, we have $t^*X\simeq X$. Hence, by the Corollary of Proposition 2 in \cite[\S 5]{Grothendieck_geoform_geoalg}, it follows that $\El$ is semi-homogeneous. Conversely, if $\El$ is semi-homogeneous, we have $t^*X\simeq X$ for any translation $t\in\Autz(C)$, so $t$ can be lifted to an automorphism of $X$. It remains to show that $t$ can be chosen in the connected component $\Autz(X)$.

    Since every fibre of $\pi$ is rational and $C$ is non-rational, every automorphism $f\in \Aut(X)$ preserves the fibration $\pi$, i.e. there exists $g\in \Aut(C)$ such that $\pi\circ f = g \circ\pi$ (see also \autoref{lem:preserves fibration}). Hence, there exists a group homomorphism $\Pi_* \colon \Aut(X)\to \Aut(C)$. 
    Recall that we denote by $\pi_*\colon \Autz(X)\to \Autz(C)$ the morphism of algebraic groups induced by Blanchard's lemma (\autoref{lem:blanchard}). By uniqueness, ${\Pi_*}_{\vert \Autz(X)} = \pi_*$.
    
   Furthermore, $\Aut(X)$ is an algebraic group. Since $\Pi_*$ is a group homomorphism, the following diagram commutes for every $f\in \Aut(X)$:
    \[
    \begin{tikzcd}
    \Autz(X) \ar[d,"\pi_*",swap]\ar[r,"t_f"] & f\Autz(X)\ar[d,"{\Pi_*}_{\vert f\Autz(X)}"]\\
    \Autz(C) \ar[r,"t_{\Pi_*(f)}",swap] & \Pi_*(f)\Autz(C), 
    \end{tikzcd}
    \]
    where $t_{\Pi_*(f)}$ and $ t_f$ denote the translations $h \to \Pi_*(f) \circ h$ and $g \to f \circ g$, respectively. Since $t_f$ and $t_{\Pi_*(f)}$ are isomorphisms of varieties, $\Pi_*$ is a morphism of algebraic groups.
    
    We have shown that any $t\in\Autz(C)$ can be lifted to $\Aut(X)$, so the image of $\Pi_*$ contains $\Autz(C)$. 
    If $\pi_*$ is trivial, it would imply that the image of $\Pi_*$ is finite, which is a contradiction; thus, $\pi_*$ is surjective. 
\end{proof}

By \autoref{thm:Atiyah} and \autoref{rmk:Atiyah}, the projective bundles $\Al_{r,d}$ are homogeneous, as the vector bundles $\El_{r,d}$ are semi-homogeneous. Notice also that $\El_{r,d}$ is not necessarily homogeneous: for example, $\El_{2,0}$ is homogeneous but $\El_{2,1}$ is not. In the next lemma, we give the list of projective bundles such that the morphism of algebraic groups \[\pi_*\colon \Autz(X)\to \Autz(C)\] induced by Blanchard's lemma is surjective (see \autoref{lem:blanchard}). 

\begin{lemma}\label{lem:hom_proj_bundles}
	Let $C$ be an elliptic curve and $\pi\colon X\to C$ a $\P^2$-bundle. Then $\pi_*$ is surjective if and only if $X$ is isomorphic as a $\P^2$-bundle to one of the following:
	\begin{enumerate}
		\item\label{homog:1} $\mathcal{A}_{3,i}$ where $i\in \{0,1,2\}$, i.e., $\pi$ is indecomposable,
		\item\label{homog:2} $\P( \mathcal{E}_{2,0}\oplus \Ml)$, where $\Ml\in \Pic^0(C)$,
		\item\label{homog:3} $\P(\Ol_C\oplus \Ll \oplus \Ml)$, where $\Ll,\Ml\in \Pic^0(C)$.
	\end{enumerate}
\end{lemma}

\begin{proof}
	For every $i=1,2,3$, the projective bundle $\Al_{3,i}$ is homogeneous (see \autoref{rmk:Atiyah}). 
    Hence, the morphism $\pi_*\colon\Autz(X)\to\Autz(C)$ induced by Blanchard's lemma (\autoref{lem:blanchard}) is surjective by \autoref{lem:pi-star surjective}. For a decomposable projective bundle $X=\P(\El)$, two cases arise:
	
	\begin{enumerate}
		\item[(i)] Assume that $\El$ is isomorphic to the direct sum of an indecomposable rank-$2$ vector bundle and a line bundle. Tensoring $\El$ by a line bundle, we do not change the isomorphism class of $X$, so by \autoref{thm:Atiyah}, we may assume that $\El \simeq \El_{2,i} \oplus \Ml$, where $i\in \{0,1\}$ and $\Ml$ is a line bundle. As $\El_{2,0}$ itself is an homogeneous vector bundle, the vector bundle $\El_{2,0} \oplus \Ml$ is semi-homogeneous if and only if $\Ml$ has degree zero. 
		
		Next, we show that for any line bundle $\Ml$, the vector bundle $\El_{2,1}\oplus \Ml$ is not semi-homogeneous. Pulling back the short exact sequence that defines $\El_{2,1}$ in \autoref{thm:Atiyah} by a translation $t\in \Autz(C)$ of infinite order, we get:
		\[
		0  \to  \mathcal{O}_C  \to  t^*\mathcal{E}_{2,1}  \to  \mathcal{O}_C(t^*p_0) \to 0.
		\]
		Moreover, $t^*p_0 \sim q$, where $q$ is defined by the translation $t\colon x \mapsto x -q$. The vector bundle $\mathcal{E}_{2,1}$ is semi-homogeneous, so there exists a line bundle $\Ml'$ such that $t^*\mathcal{E}_{2,1} \simeq \mathcal{E}_{2,1} \otimes \Ml'$. Taking the determinant from both sides, we obtain $\Ml'^{\otimes 2} \simeq \Ol(q-p_0)$. 
        
        Assume by contradiction that $\El_{2,1}\oplus \Ml$ is semi-homogeneous. Then there exists a line bundle $\Ml''$ such that $t^*(\El_{2,1}\oplus \Ml) \simeq (\El_{2,1}\oplus \Ml) \otimes \Ml''$. Since pullback commutes with direct sums, we obtain that 
        \[
        (\mathcal{E}_{2,1} \otimes \Ml') \oplus t^*\Ml \simeq (\El_{2,1} \otimes \Ml'') \oplus (\Ml \otimes \Ml'').
        \]
        By Krull-Schmidt theorem \cite[Theorem 1]{Ati56}, we get by identification that 
        \begin{equation}\label{eq:isom_line-bundle}
    	\left\{ 
    	\begin{aligned}
    		 \mathcal{E}_{2,1} \otimes \Ml' & \simeq \El_{2,1} \otimes \Ml'', \\
    		  t^*\Ml & \simeq  \Ml \otimes \Ml''
    	\end{aligned}
    	\right.
	   \end{equation}
       If $\Ml$ has degree zero, then it is homogeneous, and it follows from the isomorphisms above that $\Ml''$ is trivial and $\Ml'$ is two-torsion, which is a contradiction since $t$ has infinite order. Therefore, $\Ml$ has degree $d\neq 0$; in particular, $\Ml$ is isomorphic to $\Ol_C(dz)$ for some $z\in C$. This implies that $t^*\Ml\simeq \Ml \otimes \Ol_C(d(q-p_0))$. So by the second isomorphism in \autoref{eq:isom_line-bundle}, we get $\Ml''\simeq \Ol_C(d(q-p_0))$. The first isomorphism in \autoref{eq:isom_line-bundle} implies that $\Ml'^{\otimes 2} \simeq \Ml''^{\otimes 2}$, hence $\Ml'^{\otimes 2} \simeq \Ol_C(2d(q-p_0))$. Since $\Ml'^{\otimes 2} = \Ol(q-p_0)$ and $d\neq0$, this implies that $q-p_0$ is a torsion divisor, i.e., $q\in C$ is a torsion point, which contradicts that $t$ has infinite order.
		
		\item[(ii)] Assume that $\El$ is isomorphic to the direct sum of three line bundles. Up to tensoring by a line bundle, we can assume that $\El$ is isomorphic to $\Ol_C \oplus \Ll \oplus \Ml$, where $\Ll, \Ml$ are line bundles. Since $\Ol_C$ is invariant by translations, it follows that $\El$ is semi-homogeneous if and only if $\Ll$ and $\Ml$ have degree zero.
	\end{enumerate}
    Combining the different cases above proves the lemma.
\end{proof}

\autoref{lem:hom_proj_bundles} provides the list of homogeneous projective bundles $\pi\colon X\to C$ above an elliptic curve $C$ together with a rank-$3$ vector bundle $\El$ such that $X\simeq \P(\El)$. In the next lemma, we give a representative $\El$ for each non-homogeneous projective bundle, i.e., those such that $\pi_*\colon \Autz(X)\to \Autz(C)$ is trivial.

\begin{lemma}\label{lem:non-hom_proj_bundles}
	Let $C$ be an elliptic curve and $\pi\colon X\to C$ a $\P^2$-bundle. Then $\pi_*$ is trivial if and only if $X$ is isomorphic as $\P^2$-bundle to one of the following:
	\begin{enumerate}
		\item\label{non-hom:1} $\P(\El_{2,0} \oplus \Ml)$ where $\Ml$ is a line bundle with non-zero degree,
		\item\label{non-hom:2} $\P(\El_{2,1} \oplus \Ml)$ where $\Ml$ is any line bundle,
		\item\label{non-hom:3} $\P(\Ol_C\oplus \Ol_C \oplus \Ml)$ where $\Ml$ is a line bundle with non-zero degree,
		\item\label{non-hom:4} $\P(\Ol_C\oplus \Ll \oplus \Ml)$ where $\Ll,\Ml$ are non-trivial and non-isomorphic line bundles, such that $0\leq \deg(\Ll) \leq \deg(\Ml)$ and $\deg(\Ml)>0$.
	\end{enumerate}
\end{lemma}

\begin{proof}
Let $\El$ be a vector bundle such that $X\simeq\p(\El)$. Since $\Autz(C)$ is an elliptic curve, \autoref{lem:pi-star surjective} implies that $\pi_*$ is trivial if and only if $\El$ is not semi-homogeneous.
\autoref{thm:Atiyah} and \autoref{rmk:Atiyah} imply that $\El$ is decomposable. 
	If $\El$ decomposes as a sum of an indecomposable rank-$2$ vector bundle and a line bundle, the claim follows directly from \autoref{lem:hom_proj_bundles}. If $\El$ decomposes as a sum of three line bundles and tensoring by a line bundle, we can assume that $\El$ is isomorphic to $\Ol_C \oplus \Ll \oplus \Ml$. If $\Ll \simeq \Ml$, we tensor $\El$ by $\Ll^\vee$ and obtain the third case. Else, permuting $\Ll$ and $\Ml$ if necessary, we can also assume that $\deg(\Ll) \leq \deg(\Ml)$. Tensoring again by $\Ll^\vee$ if $\deg(\Ll)<0$, we can moreover assume that $\deg(\Ll) \geq 0$, and we obtain the last case. Again by \autoref{lem:hom_proj_bundles}, we have that $\deg(\Ml)\neq 0$ in the last two cases.
\end{proof}

We will see in \autoref{prop:trivial_action_on_C} that $\Autz(X)$ with $X$ from \autoref{lem:non-hom_proj_bundles} are not maximal connected algebraic subgroups of $\Bir(X)$.

    \subsection{When the automorphism group acts trivially on the base}

In this section, we show the following statement:

\begin{proposition}\label{prop:trivial_action_on_C}
    Let $C$ be an elliptic curve and $\pi\colon X\to C$ be a $\P^2$-bundle such that $\pi_*$ is trivial (or, equivalently, $X=\p(\El)$ and $\El$ is not semi-homogeneous by \autoref{lem:pi-star surjective}). Then $\Autz(X)$ is not a maximal connected algebraic subgroup of $\Bir(X)$.
\end{proposition}

Before we prove anything, let us state a direct consequence of \autoref{prop:Mdp to bundle} and \autoref{prop:trivial_action_on_C}:
\begin{corollary}\label{cor:trivial_action_on_C}
Let $C$ be an elliptic curve and $\pi\colon X\to C$ a Mori Del Pezzo fibration of degree $9$. If $\pi_*$ is trivial, then $\Autz(X)$ is not a maximal connected algebraic subgroup of $\Bir(X)$. 
\end{corollary}

    Let us start with the following observation:
    
    \begin{lemma}\label{lem:not_max}
    Let $\El$ be a rank-$3$ vector bundle over a smooth projective curve $C$ and $\pi\colon \p(\El)\to C$ the $\p^2$-bundle morphism. Suppose that $\pi_*$ is trivial.
		With the notations from \autoref{key_computation}, if $\Gamma(C,\Ml^{-1}\otimes \Ll) = \Gamma(C,\Ml^{-1}) =0$, then $\Autz(\p(\El))$ is not a maximal connected algebraic subgroup of $\Bir(\p(\El))$. 
	\end{lemma}
	
	\begin{proof}
    Let $X=\p(\El)$.
		By assumptions, we have $f=0$ and then $c\in \Gamma(C,\Ml^{-1})$, so $c=0$. Therefore, the section of $\pi$ defined as 
        \[
        C\times \{[0:0:1]\} \subset X,
        \]
        and corresponding to the line subbundle $\Ml\subset \El$, is an $\Autz(X)$-invariant curve, whose blow-up yields a birational morphism $\eta\colon X' \to X$ such that $\eta^{-1} \Autz(X) \eta \subseteq \Autz(X')$. Moreover, $X'$ is equipped with the structure of an $\F_1$-bundle $X'\to C$, on which $\eta^{-1} \Autz(X) \eta$ acts by automorphisms inducing the identity on $C$. This implies that $\Autz(X)$ is not maximal by \autoref{lem:F_b_trivial_action_on_C}.
	\end{proof}

In the next two lemmas, we go through the cases of \autoref{lem:non-hom_proj_bundles}.

\begin{lemma}\label{lem:dec_withdeg0term}
 Let $C$ be an elliptic curve, and let $\pi\colon X\to C$ be one of the $\P^2$-bundles 
 \[
 \p(\El_{2,0}\oplus\Ml),\quad\p(\Ol_C\oplus\Ol_C\oplus\Ml)
 \]
 with $\Ml$ a line bundle with non-zero degree, 
 or 
 \[
\p(\Ol_C\oplus\Ll\oplus\Ml)
 \]
 with $\Ll,\Ml$ non-trivial line bundles, $\Ll\nsimeq \Ml$ and $0\le\deg(\Ll)\leq\deg(\Ml)$ and $\deg(\Ml)>0$
 (as in \autoref{lem:non-hom_proj_bundles} \autoref{non-hom:1},\autoref{non-hom:3} or \autoref{non-hom:4}). Then $\Autz(X)$ is not a maximal connected algebraic subgroup of $\Bir(X)$.
\end{lemma}

\begin{proof}
    This follows directly from \autoref{lem:not_max} in the last case, and also in cases \autoref{non-hom:1} and \autoref{non-hom:3} under the additional assumption that $\deg(\Ml)>0$. We continue the proof with $\El=\El' \oplus \Ml$, where $\El' = \El_{2,0}$ or $\El'=\Ol_C\oplus \Ol_C$, and $\deg(\Ml) <0$. In both cases, $\El'$ admits  $\Ol_C$ as line subbundle and the quotient $\Ll:=\El'/\Ol_C\simeq\Ol_C$ is also trivial (see \autoref{thm:Atiyah} if $\El'=\El_{2,0}$).

    With the notation from \autoref{key_computation}, $g$ is a global section of $\Ml$; hence, $g=0$. The rational functions $\lambda_{kl}$ are the cocycles of $\Ll$, which is trivial, so we can assume that $\lambda_{kl}=1$ for all $k,l$. Then, $h$ is a global section of $\Ml$, so $h=0$ as well. This implies that for every $p\in C$, the projective line $X'_p=\{p,[u:v:0] \}$ corresponding to the fibre at $p$ of the $\P^1$-subbundle $\P(\El') \subset \P(\El)$, is $\Autz(X)$-invariant.
    
    Choose $p\in U_k$. Shrinking $U_l$ if necessary, we can assume that $p\notin U_l$ for every $l\neq k$. Let $\xi\in \k(C)^*$ such that $\mathrm{div}(\xi)_{|U_k}=p$ and define the birational map 
    \[
    \begin{array}{cccc}
        \phi_k \colon & U_k \times \P^2 & \dashrightarrow & U_k\times \P^2 \\
        & (x,[u:v:w]) & \longmapsto & (x,[\xi(x)u:\xi(x)v:w]).
    \end{array}
    \]
    On every other local trivialization, we choose $\phi_l\colon U_l\times \P^2 \to U_l \times \P^2$ to be the identity. This induces an $\Autz(X)$-equivariant birational map $\phi\colon X\dashrightarrow X'$, which is the blow-up of the $\Autz(X)$-invariant projective line $X'_p$ followed by the contraction of the strict transform of the fibre $\pi^{-1}(p)$. Then the resulting $\P^2$-bundle $\pi'\colon X'\to C$ has 
    \[
        \begin{pmatrix}
        1 & \alpha_{kl} & 0 \\ 0 & 1 & 0 \\ 0 & 0 & \xi^{-1}\mu_{kl}
        \end{pmatrix}
    \]
    as transition matrices. Thus, $X'\simeq \P(\El' \oplus (\Ml \otimes \Ol_C(-p)))$ and $\phi \Autz(X) \phi^{-1} \subseteq \Autz(X')$. 

    Moreover, the base locus of $\phi^{-1}$ is the point $(p,[0:0:1])\in X'$, which is not $\Autz(X')$-invariant. Indeed, in the notation of \autoref{key_computation} applied to $\El' \oplus (\Ml \otimes \Ol_C(-p))$, we have that
    $f$ is a global section of the line bundle $(\Ml \otimes \Ol_C(-p))^{-1}$, which has degree at least two and whose associated complete linear system has no base points. Choosing $f$ non-vanishing at $p$ yields an automorphism of $X'$ that does not fix the point $(p,[0:0:1])$ and this concludes the proof.
\end{proof}

\begin{lemma}\label{lem:decomposable_E21}
Let $C$ be an elliptic curve, and let $X=\P(\El_{2,1} \oplus \Ml)$, where $\Ml\in \Pic(C)$ is any line bundle (as in \autoref{lem:non-hom_proj_bundles} \autoref{non-hom:2}). 
Then $\Autz(X)$ is not a maximal connected algebraic subgroup of $\Bir(X)$.
\end{lemma}

\begin{proof}
    First, note that by \autoref{thm:Atiyah}, the quotient of $\El_{2,1}$ by its line subbundle $\Ol_C$ is isomorphic to $\Ol_C(p_0)$, where $p_0$ denotes the neutral element of $C$. Using the notations from \autoref{key_computation}, we have $\Ll \simeq \Ol_C(p_0)$. By \autoref{lem:not_max}, if $\deg(\Ml)\geq 1$ and $\Ml \nsimeq \Ol_C(p_0)$, then $\Autz(X)$ is not maximal. We therefore continue the proof under the assumption that $\deg(\Ml) \leq 0$ or $\Ml = \Ol_C(p_0)$. Without loss of generality, we can assume that $p_0\in U_k$ for a fixed index $k$ and $p_0\notin U_l$ for every $l\neq k$.

    First, suppose that $\Ml$ is non-trivial. Then $d_k(p_0)=g_k(p_0)=0$. This implies that the point $\{(p_0,[1:0:0])\}$ is $\Autz(X)$-invariant. Let $\psi\colon X\dashrightarrow X'$ be the birational map defined locally as
    \[
    \begin{array}{cccc}
        \psi_k \colon & U_k \times \P^2 & \dashrightarrow & U_k \times \P^2 \\
        & (x,[u:v:w]) & \longmapsto & (x,[\zeta(x)u:v:w])
    \end{array}
    \]
    where $\zeta\in \k(C)^*$ satisfies $\mathrm{div}(\zeta)_{\vert U_k} = p_0$; and $\psi_l$ as the identity on $U_l \times \P^2$. Notice that $\psi$ is $\Autz(X)$-equivariant, as it is the blowup-up of the $\Autz(X)$-invariant point $\{(p_0,[1:0:0])\}$ followed by the contraction of the strict transform of the fibre $\pi^{-1}(p_0)$, where $\pi\colon X\to C$ denotes the $\p^2$-bundle structure morphism.
    Moreover, $\lambda_{kl}$ are the cocycles of the line bundle $\Ol_C(p_0)$, so for every $l\neq k$, we can in fact choose $\lambda_{kl} = \zeta$. Then the target $\P^2$-bundle $\pi'\colon X'\to C$ has transition matrices of the form
    \[
    \begin{pmatrix}
    1 & \alpha_{kl} & 0 \\ 0 & 1 & 0 \\ 0 & 0 & \zeta^{-1}\mu_{kl}
    \end{pmatrix},
    \]
    where $\alpha_{kl}\in \Ol_C(U_{kl})$.
    Therefore, $X'\simeq \P(\El'\oplus (\Ol_C(-p_0)\otimes \Ml))$, where $\El'\simeq\Ol_C \oplus \Ol_C$ or $\El'\simeq\El_{2,0}$ (see \autoref{thm:Atiyah}).  
    Our assumptions on $\Ml$ imply that $\deg(\Ol_C(-p_0)\otimes \Ml) < 0$ or that $\Ol_C(-p_0)\otimes \Ml$ is trivial. 
    The first case is listed in \autoref{lem:non-hom_proj_bundles} and is not semi-homogeneous, and \autoref{lem:dec_withdeg0term} implies that $\Autz(X)$ is not maximal.
    The second case is homogeneous, listed in \autoref{lem:hom_proj_bundles}, and \autoref{lem:not_max} implies that $\Autz(X)$ is not maximal.
    
    Suppose that $\Ml$ is trivial. Then $d,f\in \Gamma(C,\Ol_C(p_0))$ (see \autoref{key_computation}), hence $d_k(p_0)=f_k(p_0)=0$. This implies that the projective line $\{(p_0,[u:0:w])\mid u,w\in \k\}\subset\{p_0\}\times \p^2$ is $\Autz(X)$-invariant. We conclude similarly as in the proof of \autoref{lem:dec_withdeg0term}: we consider the $\Autz(X)$-equivariant birational map $\phi\colon X\dashrightarrow X'$, locally defined as
    \[
    \begin{array}{cccc}
        \phi_k \colon & U_k \times \P^2 & \dashrightarrow & U_k \times \P^2 \\
        & (x,[u:v:w]) & \longmapsto & (x,[\zeta(x) u:v:\zeta(x) w]),
    \end{array}
    \]
    where $\zeta\in\k^*(C)$ such that $\rm{div}(\zeta)|_{U_k}=p_0$, and $\phi_l$ as the identities on $U_l\times \P^2$ for every $l\neq k$. 
    Notice that as before, we can choose $\lambda_{kl}=\zeta$, so the target $\p^2$-bundle has transition matrices of the form
    \[
    \begin{pmatrix}
    1 & \alpha_{kl} & 0 \\ 0 & 1 & 0 \\ 0 & 0 & 1
    \end{pmatrix}.
    \]
    Then  $X'\simeq\p(\El'\oplus \Ol_C)$, where $\El'\simeq\Ol_C \oplus \Ol_C$ or $\El'\simeq\El_{2,0}$. Then $\pi'_*\colon \Autz(X')\to  \Autz(C)$ is surjective by \autoref{lem:hom_proj_bundles}. Since $\pi_*$ is trivial by \autoref{lem:non-hom_proj_bundles}, it follows that $\Autz(X)=\Autz(X)_C$ is not maximal. 
\end{proof}

\begin{proof}[Proof of \autoref{prop:trivial_action_on_C}]
If $\pi_*$ is trivial, then $X$ is isomorphic, as $\p^2$-bundle, to one of the bundles in \autoref{lem:non-hom_proj_bundles}.
The claim follows from \autoref{lem:dec_withdeg0term} and \autoref{lem:decomposable_E21}, treating the cases \autoref{lem:non-hom_proj_bundles} \autoref{non-hom:1}, \autoref{non-hom:3}, \autoref{non-hom:4} and \autoref{lem:non-hom_proj_bundles} \autoref{non-hom:2}, respectively.
\end{proof}

\subsection{Automorphisms of homogeneous decomposable projective bundles}\label{section:decomposable_hom_bundles}
In the last section, we showed in \autoref{prop:trivial_action_on_C} that the automorphism groups of non-homogeneous $\p^2$-bundles are not maximal connected algebraic subgroups in $\Bir(X)$.
In this section, we treat the cases when the $\p^2$-bundle is homogeneous and decomposable. By \autoref{lem:pi-star surjective}, these are precisely the cases from \autoref{lem:hom_proj_bundles} \autoref{homog:2}, \autoref{homog:3}, i.e. 
\[
X=\p(\Ol_C\oplus\Ll\oplus\Ml)\quad\text{or}\quad X=\p(\El_{2,0}\oplus\Ml), \quad\text{with }\Ll,\Ml\in\Pic^0(C).
\]
Notice that if $\Ll\simeq\Ml\simeq\Ol_C$, then $X\simeq C\times \P^2$ and $\Autz(X)\simeq \Autz(C)\times \Autz(\p^2)$ is a maximal connected algebraic subgroup of $\Bir(X)$.

\begin{lemma}\label{lem:O+L+M}
Let $C$ be an elliptic curve and  
\[
X=\p(\Ol_C\oplus\Ll\oplus\Ml),\quad\text{with }\Ll,\Ml\in\Pic^0(C)\setminus\{\Ol_C\}\ \text{and}\ \Ll\nsimeq\Ml
\]
and $\pi\colon X\to C$ the $\p^2$-bundle morphism.
    Then $\ker(\pi_*) =\mathbb{G}_m^2$ and the following hold:
    \begin{enumerate}
        \item\label{O+L+M:1} There exist exactly three $\Autz(X)$-invariant sections of $\pi$, corresponding to the line subbundles $\Ol_C$, $\Ll$, $\Ml$, and which are the $\Autz(X)$-orbits of minimal dimension. 
        The blowup of such an $\Autz(X)$-orbit gives rise to a Sarkisov link of type I, whose target is an $\F_1$-bundle over $C$. Moreover, there is no other $\Autz(X)$-equivariant Sarkisov link starting from $\pi$.
        \item\label{O+L+M:2} The group $\Autz(X)$ is conjugate to $\Autz\left(\P(\Ol_C \oplus \Ll) \times_C \P(\Ol_C \oplus \Ml)\right)$, and $\P(\Ol_C \oplus \Ll) \times_C \P(\Ol_C \oplus \Ml)\to C$ is an $\F_0$-bundle.
    \end{enumerate}
\end{lemma}

\begin{proof}
    Since $\Ml,\Ll\in\Pic^0(C)\setminus\{\Ol_C\}$ are non-isomorphic, we have $\Gamma(C,\Ll) = \Gamma(C,\Ml) =\Gamma(C,\Ml^{-1}\otimes\Ll)=0$. 
    Let $\El:=\Ol_C\oplus\Ll\oplus\Ml$.
    Using the notation from \autoref{key_computation}, we have $d=g=f=\alpha_{kl}=0$ and hence $b=c=h=0$. We obtain that $\Aut(\El)_C$ is isomorphic to $\mathbb{G}_m^3$, consisting of diagonal matrices with coefficients in $\k$. By \autoref{lem:sum_line_bundles}, $\Aut(X)_C$ is either isomorphic to $\mathbb{G}_m^2$ or $\mathbb{G}_m^2 \rtimes \mathbb{Z}/3\mathbb{Z}$. Since $\ker(\pi_*)=\Autz(X)_C=\Aut(X)_C \cap \Autz(X)$, it follows in both cases that $\ker(\pi_*) \simeq \mathbb{G}_m^2$, and this group acts on $X$ via diagonal matrices in $\PGL_3(\k)$.
    
    \autoref{O+L+M:1} The homomorphism $\pi_*$ is surjective by \autoref{lem:hom_proj_bundles}, so the $\Autz(X)$-orbits have dimension at least one. Since $\mathbb{G}_m^2$ acts via diagonal matrices on the fibres, the $\Autz(X)$-invariant curves are precisely the sections $C\times \{[1:0:0]\}$, $C\times \{[0:1:0]\}$, and $C\times \{[0:0:1]\}$, corresponding to the three line subbundles given in the statement. The blowup of any of the three $\Autz(X)$-invariant section yields a birational morphism, locally of the form $U\times \F_1\to U\times \p^2$. This gives an $\Autz(X)$-equivariant Sarkisov link of type I, whose target is an $\F_1$-bundle. This determines all possible $\Autz(X)$-equivariant link of type I and II. Moreover, there is no link of type III and IV by \autoref{lem: links III and IV}.

    \autoref{O+L+M:2} Let $\eta\colon X'\to X$ be the blow-up of the section $C\times \{[0:0:1]\}$ corresponding to the line subbundle $\Ml$. Since this latter is $\Autz(X)$-invariant, we have $\eta^{-1} \Autz(X) \eta \subseteq \Autz(X')$. Locally, on every trivializing open subset $U$, the birational morphism $\eta$ is given by
    \[
    \begin{array}{ccc}
        U\times \F_1 & \to & U  \times \P^2 \\
        (x,[y_0:y_1\ ;z_0:z_1]) & \longmapsto & (x,[y_0z_0:y_0z_1:y_1]);
    \end{array}
    \]
    hence $X'$ has a structure of $\F_1$-bundle over $C$. Equivalently, $\eta$ is the contraction of the surface $\{y_0=0\}$ spanned by the $(-1)$-sections along the fibres of the $\F_1$-bundle, which is $\Autz(X')$-invariant. This implies that $\eta^{-1} \Autz(X) \eta = \Autz(X')$. Computing the transition maps of the $\F_1$-bundle $X'\to C$ (in the notation from \autoref{key_computation}), we obtain
    \[
    \begin{array}{ccc}
        U_l\times \F_1 & \dashrightarrow & U_k \times \F_1 \\
        (x,[y_0:y_1\ ;z_0:z_1]) & \longmapsto & (x,[y_0:\mu_{kl}(x) y_1\ ;z_0:\lambda_{kl}(x) z_1]).
    \end{array}
    \]
    This shows that $X'$ is equipped with a $\P^1$-bundle structure morphism $\pi'\colon X' \to \p(\Ol_C\oplus\Ll)$ with invariants $(\p(\Ol_C\oplus\Ll),1,\Ml)$, see \autoref{rmk:invariants} for the notation. Recall that by \autoref{thm:dim_2}, the ruled surface $\tau\colon \p(\Ol_C\oplus\Ll)\to C$ admits precisely two minimal sections $\sigma$ and $\sigma'$, which are disjoint and $\Autz(\p(\Ol_C\oplus\Ll))$-invariant, respectively defined by the equations $\{z_0=0\}$ and $\{z_1=0\}$. In fact, by \autoref{lem:F_b-bundles}\autoref{F_b-bundles:3}, the transition maps of $\tau\pi'\colon X'\to C$ shows that $\pi'$ is a decomposable $\p^1$-bundle and by we have 
    \[X'\simeq \P(\Ol_{\p(\Ol_C\oplus\Ll)} \oplus \Ol_{\p(\Ol_C\oplus\Ll)}(\sigma + \tau^*(D))),\]
    where $D$ is a divisor such that $\Ol_C(D)\simeq \Ml$.
    
    By Blanchard's lemma, the ruled surface $X'_{\sigma}=\pi'^{-1}(\sigma)$ is $\Autz(X')$-invariant. Setting $z_0=0$ in the transition maps of the above $\F_1$-bundle, we obtain the transition maps of
    the $\p^1$-bundle $X_{\sigma'}$, which are of the form:
    \[
    \begin{array}{ccc}
        U_l\times \P^1 & \dashrightarrow & U_k \times \P^1 \\
        (x,[y_0:y_1]) & \longmapsto & (x,[y_0:\lambda_{kl}^{-1}(x)\mu_{kl}(x) y_1]).
    \end{array}
    \]
    In particular, $X'_{\sigma} \simeq \P(\Ol_{\sigma}\oplus (\Ll^\vee \otimes \Ml))$. Since $\Ll^\vee \otimes \Ml$ is a non-trivial line bundle of degree zero, this implies that $X'_{\sigma}$ is also a ruled surface admitting precisely two disjoint minimal sections (see \autoref{thm:dim_2}), which are given by the equations $\{y_0=0\}$ and $\{y_1=0\}$. Therefore, the curve $C'\subset X'$ with equations $\{y_1=z_0=0\}$ is $\Autz(X')$-invariant. 
    The blow-up of $C'$ followed by the contraction of the strict transform of $X'_{\sigma}$ yields an $\Autz(X')$-equivariant birational map $\phi\colon X' \dashrightarrow X_0$, locally given by
    \[
    \begin{array}{ccc}
        U\times \F_1 & \dashrightarrow & U \times \F_0 \\
        (x,[y_0:y_1\ ;z_0:z_1]) & \longmapsto & (x,[z_0y_0:y_1\ ;z_0: z_1]),
    \end{array}
    \]
    where $U$ is any trivializing open subset. Computing the transition maps of the resulting $\F_0$-bundle, we obtain that $X_0$ is isomorphic to the fibre product $\p(\Ol_C\oplus\Ll)\times_C\p(\Ol_C\oplus\Ml)$. The base locus of $\phi^{-1}$ is the curve with equation $\{y_0=z_0=0\}\subset X_0$, which is $\Autz(X_0)$-invariant by \autoref{thm:dim_2} and \autoref{lem:F_0-bundles}. Thus, we have the equality $\phi\Autz(X') \phi^{-1}= \Autz(X_0)$.
\end{proof}

\begin{remark}
    In the beginning of the proof of \autoref{lem:O+L+M} \eqref{O+L+M:2}, we blowup the section $C\times \{[0:0:1]\}$ corresponding to the line subbundle $\Ml$. Instead, we could also blowup the sections $C\times \{[1:0:0]\}$ or $C\times \{[0:1:0]\}$, and this would have given another $\F_0$-bundle $X'\to C$ such that $\Autz(X)$ is conjugate to $\Autz(X')$, where $X'$ is determined in \cite[Proposition 8.13]{Fong23}.
\end{remark}

\begin{proposition}\label{lem:O+L+M_conjugacy}
    Let $C$ be an elliptic curve and 
\[
X=\p(\Ol_C\oplus\Ll\oplus\Ml) \quad\text{with }\Ll,\Ml\in\Pic^0(C)\setminus\{0\} \ \text{and}\ \Ll\nsimeq\Ml
\]
and $\pi\colon X\to C$ the $\p^2$-bundle structure morphism.
    Then the connected algebraic subgroup $\Autz(X)$ is maximal in $\Bir(X)$ if and only if for every $(n,m) \in \mathbb{Z}^2$ coprime, the line bundle $\Ll^{\otimes n} \otimes \Ml^{\otimes m}$ is not trivial. In that case, $\Autz(X)$ is conjugate to $\Autz(X')$ if and only if there exists 
    \[
    \begin{pmatrix}
    \alpha & \beta \\ \gamma & \delta
    \end{pmatrix}\in \mathrm{SL}_2(\Z) 
    \]
    such that $X'$ is one of the following Mori fibre space:
    \begin{enumerate}
        \item\label{O+L+M_conjugacy:1} There exist $b\in \mathbb{Z}$ and a $\p^1$-bundle structure $\pi'\colon X'\to S$ such that 
        \[
        X' \simeq \P\left(\Ol_S \oplus (\Ol_S(b \sigma) \otimes \tau^*(\Ll^{\otimes \alpha} \otimes \Ml^{\otimes \beta}))\right),
        \]
        where $S:=\P(\Ol_C \oplus (\Ll^{\otimes \gamma} \otimes \Ml^{\otimes \delta}))$, and $\tau\colon S\to C$ denotes the $\p^1$-bundle structure morphism, and $\sigma$ is the section of $\tau$ corresponding to the line subbundle $\Ll^{\otimes\gamma}\oplus\Ml^{\otimes\delta}$.
        
        \item\label{O+L+M_conjugacy:2} There exists a $\p^2$-bundle structure $\pi'\colon X'\to C$ such that 
        \[X'\simeq \P(\Ol_C \oplus (\Ll^{\otimes \alpha} \otimes \Ml^{\otimes \beta}) \oplus (\Ll^{\otimes \gamma} \otimes \Ml^{\otimes \delta})).\]
    \end{enumerate}
\end{proposition}

\begin{proof}
Let $S_{\Ll}:=\P(\Ol_C \oplus \Ll)$ and $S_{\Ml}:=\P(\Ol_C \oplus \Ml)$, which are ruled surfaces over $C$.
By \autoref{lem:O+L+M}, $\Autz(X)$ is conjugate to $\Autz(S_{\Ll}\times_C S_{\Ml})$.
    By \cite[Proposition 8.13]{Fong23}, if there exists $(n,m)\in \mathbb{Z}^2$ coprime such that $\Ll^{\otimes n} \otimes \Ml^{\otimes m}$ is trivial, then $\Autz(S_{\Ll}\times_C S_{\Ml})$ is not maximal. In particular, $\Autz(X)$ is not maximal.
    
    Suppose that for any $(m,n)\in\Z^2$ coprime, the line bundle $\Ll^{\otimes n} \otimes \Ml^{\otimes m}$ is non-trivial. Let $X'\to S$ be a $\p^1$-bundle over a ruled surface $S$. 
    By \cite[Proposition 8.13]{Fong23}, $\Autz(X)$ is conjugate to $\Autz(X')$ if and only if $X'$ is as in \autoref{O+L+M_conjugacy:1}.

    For each $\P^1$-bundle $X'$ in case \autoref{O+L+M_conjugacy:1}, there exists an $\Autz(X')$-equivariant Sarkisov link whose target is a $\p^2$-bundle if and only if $b=\pm 1$. 
    Indeed, this link has to be of type I, and it induces a birational contraction $X'_{\k(S)}\to \p^2_{\k(S)}$, hence $X_{\k(S)}\simeq\F_{1,\k(S)}$.
    In that case, the link is given by the contraction $\eta$ of the surface spanned by the $(-1)$-sections along the fibres of the corresponding $\F_1$-bundle. By \autoref{lem:F_b-bundles}, its transition maps can be written as
    \[
    \begin{array}{ccc}
        U_l \times \F_1 & \dashrightarrow & U_k \times \F_1 \\
        (x,[y_0:y_1 \ ;z_0:z_1]) &\longmapsto & (x,[y_0:\lambda_{kl}(x)^\alpha \mu_{kl}(x)^\beta y_1 \ ;z_0:\lambda_{kl}(x)^\gamma \mu_{kl}(x)^\delta z_1]).
    \end{array}
    \]
    Moreover, on every trivializing open subset $U$, the birational morphism $\eta$ is given locally by
    \[
    \begin{array}{ccc}
        U\times \F_1 & \to & U  \times \P^2 \\
        (x,[y_0:y_1\ ;z_0:z_1]) & \longmapsto & (x,[y_0z_0:y_1:y_0z_1]).
    \end{array}
    \]
    Computing the transition maps of the resulting $\p^2$-bundle $X''\to C$, we get the $\p^2$-bundles of case \autoref{O+L+M_conjugacy:2}. By \autoref{lem:O+L+M}, the base locus of $\eta^{-1}$ is $\Autz(X'')$-invariant; hence, $\eta$ conjugates $\Autz(X')$ with $\Autz(X'')$. 
    Moreover, by \autoref{lem:hom_proj_bundles} and \autoref{lem:sum_line_bundles}, the group $\Autz(X)$ acts transitively on $C$ and $\mathbb{G}_m^2$ is a subgroup of $\Autz(X)_C$. Hence, by \autoref{lem:no_map_to_lower_degree}, there is no $\Autz(X)$-equivariant birational map from $\pi\colon X\to C$ to a Mori Del Pezzo fibration of degree $d<9$. 
\end{proof}

In \autoref{lem:O+L+M_conjugacy}, we assume that $\Ll$ and $\Ml$ are non-trivial and $\Ll \nsimeq \Ml$. If $\Ll$ and $\Ml$ are non-trivial and isomorphic, then $X\simeq\P(\Ol_C \oplus \Ol_C \oplus \Ll^\vee)$. We treat this case in \autoref{lem:degree_0_L_trivial} and \autoref{prop:E'=O+O}.

\begin{lemma}\label{lem:degree_0_L_trivial}
    Let $C$ be an elliptic curve, $\Ml\in \Pic^0(C)\setminus\{\Ol_C\}$ and $\El=\El' \oplus \Ml$, where $\El'=\Ol_C\oplus \Ol_C$ or $\El'=\El_{2,0}$. Let $X=\p(\El)$.
    Then the curve $C\times \{[0:0:1]\}$ corresponding to the line subbundle $\Ml\subset \El$ is $\Autz(X)$-invariant, and it is an orbit of minimal dimension. Moreover, the following hold:
    \begin{enumerate}
        \item If $\El'=\Ol_C\oplus \Ol_C$, then $C\times \{[0:0:1]\}$ is the only $\Autz(X)$-invariant curve. Moreover, $\Aut(X)_C \simeq \GL_2(\k)$.
        \item If $\El'=\El_{2,0}$, then there exists exactly one other $\Autz(X)$-orbit of dimension $1$, which is the curve $C\times \{[1:0:0]\}$, corresponding to the line subbundle $\Ol_C\subset \El_{2,0}$. Moreover, $\Autz(X)_C \simeq \G_m\rtimes \G_a$.
    \end{enumerate}
\end{lemma}

\begin{proof}
Note that for the rank-$3$ vector bundles $\El_{2,0} \oplus \Ml$ and $\Ol_C\oplus \Ol_C \oplus \Ml$, there exist trivializations such that the transition matrices are of the form
\[
\begin{pmatrix}
1 & \alpha_{kl} & 0 \\ 0 & 1 & 0 \\ 0 & 0 & \mu_{kl}
\end{pmatrix},
\]
where $\alpha_{kl}\in \Ol_C(U_{kl})$, and where $\mu_{kl}\in \Ol_C(U_{kl})^*$ are the cocycles of the line bundle $\Ml$. 
Recall the short exact sequence from \autoref{rmk:seq}
\[
0\to\Aut(\El)_C/\G_m\to \Aut(X)_C\to\Omega\to0.
\]
In our case we have $\Omega=0$. Indeed, if $\El'=\Ol_C\oplus\Ol_C$, this is \autoref{lem:sum_line_bundles}, and if $\El'=\El_{2,0}$, then this is \cite[Theorem 5]{Ati57}.

Therefore, $\Aut(X)_C \simeq \Aut(\El)_C/\mathbb{G}_m$.
Using $\Gamma(C,\Ml)=0$ and $\Ll=\Ol_C$ in the notation of \autoref{key_computation}, we obtain that $g=h=f=c=0$, that is, every element of $\Aut(\El)_C$ is on the $U_k \times \A^3$ given by
\[
	M_{k} = 
	\begin{pmatrix}
		a_k & b_k & 0 \\
		d_k & e_k & 0 \\
		0 & 0 & i
	\end{pmatrix}, 
	\]	
    for some $i\in \k^*$. This implies that $\Aut(\El)_C \simeq \Aut(\El')_C \times \mathbb{G}_m$; thus, $\Aut(X)_C\simeq \Aut(\El')_C$. Notice that the section $C\times \{[0:0:1]\}$ is $\Aut(X)_C$-invariant and that it corresponds to the line subbundle $\Ml$. Since $\Ml$ is invariant by translations of $C$, the section $C\times \{[0:0:1]\}$ is $\Autz(X)$-invariant. For the rest of the proof, we distinguish two cases according to whether $\El' = \El_{2,0}$ or $\Ol_C\oplus \Ol_C$.

    Assume that $\El'=\Ol_C\oplus \Ol_C$. Then $\Aut(\El')_C = \GL_2(\k)$ and any $\Autz(X)$-orbit different than $C\times \{[0:0:1]\}$ has dimension at least $2$. Thus, the $\Autz(X)$-invariant curve $C\times \{[0:0:1]\}$ is the only $\Autz(X)$-orbit of dimension $1$.

    Assume that $\El' = \El_{2,0}$. Then by \cite[\S1]{Maruyama}, we have the isomorphism $\Aut(\El_{2,0})_C\simeq \mathbb{G}_m \rtimes \mathbb{G}_a$ given by the upper triangular matrices
    \[
    \begin{pmatrix}
    a & b \\ 0 & a
    \end{pmatrix},
    \quad\text{where } a\in \k^*,\ b\in \k.
    \]
    Thus, there exists precisely one other $\Autz(X)_C$-invariant section, that is the curve $C\times \{[1:0:0]\}$. Its corresponds the the line subbundle $\Ol_C \subset \El$ given by the inclusion $\Ol_C \subset \El_{2,0}$. Moreover, the vector bundle $\El_{2,0}\oplus \Ml$ is homogeneous, and the line subbundle $\Ol_C\subset \El_{2,0}$ is invariant by translations of $C$. Hence, the sections $C\times [1:0:0]$ and $C\times \{[0:0:1]\}$ are precisely the $\Autz(X)$-orbits of dimension $1$ and the other $\Autz(X)$-orbits have dimension at least $2$.
\end{proof}

\begin{proposition}\label{prop:E'=O+O}
     Let $C$ be an elliptic curve, $\Ml\in \Pic^0(C)\setminus\{\Ol_C\}$ and $\El=\Ol_C\oplus \Ol_C \oplus \Ml$ and $X=\p(\El)$.
 Then $\Autz(X)$ is a maximal connected algebraic subgroup of $\Bir(X)$, which fits into a short exact sequence
 \[
 1 \to \GL_2(\k) \to \Autz(X) \to \Autz(C) \to 1.
 \]
 Moreover, $\Autz(X)$ is conjugate to $\Autz(X')$ if and only if $X'$ is one of the following Mori fibre space:
 \begin{enumerate}
    \item $X=X'$ is a $\p^2$-bundle over $C$.
    
    \item\label{prop:E'=O+O.2} There exists a $\p^1$-bundle structure $\pi'\colon X' \to C\times \p^1$ such that
    $$X' \simeq \P(\Ol_{C\times \P^1} \oplus \Ol_{C\times \P^1}(\sigma) \otimes \tau^*(\Ml)),$$
    where $\sigma$ denotes the class of a constant section of $\tau\colon C\times \p^1\to C$.
\end{enumerate}
In Case \autoref{prop:E'=O+O.2}, the morphism $\tau\pi'$ is an $\F_1$-bundle over $C$ and the contraction of the surface spanned by the $(-1)$-curves along the fibres of $\tau\pi'$ yields an $\Autz(X')$-equivariant birational morphism $X'\rightarrow X$ that conjugates $\Autz(X)$ and $\Autz(X')$.
\end{proposition}

\begin{proof}
    The short exact sequence follows from \autoref{lem:hom_proj_bundles} and \autoref{lem:degree_0_L_trivial}.
    Also, by \autoref{lem:degree_0_L_trivial}, $X$ has a unique $\Autz(X)$-orbit of dimension $\leq 1$, which is the section $\Sigma:=C\times\{[0:0:1]\}$ of $\pi\colon X\to C$ corresponding to $\Ml$.
    The blow-up $\eta\colon X'\to X$ of $\Sigma$ is locally given by
    \[
    \begin{array}{ccc}
        U \times \F_1 & \to & U \times \P^2 \\
        (x,[y_0:y_1\ ; z_0:z_1]) & \longmapsto & (x,[y_0z_0:y_0z_1:y_1]),
    \end{array}
    \]
    where $U$ is any trivializing open subset. The birational morphism $\eta$ contracts the surface spanned by the $(-1)$-sections along the fibres of the $\F_1$-bundle $X'\to C$. Since this surface is $\Autz(X')$-invariant, it follows that we have the equality $\eta^{-1} \Autz(X) \eta = \Autz(X')$. Moreover, computing the transition maps of the resulting $\F_1$-bundle, we obtain:
    \[
    \begin{array}{ccc}
        U_l \times \F_1 & \to & U_k \times \F_1 \\
        (x,[y_0:y_1\ ; z_0:z_1]) & \longmapsto & (x,[y_0:\mu_{kl}(x)y_1\ ;z_0:z_1]).
    \end{array}
    \]
    The projection onto $(x,[z_0:z_1])$ equips $X'$ with a structure of $\P^1$-bundle $\pi'\colon X'\to C\times \p^1$ and by \autoref{lem:F_b-bundles}, we have 
    \[
    X' \simeq \P(\Ol_{C\times \p^1} \oplus \Ol_{C\times \p^1}(\sigma) \otimes \tau^*(\Ml)),
    \]
    where $\tau\colon C\times 
    \p^1 \to C$ is the first projection and $\sigma$ is the class of a constant section.

    By \cite[Theorems A and C, (vii)]{Fong23} with $S=C\times \P^1$, then $\Autz(X')$ is relatively maximal and there is no $\Autz(X')$-equivariant Sarkisov link to another conic bundle over a surface. 
        Since $X$ is a decomposable $\p^2$-bundle, it follows that $\mathbb{G}_m\subset \Autz(X)_C$. Hence, there is also no $\Autz(X)$-equivariant link to a Mori Del Pezzo fibration of degree $<9$ by \autoref{lem:no_map_to_lower_degree}.  Therefore, the only $\Autz(X')$-equivariant starting from $X'$ is exactly the birational morphism $\eta$. 
        This shows that $\Autz(X')$ is conjugate to $\Autz(X)$ by $\eta$, and those are the only representatives of the conjugacy class. It follows that $\Autz(X)$ is a maximal connected algebraic subgroup of $\Bir(X)$.
\end{proof}

It remains to treat the case where $\El \simeq \El_{2,0} \oplus \Ml$, where $\Ml \in \Pic^0(C)$, in \autoref{lem:degree_0_L_trivial} (2). Let us first introduce a family of threefolds $\Al_{(S,b,\Ml^{\otimes n})}$, defined in \cite[Lemma 8.6]{Fong23}. Later, in \autoref{prop:E'=E_20}, we will see that $\Autz(\Al_{(S,b,\Ml^{\otimes n})})$ is maximal if $\Ml$ has infinite order. In this case, $\Autz(\Al_{(S,b,\Ml^{\otimes n})})$ is conjugate to $\Autz(\El_{2,0} \oplus \Ml)$.

\begin{example}\label{ex:A(S,b,M)}
    Let $C$ be an elliptic curve, let $\Ml\in \Pic^0(C)$ be non-trivial and let $S:=\P(\Ol_C \oplus \Ml)$. We define $$\Al_{(S,0,\Ol_C)}:=\Al_{2,0} \times_C S.$$ Recall from \autoref{thm:dim_2} that $\Al_{2,0}$ has a unique minimal section $\sigma_0$ of self-intersection zero, corresponding to the line subbundle $\Ol_C\subset \El_{2,0}$ and $S$ has exactly two minimal sections $\sigma$ and $\sigma'$ of self-intersection zero, corresponding respectively to the line subbundles $\Ol_C \subset \Ol_C \oplus \Ml$ and $\Ml\subset \Ol_C \oplus \Ml$.
    
    There exist trivializations of the $\F_0$-bundle $\Al_{(S,0,\Ol_C)}$ such that the transition maps are of the form
    \[
    \begin{array}{ccc}
        U_l \times \F_0 & \dashrightarrow & U_k\times \F_0 \\
        (x,[y_0:y_1\ ; z_0:z_1]) & \longmapsto & (x,[y_0:y_1+ \alpha_{kl}(x)y_0\ ; z_0: \mu_{kl}(x)z_1]),
    \end{array}
    \]
    where $\begin{pmatrix} 1 & 0 \\ \alpha_{kl} & 1 \end{pmatrix} \in \GL_2(\Ol_C(U_{kl}))$ are the transition matrices of $\Al_{2,0}$ and $\mu_{kl}\in \Ol_C(U_{kl})^*$ are the cocycles of the line bundle $\Ml$. 
    Under this choice of trivializations, we have 
    \[
    \sigma_0=\{y_0=0\} \subset \Al_{2,0},\quad\sigma=\{z_0=0\}\subset S,\quad \sigma'=\{z_1=0\}\subset S.
    \]
    \autoref{lem:F_0-bundles} and the uniqueness of the sections $\sigma_0,\sigma,\sigma'$ implies that there are precisely two $\Autz(\Al_{(S,0,\Ol_C)})$-invariant curves, given by the equations $l_{00}=\{y_0=z_0=0\}$ and $l_{01}=\{y_0=z_1=0\}$. 
    
    We define the $\F_1$-bundle $\Al_{(S,1,\Ol_C)}$ as follows. The blowup of $l_{00}$ followed by the contraction of the strict transform of the surface $\{z_0=0\}$ yields an $\Autz(\Al_{(S,0,\Ol_C)})$-equivariant birational map $\Al_{(S,0,\Ol_C)}\rat \Al_{(S,1,\Ol_C)}$ given on every chart $U\times\F_0$ by
    \[
    U\times\F_0\rat U\times\F_1,\quad (x,[y_0:y_1;z_0:z_1])\mapsto(x,[y_0:y_1z_0;z_0:z_1]).
    \]
    to the $\F_1$-bundle $\Al_{(S,1,\Ol_C)}$. 
    
    Similarly, the $\F_1$-bundle $\Al_{(S,1,\Ml)}$ as follows. The blowup of $l_{01}$ followed by the contraction of the strict transform of the surface $\{z_1=0\}$ yields an $\Autz(\Al_{(S,0,\Ol_C)})$-equivariant birational map $\Al_{(S,0,\Ol_C)}\rat \Al_{(S,1,\Ml)}$ given on every chart $U_l\times\F_0$ by
    \[
    U_l\times\F_0\rat U_l\times\F_1,\quad (x,[y_0:y_1;z_0:z_1])\mapsto(x,[y_0:y_1z_1;z_0:z_1]).
    \]
    to the $\F_1$-bundle $\Al_{(S,1,\Ml)}$.

    For every $b\geq 0$ and $n\in \{0,\ldots, b\}$, we construct the $\F_b$-bundle $\Al_{(S,b,\Ml^{\otimes n})}$ by induction as follows. 
    The threefold $X= \Al_{(S,b,\Ml^{\otimes n})}$ is equipped with a structure of $\P^1$-bundle $\pi\colon X \to S$. By Blanchard's lemma (\autoref{lem:blanchard}), the surfaces $\pi^{-1}(\sigma)$ and $\pi^{-1}(\sigma')$ are $\Autz(X)$-invariant. The intersection of those two surfaces with the surface spanned by the $(-b)$-sections along the fibres of the $\F_b$-bundle yields two $\Autz(X)$-invariant curves $l_{00}^X$ and $l_{01}^X$. The blow-up of $l_{00}^X$ (resp. $l_{01}^X$) followed by the contraction of the strict transform of the surface $\pi^{-1}(\sigma)$ (resp. $\pi^{-1}(\sigma')$) gives rise to the $\F_{b+1}$-bundle $\Al_{(S,b+1,\Ml^{\otimes n})}$ (resp. $\Al_{(S,b+1,\Ml^{\otimes n+1})}$).

    Using the transition maps of $\Al_{(S,0,\Ol_C)}$ and the above construction, we obtain that $\Al_{(S,b,\Ml^{\otimes n})}$ has transition maps of the form
    \[
    \begin{array}{ccc}
        U_l \times \F_b & \dashrightarrow & U_k\times \F_b \\
        (x,[y_0:y_1\ ; z_0:z_1]) & \longmapsto & (x,[y_0:\mu_{kl}(x)^ny_1+ \mu_{kl}(x)^n\alpha_{kl}(x) z_0^{b-n}z_1^n y_0\ ; z_0: \mu_{kl}(x)z_1]).
    \end{array}
    \]
    By \cite[Propositions 8.8 and 8.9]{Fong23}, the $\P^1$-bundles $\Al_{(S,b,\Ml^{\otimes n})} \to S$ constructed above are indecomposable, with invariants $(S,b,\Ml^{\otimes n})$, and their automorphism groups $\Autz(\Al_{(S,b,\Ml^{\otimes n})})$ are all conjugate if and only if $\Ml$ has infinite order.
\end{example}

\begin{proposition}\label{prop:E'=E_20}
    Let $C$ be an elliptic curve, $\Ml\in \Pic^0(C)\setminus\{\Ol_C\}$ and $\El=\El_{2,0} \oplus \Ml$ and $X=\p(\El)$.
    Then $\Autz(X)$ is a maximal connected algebraic subgroup of $\Bir(X)$ if and only if $\Ml$ has infinite order. In that case, $\Autz(X)$ fits into a short exact sequence
    \[
    1\to \G_m \rtimes \G_a \to \Autz(X) \to \Autz(C) \to 1,
    \]
    and $\Autz(X)$ is conjugate to $\Autz(X')$ if and only if $X'$ is one of the following Mori fibre space:
    \begin{enumerate}
        \item\label{E'=E_20:1} There exist $b\in \mathbb{Z}$ and a $\p^1$-bundle structure $\pi'\colon X'\to \Al_{2,0}$
        such that 
        \begin{align*}
        X'&\simeq \P(\Ol_{\Al_{2,0}} \oplus \Ol_{\Al_{2,0}}(b\sigma)\otimes \tau^*(\Ml)),
        \end{align*}
        where $\sigma$ denotes the class of the unique minimal section of the ruled surface $\Al_{2,0}$.
        \item\label{E'=E_20:2} There exist $b\geq 0$, $n\in \{0,\ldots,b\}$, and a $\p^1$-bundle structure $\pi'\colon X'\to \P(\Ol_C \oplus \Ml)$, such that
        \[
           X'\simeq\Al_{(\P(\Ol_C \oplus \Ml),b,\Ml^{\otimes n})} 
        \]
         constructed in \autoref{ex:A(S,b,M)}.
        \item\label{E'=E_20:3} There exists a structure of $\p^2$-bundle $\pi'\colon X'\to C$ such that $X'\simeq X$ or $$X'\simeq \P(\El_{2,0}\oplus \Ml^\vee) = :X^\vee.$$ 
    \end{enumerate}    
\end{proposition}

\begin{proof}
The short exact sequence follows from \autoref{lem:hom_proj_bundles} and \autoref{lem:degree_0_L_trivial}. 
We will show that all $\Autz(X)$-equivariant birational maps from $ X$ to another Mori fibre space are summarized in the following commutative diagram:
{\footnotesize
\begin{equation}
\begin{tikzcd}[column sep=0.1em,row sep = 4em] 
 & \P(\Ol_{\Al_{2,0}} \oplus \Ol_{\Al_{2,0}}(b\sigma)\otimes \tau^*(\Ml)) \arrow[rr,dashed, "(ii)"]\arrow[ld,dashed,"(i)" swap] \arrow[rd] \arrow[d,dashed, "(iv)"] & & \Al_{2,0}\times_C \p(\Ol_C \oplus \Ml) \arrow[rr,dashed,"(ii)"] \arrow[rd] \arrow[ld]&& \Al_{(\P(\Ol_C \oplus \Ml),b,\Ml^{\otimes n})} \arrow[rd,dashed,"(iii)"] \arrow[ld] \arrow[d,dashed, ,"(iii)"] & \\
X \arrow[rrrd,"\pi" swap] & X^\vee \arrow[rrd] & \Al_{2,0} \arrow[rd] & & \p(\Ol_C \oplus \Ml) \arrow[ld] & X \arrow[lld] & X^\vee\arrow[llld] \\
&&& C &&&.
\end{tikzcd} 
\tag*{\llap{\quad(\(\clubsuit\))}}
\label{diagram}
\end{equation}
}
All morphisms with targets to $\Al_{2,0}$, $\p(\Ol_C\oplus \Ml)$, and $C$ are structure morphisms of projective bundles. All dashed arrows are $\Autz(X)$-equivariant birational maps. The numbers $(i)$, $(ii)$, $(iii)$, and $(iv)$ indicate where these maps are constructed in the proof below. 

Since $\El=\El_{2,0}\oplus\Ml$ is decomposable, we have $\mathbb{G}_m\subset \Autz(X)_C$. 
Hence, by \autoref{lem:no_map_to_lower_degree}, there is no $\Autz(X)$-equivariant birational map from $X$ to a Mori Del Pezzo fibration of degree $<9$. Now, let's start.

(i) First, recall that by \autoref{lem:degree_0_L_trivial}, the $\p^2$-bundle $\pi\colon X\to C$ admits precisely two $\Autz(X)$-invariant sections, which are $C\times \{[0:0:1]\}$ and $C\times \{[1:0:0]\}$. Let's leave the section $C\times \{[1:0:0]\}$ for later (see (iii), case (c)).

    We start as in the beginning of the proof of \autoref{prop:E'=O+O}, by blowing up the $\Autz(X)$-invariant section $\Sigma:=C\times\{[0:0:1]\}$.
    The blow-up $\eta\colon X'\to X$ of $\Sigma$ is locally given by
    \[
    \begin{array}{ccc}
        U \times \F_1 & \to & U \times \P^2 \\
        (x,[y_0:y_1\ ; z_0:z_1]) & \longmapsto & (x,[y_0z_0:y_0z_1:y_1]),
    \end{array}
    \]
    where $U$ is any trivializing open subset. The birational morphism $\eta$ contracts the surface spanned by the $(-1)$-sections along the fibres of the $\F_1$-bundle $X'\to C$. Since this surface is $\Autz(X')$-invariant, it follows that we have the equality $\eta^{-1} \Autz(X) \eta = \Autz(X')$. Moreover, computing the transition maps of the resulting $\F_1$-bundle, we obtain:
    \[
    \begin{array}{ccc}
        U_l \times \F_1 & \dashrightarrow & U_k \times \F_1 \\
        (x,[y_0:y_1\ ; z_0:z_1]) & \longmapsto & (x,[y_0:\mu_{kl}(x)y_1\ ;z_0+\alpha_{kl}(x)z_1:z_1]).
    \end{array}
    \]
    The projection onto $(x,[z_0:z_1])$ equips $X'$ with a structure of $\P^1$-bundle $\pi'\colon X'\to \Al_{2,0}$ and by \autoref{lem:F_b-bundles}, we have 
    \[
    X' \simeq \P(\Ol_{\Al_{2,0}} \oplus (\Ol_{\Al_{2,0}}(\sigma) \otimes \tau^*(\Ml))),
    \]
    where $\tau\colon \Al_{2,0} \to C$ is the structure morphism and $\sigma$ is the minimal section of $\Al_{2,0}$, which is unique (see \autoref{thm:dim_2}).
    
    For every $b\in \mathbb{Z}$, there exist $\Autz(X')$-equivariant birational maps 
    \[
    X'\rat\p(\Ol_{\Al_{2,0}}\oplus(\Ol_{\Al_{2,0}}(b\sigma)\otimes \tau^*(\Ml)))
    \]
    over $\Al_{2,0}$. Together with the birational map $\eta$, this gives us the $\Autz(X)$-equivariant birational map 
    \[
    \p(\Ol_{\Al_{2,0}}\oplus(\Ol_{\Al_{2,0}}(b\sigma)\otimes \tau^*(\Ml))) \overset{(i)}{\rat}  X
    \]
    in Diagram~(\(\clubsuit\)). This also gives \autoref{E'=E_20:1}.
    
    \smallskip
    
    (ii) At this step, the only way to obtain new $\Autz(X)$-equivariant birational maps is to consider the cases where $b\in \{-1,0\}$ in \autoref{E'=E_20:1}, that means, the $\p^1$-bundles 
    \begin{enumerate}
        \item[(a)] $\p(\Ol_{\Al_{2,0}}\oplus \tau^*(\Ml))$,
        \item[(b)] $\p(\Ol_{\Al_{2,0}}\oplus(\Ol_{\Al_{2,0}}(-\sigma)\otimes \tau^*(\Ml)))$.
    \end{enumerate}
     In case (b), we have an $\F_1$-bundle over $C$ and we can produce an $\Autz(X)$-equivariant Sarkisov link
     by contracting the invariant surface spanned by the $(-1)$-sections. Let's leave this case for later (see (iv)), and now we focus on case (a). In fact, $\p(\Ol_{\Al_{2,0}}\oplus \tau^*(\Ml))$ is isomorphic as $\p^1$-bundle to the fibre product $\Al_{2,0}\times_C \p(\Ol_C\oplus \Ml)$, which is equipped with another structure of $\p^1$-bundle, namely the projection onto $\p(\Ol_C\oplus \Ml)$.
    
    As explained in \autoref{ex:A(S,b,M)}, there exist $\Autz(\Al_{2,0}\times_C \P(\Ol_C \oplus \Ml))$-equivariant birational maps 
    \[
    \Al_{2,0}\times_C \P(\Ol_C \oplus \Ml) \dashrightarrow \Al_{(\P(\Ol_C \oplus \Ml),\ b,\ \Ml^{\otimes n})}
    \]
    over $\p(\Ol_C\oplus\Ml)$ for any $b\geq0$ and $0\leq n\leq b$. This gives us the top right equivariant birational map in Diagram~(\(\clubsuit\)) and \autoref{E'=E_20:2}.
    
    (iii) To obtain new $\Autz(X)$-equivariant birational maps, the only way is to consider the case $b=1$ in \autoref{E'=E_20:2}, i.e. the $\p^1$-bundles 
    \[
    \Al_{(\P(\Ol_C \oplus \Ml),\ 1,\ \Ml^{\otimes n})} 
    \]
    and contracting the invariant surface $S_1$ spanned by the $(-1)$-sections along the fibres. Since $n\in \{0,1\}$, we have two cases:
    \begin{enumerate}
        \item[(c)] Let $n=0$. By \autoref{ex:A(S,b,M)}, the $\F_1$-bundle $\Al_{(\p(\Ol_C\oplus \Ml),1,\Ol_C)}\to C$ has transition maps of the form
        \[
        \begin{array}{ccc}
        U_l \times \F_1 & \dashrightarrow & U_k\times \F_1 \\
        (x,[y_0:y_1\ ; z_0:z_1]) & \longmapsto & (x,[y_0:y_1+ \alpha_{kl}(x) z_0 y_0\ ; z_0: \mu_{kl}(x)z_1]).
        \end{array}
        \]
        The contraction 
        \[
        \Al_{(\p(\Ol_C\oplus \Ml),1,\Ol_C)}\to Y,\quad (x,[y_0:y_1\ ; z_0:z_1])\mapsto (x,[y_1 : y_0 z_0 : y_0 z_1])
        \]
        is an $\Autz(\Al_{(\p(\Ol_C\oplus \Ml,1),\Ol_C)})$-equivariant birational morphism onto a $\p^2$-bundle $Y$ over $C$. Computing the transition matrices of $Y$, we see that $Y\simeq X$, and the surface $S_1$ is contracted to the section $C\times \{[1:0:0]\}$ that we have left in the beginning of (i) of this proof.

        \item[(d)] Let $n=1$. By \autoref{ex:A(S,b,M)}, the $\F_1$-bundle $\Al_{(\p(\Ol_C\oplus \Ml),1,\Ml)}\to C$ has transition maps of the form
        \[
        \begin{array}{ccc}
        U_l \times \F_1 & \dashrightarrow & U_k\times \F_1 \\
        (x,[y_0:y_1\ ; z_0:z_1]) & \longmapsto & (x,[y_0:\mu_{kl}(x)y_1+ \mu_{kl}(x)\alpha_{kl}(x) z_1 y_0\ ; z_0: \mu_{kl}(x)z_1]).
        \end{array}
        \]
        The contraction 
        \[
         \Al_{(\p(\Ol_C\oplus \Ml),1,\Ml)}\to Z,\quad (x,[y_0:y_1\ ; z_0:z_1])\mapsto (x,[y_1 : y_0 z_1 : y_0 z_0])
        \]
        is an $\Autz(\Al_{(\p(\Ol_C\oplus \Ml),1,\Ml)})$-equivariant birational morphism onto a $\p^2$-bundle $Z$ over $C$, with transition matrices
        \[
        \begin{pmatrix}
        1 & \alpha_{kl} & 0 \\ 0 & 1 & 0 \\ 0&0&\mu_{kl}^{-1}
        \end{pmatrix}.
        \]
        This implies that the resulting $\P^2$-bundle is isomorphic to $Z\simeq \P(\El_{2,0}\oplus \Ml^\vee)$. Since $\El_{2,0}$ is self-dual, it follows that $Z\simeq X^\vee$ and this gives \autoref{E'=E_20:3}. 
        \end{enumerate}
    In cases (a) and (b) above, we have constructed the $\Autz(\Al_{(\P(\Ol_C \oplus \Ml),b,\Ml^{\otimes n})})$-equivariant birational maps
    \begin{align*}
    \Al_{(\P(\Ol_C \oplus \Ml),b,\Ml^{\otimes n})} & \rat  X, \\ 
    \Al_{(\P(\Ol_C \oplus \Ml),b,\Ml^{\otimes n})} & \rat  X^\vee, 
    \end{align*}
    in Diagram~(\(\clubsuit\)). 
         
   \smallskip
     
    (iv) By \autoref{lem:degree_0_L_trivial}, $X^\vee$ has precisely two $\Autz(X^\vee)$-invariant sections. In (iii), case (d), notice that the surface $S_1$ is contracted to the section $C\times \{[1:0:0]\}$. Then the other invariant section is $C\times \{[0:0:1]\}$.
    Its blow-up is an $\Autz(X^\vee)$-equivariant birational morphism $\theta\colon X'''\to X^{\vee}$, locally given by
    \[
    \begin{array}{ccc}
    (x,[y_0:y_1\ ;  z_0:z_1]) & \mapsto & (x,[y_0z_0:y_0z_1:y_1]).
    \end{array}
    \]
    Hence, $X'''$ carries the structure of an $\F_1$-bundle over $C$ with transition maps
    \[
        \begin{array}{ccc}
        U_l \times \F_1 & \dashrightarrow & U_k\times \F_1 \\
        (x,[y_0:y_1\ ; z_0:z_1]) & \longmapsto & (x,[y_0:\mu_{kl}^{-1}(x)y_1\ ; z_0 + \alpha_{kl}(x)z_1: z_1]).
        \end{array}
    \]
    The projection onto $(x,[z_0:z_1])$ endows $X'''$ with a structure of $\p^1$-bundle, which is isomorphic to 
    \[
    \p(\Ol_{\Al_{2,0}}\oplus(\Ol_{\Al_{2,0}}(\sigma)\otimes \tau^*(\Ml^\vee))),
    \]
    where $\sigma$ is the unique minimal section of $\Al_{2,0}$. Notice that this $\P^1$-bundle is also isomorphic to 
    \[
    \p(\Ol_{\Al_{2,0}}\oplus(\Ol_{\Al_{2,0}}(-\sigma)\otimes \tau^*(\Ml))),
    \]
    which was left earlier in the proof, in (ii) case (b).
    This gives the $\Autz(X)$-equivariant birational map $\p(\Ol_{\Al_{2,0}}\oplus(\Ol_{\Al_{2,0}}(b\sigma)\otimes \tau^*(\Ml))) \rat X^\vee$ in Diagram ~(\(\clubsuit\)).

    At each step, we studied all the $\Autz(X)$-equivariant Sarkisov links. This proves the Proposition.
\end{proof}

\begin{lemma}\label{lem:E_20+O}
    Let $C$ be an elliptic curve.
    Let $X=\P(\El_{2,0} \oplus \Ol_C)$. Then $\Autz(X)$ is not a maximal connected algebraic subgroup of $\Bir(X)$.
\end{lemma}

\begin{proof}
    For the rank-$3$ vector bundle $\El_{2,0} \oplus \Ol_C$, there exist trivializations such that the transition matrices are of the form
    \[
        \begin{pmatrix}
        1 & \alpha_{kl} & 0 \\ 0 & 1 & 0 \\ 0 & 0 & 1
        \end{pmatrix},
    \]
    where $\alpha_{kl}\in \Ol_C(U_{kl})$. From \autoref{key_computation}, we get the equality
    \[
        \begin{pmatrix}
            1 & \alpha_{kl} \\ 0  & 1
        \end{pmatrix}
        \cdot
        \begin{pmatrix}
            e_k \\ d_k
        \end{pmatrix}
        =
        \begin{pmatrix}
            e_l \\ d_l
        \end{pmatrix} 
	\]
        The morphisms $(e_k,d_k)\colon U_k\to \A^2$ and $(e_l,d_l)\colon U_l\to \A^2$
    yield a global section $(e,d)$ of $\El_{2,0}$.
    Next, we show that $d=0$. Indeed, the short exact sequence defining $\El_{2,0}$ in \autoref{thm:Atiyah} yields a long exact sequence of cohomology
    \[
    0 \to \k \to H^0(C,\El_{2,0}) \to \k \to \k \to H^1(C,\El_{2,0}) \to \k \to 0
    \]
    Moreover, $H^0(C,\El_{2,0})$ has dimension $1$ by \cite[Lemma 15]{Ati57}. Hence, we have the isomorphism
    \begin{align*}
    \k = H^0(C,\Ol_C)&\longrightarrow  H^0(C,\El_{2,0})  \\
    s & \longmapsto  \begin{pmatrix}
        s \\ 0
        \end{pmatrix}
    \end{align*}
    This shows that $d=0$.
    Using again \autoref{key_computation}, we have the equality
    \[
        \begin{pmatrix}
            1 & \alpha_{kl} \\ 0  & 1
        \end{pmatrix}
        \cdot
        \begin{pmatrix}
            h_k \\ g_k
        \end{pmatrix}
        =
        \begin{pmatrix}
            h_l \\ g_l
        \end{pmatrix},
	\]
    and this yields a global section $(h,g)$. By the same argument as before, we have $g=0$. 
    
    Since $d=g=0$, the curve $C':=C\times \{[1:0:0]\}\subset X$ is $\Autz(X)_C$-invariant. 
    The vector bundle $\El_{2,0} \oplus \Ol_C$ is homogeneous, as $\El_{2,0}$ and $\Ol_C$ are both homogeneous (see \autoref{rmk:Atiyah}). Moreover, by \autoref{rmk:Atiyah}, the line subbundle $\Ol_C \subset \El_{2,0}$ is unique and hence invariant by translations. Therefore, $C'$ is $\Autz(X)$-invariant.

    The blowup of $C'$ yields a birational morphism $\eta\colon X'\to X$, locally given by
    \[
    \begin{array}{ccc}
    U \times \F_1 & \to & U \times \P^2 \\ 
    (x,[y_0 :y_1 \ ; z_0 : z_1]) &\to & (x,[y_1 : y_0z_0 : y_0z_1]).
    \end{array}
    \]
    where $U\subset C$ is any trivializing open subset. Thus, the transition maps of $X'$ are given by 
    \[
    \begin{array}{ccc}
    U_l \times \F_1 & \to & U_k \times \F_1 \\ 
    (x,[y_0 :y_1 \ ; z_0 : z_1]) &\to & (x,[y_0 :y_1 + \alpha_{kl}(x)z_0y_0 \ ; z_0 : z_1]).
    \end{array}
    \]
  The projection onto $(x,[z_0:z_1])$ equips $X'$ with a structure of $\P^1$-bundle over $C\times \p^1$. The restriction of $X'$ over the constant section $C\times [1:0]\subset C\times \p^1$ yields a $\p^1$-bundle with transition matrices
    \[
    \begin{pmatrix}
        1 & \alpha_{kl} \\ 0 & 1
    \end{pmatrix},
    \]
    hence it is isomorphic to $\Al_{2,0}$. A fortiori, $X'$ is an indecomposable $\p^1$-bundle over $C\times \p^1$. Indeed, if $X'$ was decomposable, it would have two disjoint sections, which would induce two disjoint sections in each $\Al_{2,0}$, which is impossible.
    
    By \cite[Proposition 5.9]{Fong23}, $\Autz(X')$ is not maximal. Since $C'$ is $\Autz(X)$-invariant, we have $\eta^{-1} \Autz(X) \eta \subseteq \Autz(X')$, and it follows that $\Autz(X)$ is also not maximal.
\end{proof}

\subsection{Automorphisms of indecomposable projective bundles}\label{section: indecomposable}

In this section, $C$ is again an elliptic curve, and we focus on the automorphism groups of the Atiyah $\p^2$-bundles $\Al_{3,i}$, where $i\in \{0,1,2\}$. Below, we treat the two cases $i=0$ and $i\in \{1,2\}$ separately.

\subsubsection{The stable projective bundles}\label{ss:stable_proj}

For $i\in \{1,2\}$, the vector bundles $\El_{3,i}$ are stable, see for instance \cite[Lemma 1]{Hein_Ploog}. In this section, we will not use the notion of the \emph{slope} of a vector bundle nor the definition of a \emph{stable bundle} (see \autoref{def:slope}). We will only use the property that stable bundles are \emph{simple}, i.e. $\Aut(\El_{3,i})_C=\G_m$ (see e.g. \cite[Corollary 1.2.8]{Huybrechts_Lehn}).

\medskip

We write $X=\p(\El_{3,i})$. From the short exact sequence \autoref{SES:Grothendieck}, we get that $\Aut(X)_C\simeq \Omega$. Recall that $\Omega$ is the subgroup of the $3$-torsion of $\Pic^0(C)$ consisting of line bundles $\mathcal{N}$ such that 
\[
\El_{3,i} \otimes \mathcal{N} \simeq \El_{3,i}.
\]
Then by \cite[Corollary(iii) to Theorem 7]{Ati57}, it follows that $\Omega \simeq (\mathbb{Z}/3\mathbb{Z})^2$. 

Our projective bundle $\pi\colon X\to C$ is irreducible, i.e., there exists an irreducible projective representation $G\to \PGL_3(\k)$ such that \[X \simeq \Autz(X)\times^G \p^2,\] 
and via this isomorphism, the $\P^2$-bundle structure morphism $\pi\colon X\to C$ is identified with the projection onto $\Autz(X)/\Autz(X)_C\simeq C$ (cf. \cite[Theorem 2.1, Proposition 3.5]{Brion_homogeneous_projective_bundles}).
By \cite[Proposition 3.1]{Brion_homogeneous_projective_bundles}, the finite group $\Autz(X)_C$ has order $9$. It follows that 
$\Autz(X)_C=\Aut(X)_C\simeq(\mathbb{Z}/3\mathbb{Z})^2$.

\begin{lemma}\label{lem:decomposable_preserves_constant_section}
    Let $E$ be an elliptic curve, $G\subset E$ a finite subgroup and
   $X=E\times^G \P^2$ and assume that $\pi \colon X \to C$ is decomposable. Then $G$ preserves a constant section in $E\times \P^2$.
\end{lemma}

\begin{proof}
    The projection $p_E\colon E\times\p^2\to E$ and the quotient $q\colon E\times\p^2\to X$ by the diagonal action of $G$ induce a cartesian square
    \[
    \begin{tikzcd}
        E\times \P^2 \ar[r,"q"]\ar[d,"p_E"] & X \ar[d,"\pi"] \\
        E\ar[r,"i"] & C,
    \end{tikzcd}
    \]
   where $i$ is the quotient by $G$, which is an isogeny of elliptic curves. 
    
    Since $\pi$ is decomposable, there exists a section $\sigma$ of $\pi$ and a $\P^1$-subbundle $S$ of $X$, such that $\sigma$ and $S$ are disjoint. Their pullbacks by $i$ yield respectively a line bundle $\Ll$ and a rank-$2$ vector bundle $\El$, such that $\P(\Ll \oplus \El) \simeq E\times \P^2$. In particular, there exists a line bundle $\Ml$ such that $\Ml \otimes (\Ll \oplus \El) \simeq \Ol_C \oplus \Ol_C \oplus \Ol_C$. By Krull-Schmidt theorem, this decomposition is unique; hence, the pullback section $i^*\sigma$ is a $G$-invariant constant section of $p_E\colon E\times \P^2 \to E$.
\end{proof}

The following example gives an explicit construction of $\Al_{3,1}$ and $\Al_{3,2}$, via a projective representation of $(\mathbb{Z}/3\mathbb{Z})^2$.

\begin{example}\label{ex:indec_V9}
    Let $G= (\mathbb{Z}/3\mathbb{Z})^2$, that we see as the $3$-torsion subgroup of an elliptic curve $E$. Consider the projective representation of $G$ in $\PGL_3$, by sending the generators to the matrices
    \[
	\begin{pmatrix}
		1 & 0 & 0 \\
		0 & j & 0 \\
		0 & 0 & j^2
	\end{pmatrix} \text{ and }
	\begin{pmatrix}
		0 & 1 & 0 \\
		0 & 0 & 1 \\
		1 & 0 & 0
	\end{pmatrix},
	\]	    
    where $j$ is a $3$-rd primitive root of unity. This induces a diagonal action of $G$ on $E\times \P^2$, which does not preserve any constant section. 
    By \autoref{lem:decomposable_preserves_constant_section}, the corresponding quotient $X = E\times^G \P^2$ is an indecomposable $\P^2$-bundle over the elliptic curve $C := E/G$.

    Alternatively, we can consider the dual representation of $G$, by sending the generators to the matrices
    \[
	\begin{pmatrix}
		1 & 0 & 0 \\
		0 & j^2 & 0 \\
		0 & 0 & j
	\end{pmatrix} \text{ and }
	\begin{pmatrix}
		0 & 1 & 0 \\
		0 & 0 & 1 \\
		1 & 0 & 0
	\end{pmatrix}.
	\]	    
    As above, the associated quotient $X':=E\times^G\p^2$ is an indecomposable $\P^2$-bundle over $C$. 
    Since the two above projective representations are dual to each other, but not self-dual, $X$ and $X'$ are isomorphic to $\Al_{3,1}$ and $\Al_{3,2}$, which are also dual to each other (see \cite[Theorems 5,6, and 7]{Ati57}).
\end{example}

In \autoref{prop:A31 and A32}, we prove that $\Autz(\Al_{3,1})$ and $\Autz(\Al_{3,2})$ are maximal. In fact, each of them is the unique representative of its conjugacy class. 

\begin{definition}\label{def:superrigid}
    Let $\pi\colon X\to C$ be a Mori fibre space. The pair $(\pi,\Autz(X))$ is \emph{superrigid} if every $\Autz(X)$-equivariant Sarkisov link starting from $\pi\colon X\to C$ is an isomorphism.
\end{definition}

\begin{lemma}\label{lem:4:1_maximal}
    Let $C$ be a smooth projective curve of genus $g(C)\geq 1$.
    Let $\pi\colon X\to C$ be a $\P^2$-bundle such that every $\Autz(X)$-orbit is a curve whose projection onto $C$ is $n$-to-$1$ for $n\geq9$. Then $\Autz(X)$ is a maximal connected algebraic subgroup of $\Bir(X)$ and the pair $(X,\pi)$ is superrigid.
\end{lemma}

\begin{proof}
    By \autoref{lem: links III and IV},
    there is no $\Autz(X)$-equivariant Sarkisov diagram of type III or IV starting from $X$.
   The remaining claim follows from \autoref{lem:links I} and \autoref{lem:type II} and the fact that the blow-up of $\p^2_{\k(C)}$ in $n\geq9$ points is not a Del Pezzo surface.
\end{proof}

\begin{proposition}\label{prop:A31 and A32}
    Let $X$ be the projective bundle $\Al_{3,1}$ or $\Al_{3,2}$. Then $\Autz(X)$ is a maximal connected algebraic subgroup of $\Bir(X)$ and the pair $(X,\pi)$ is superrigid. Moreover, $\Autz(X)$ is an elliptic curve sitting in a short exact sequence
    \[
    1 \to (\mathbb{Z}/3\mathbb{Z})^2 \to \Autz(X) \to \Autz(C) \to 1
    \]
    and $X$ is isomorphic to \[\Autz(X)\times^{(\mathbb{Z}/3\mathbb{Z})^2} \p^2,\] where $(\mathbb{Z}/3\mathbb{Z})^2$ is seen as the three-torsion subgroup of $\Autz(X)$ acting on $\p^2$ via the projective representations given in \autoref{ex:indec_V9}.
\end{proposition}

\begin{proof}
    In \autoref{ex:indec_V9}, we have seen that $X\simeq G\times^{(\mathbb{Z}/3\mathbb{Z})^2} \p^2$, where $G=\Autz(X)$ is an elliptic curve. It follows that every $G$-orbit is a $9$-to-$1$ curve onto $C$. Hence, by \autoref{lem:4:1_maximal}, $G$ is a maximal connected algebraic subgroup of $\Bir(X)$. Moreover, $(\mathbb{Z}/3\mathbb{Z})^2$ acts by translations on the first factor $\Autz(X)\times^{(\mathbb{Z}/3\mathbb{Z})^2} \p^2$, without changing the image in the quotient $\Autz(X) /(\mathbb{Z}/3\mathbb{Z})^2 \simeq C$. Hence, $\Autz(X)_C \simeq (\mathbb{Z}/3\mathbb{Z})^2$. Since $\pi_*$ is surjective by \autoref{lem:hom_proj_bundles}, we obtain the stated short exact sequence.
\end{proof}

\subsubsection{The projectivization of the unipotent bundle}

It remains to study the automorphism group of the projective bundle $\P(\El_{3,0})$. The following remark is elementary and well-known for experts in vector bundles, see e.g. \cite[Section 4.2]{Brion_via_representation} for details.

\begin{remark} 
    Recall that a vector bundle $\Fl$ is \emph{unipotent} if there exists a filtration 
    \[
    0 = \Fl_0 \subset \Fl_1 \subset \dots \subset \Fl_r = \Fl,
    \]
    such that $\Fl_i/\Fl_{i-1} \simeq \Ol_C$ for every $i=1,\dots,r$. Using that $\El_{3,0}$ is self-dual (see \cite[Corollary 1]{Ati57}), and the short exact sequences in \autoref{thm:Atiyah}, we obtain a filtration
    \[
    0 \subset \Ol_C \subset \El_{2,0} \subset \El_{3,0},
    \]
    such that each successive quotient is trivial, i.e., $\El_{3,0}$ is unipotent. 
    Moreover, in the above filtration, $\Ol_C$ is the unique degree-zero line subbundle of $\El_{2,0}$, which in turn is itself is the unique unipotent rank-$2$ subbundle of $\El_{3,0}$.  
\end{remark}

\begin{proposition}\label{pro:E30}
    Let $X=\P(\El_{3,0})$. Then $\Autz(X)$ is not a maximal connected algebraic subgroup of $\Bir(X)$.
\end{proposition}

\begin{proof}
    By the previous remark, the line subbundle $\Ol_C \subset \El_{3,0}$ is unique. Hence, it corresponds to an $\Autz(X)$-invariant section of $\pi$. Since the vector bundle $\El_{3,0}$ is unipotent, we can choose the transition matrices of $\P(\El_{3,0})$ to be upper triangular with $1$'s on the diagonal. 
    
    Over $\mathbb{C}$, these were explicitly computed by Takahashi in \cite[Section 4, case (xv)]{Takahashi}. We now recall his result, which extend without additional assumptions over any algebraically closed field of characteristic zero. Denoting by $p_0$ the neutral element of the elliptic curve $C$, the vector bundle $\El_{3,0}$ is trivial over $U_k = C\setminus \{p_0\}$, and over any affine open neighborhood $U_l$ containing $p_0$, and the transition map is given by
    \[
    \begin{array}{ccc}
        U_l \times \A^3 & \dashrightarrow & U_k \times \A^3 \\
        (x,(u,v,w)) & \longmapsto & (x,(u +\alpha_{kl}(x)^{-1}v,v+ \alpha_{kl}(x)^{-1}w,w)),
    \end{array}
    \]
    where $\alpha_{kl}\in \Ol_C(U_{kl})$ is a rational function such that $\mathrm{div}(\alpha_{kl})_{\vert U_l} = p_0$. Moreover, it follows from Takahashi's results that 
   \[
   \begin{pmatrix}
    1 & \alpha_{kl}^{-1} \\ 0 & 1
   \end{pmatrix}
   \]
   are the transition matrices of the indecomposable vector bundle $\El_{2,0}$.

    Under this choice of trivializations, the $\Autz(X)$-invariant section corresponding to the line subbundle $\Ol_C\subset \El_{3,0}$ is given by the equation $\{v=w=0\}$. Its blow-up yields a birational morphism $\eta\colon X_1 \to X$ such that $\eta^{-1} \Autz(X) \eta \subseteq \Autz(X_1)$, where $X_1$ is an $\F_1$-bundle with transition maps
    \[
    \begin{array}{ccc}
        U_l \times \F_1 & \dashrightarrow & U_k \times \F_1 \\
        (x,[y_0:y_1 \ ;z_0:z_1]) & \longmapsto & (x,[y_0:y_1 + \alpha_{kl}(x)^{-1}z_0y_0 \ ;z_0+\alpha_{kl}(x)^{-1}z_1:z_1]).
    \end{array}
    \]
   The projection onto the coordinates $(x,[z_0:z_1])$ equips $X_1$ with a structure of $\p^1$-bundle over $\Al_{2,0}$. By \autoref{rmk:invariants}, the $\P^1$-bundle $X_1 \to \Al_{2,0}$ has invariants $(\Al_{2,0},1,0)$. Thus, by \cite[Proposition 7.6]{Fong23}, $\Autz(X_1)$ is not maximal. 
\end{proof}

\subsection{Maximal connected algebraic subgroups of $\Bir(C\times \p^2)$}            

\begin{theorem}\label{thm:DP9 max}
    Let $C$ be an elliptic curve and $\pi\colon X\to C$ be a $\p^2$-bundle. Then $\Autz(X)$ is maximal connected algebraic subgroup of $\Bir(X)$ if and only if one of the following holds:
    \begin{enumerate}
        \item\label{DP9 max:1} $X$ is isomorphic to one of the indecomposable $\p^2$-bundles $\Al_{3,1}$ or $\Al_{3,2}$.
        In both cases, there exists a short exact sequence
            \[
            1 \to (\Z/3\Z)^2 \to \Autz(X) \to \Autz(C) \to 1.
            \]
            
        \item\label{DP9 max:2} $X\simeq \p(\El_{2,0}\oplus \Ml)$, where $\Ml\in \Pic^0(C)$ has infinite order. In this case, there exists a short exact sequence
            \[
            1\to \G_m \rtimes \G_a \to \Autz(X) \to \Autz(C) \to 1. 
            \]

        \item\label{DP9 max:3} $X\simeq \p(\Ol_C \oplus \Ll \oplus \Ml)$, where $\Ll\in \Pic^0(C)$ or $\Ml \in \Pic^0(C)$ is trivial, or both non-trivial and non-isomorphic such that $\Ll^{\otimes n} \otimes \Ml^{\otimes m}$ is not trivial for every $(n,m) \in \mathbb{Z}^2$ coprime. In these cases, there exists a short exact sequence
             \[
             1 \to \Autz(X)_C \to \Autz(X) \to \Autz(C) \to 1,
             \]
             where 
             \[
             \Autz(X)_C \simeq
             \begin{cases}
                \PGL_3(\k) \text{ if $\Ll\simeq \Ml\simeq \Ol_C$ (equivalently, $X\simeq C\times \p^2$)},\\
                \GL_2(\k) \text{ if exactly one of $\Ll$ and $\Ml$ is trivial}, \\
                \G_m^2 \text{ if both $\Ll$ and $\Ml$ are non-trivial and non-isomorphic.}
             \end{cases}
             \]
    \end{enumerate}
\end{theorem}

\begin{proof}
    Assume that $\Autz(X)$ is maximal. By \autoref{prop:trivial_action_on_C}, the morphism of algebraic groups $\pi_*\colon \Autz(X) \to \Autz(C)$ is surjective, so $X$ is one of the $\p^2$-bundles listed in \autoref{lem:hom_proj_bundles}. All cases were studied in \autoref{section:decomposable_hom_bundles} and \autoref{section: indecomposable}.
    
    If $\pi$ is indecomposable then $X=\Al_{3,i}$ for some $i\in \{0,1,2\}$. Then $\Autz(\Al_{3,i})$ is conjugate to a maximal connected algebraic subgroup of $\Bir(C\times \p^2)$ if and only if $i\in \{1,2\}$ by \autoref{prop:A31 and A32} and \autoref{pro:E30}.
    
    Else, $\pi$ is decomposable, and two cases arise:
    \begin{enumerate}
    \item Either $X = \p(\El_{2,0}\oplus \Ml)$ for some line bundle $\Ml$, and the result follows from \autoref{prop:E'=E_20} and \autoref{lem:E_20+O}.
    \item Or, $X = \p(\Ol_C \oplus \Ll \oplus \Ml)$, where $\Ll$ and $\Ml$ are line bundles. If $\Ll$ and $\Ml$ are both trivial, then $X=C\times \p^2$ and $\Autz(X)\simeq \Autz(C)\times \PGL_3(\k)$ is maximal. If $\Ll$ or $\Ml$ is trivial, but not both, then the result follows from \autoref{prop:E'=O+O}. If they are both non-trivial, we can assume that they are non-isomorphic, and the result follows from \autoref{lem:O+L+M} and \autoref{lem:O+L+M_conjugacy}.
    \end{enumerate}
\end{proof}

\section{$\p^2$-bundles over a smooth projective curve of higher genus}\label{s:higher genus}

Let $\pi\colon X\to C$ be a Mori Del Pezzo fibration of degree $9$ over a smooth projective curve $C$ of genus $g(C)\geq2$ such that $\Autz(X)$ is maximal. Then by \autoref{prop:Mdp to bundle}, we can assume that $\pi$ is a $\p^2$-bundle. Moreover, Blanchard's lemma (\autoref{lem:blanchard}) and $\Autz(C)$ being trivial imply that $\Autz(X)=\Autz(X)_C$.

\begin{definition}\label{def:slope}
    Let $\El$ be a vector bundle over a smooth projective curve $C$. The \emph{slope} of $\El$ is the quantity
    \[
    \mu(\El) = \deg(\El)/\rk(\El).
    \]
    We say that $\El$ is:
    \begin{enumerate}
        \item \emph{stable} if $\mu(\El)>\mu(\Fl)$ for every non-zero proper subbundle $\Fl\subset \El$.
        
        \item \emph{semistable} if $\mu(\El)\geq \mu(\Fl)$ for every non-zero proper subbundle $\Fl\subset \El$.

        \item \emph{polystable} if $\El$ is semistable and is a direct sum of stable bundles.
        
        \item \emph{strictly semistable} if $\El$ is semistable but not stable.
        \item \emph{unstable} if $\El$ is not semistable.
    \end{enumerate}
\end{definition}

In \autoref{lem:stability_dim2}, we will see that if $\rk(\El)=2$, then the stability of $\El$ can be understood through the geometry of the associated $\p^1$-bundle. 

\begin{remark}\label{rem:sections_linebundles}
    \begin{enumerate}
    \item Assume that $\rk(\El)=2$. Then $\pi\colon S=\p(\El)\to C$ is a $\p^1$-bundle over $C$. 
    Then by \cite[Lemma 5]{Maruyama}, there exists a one-to-one correspondance 
    \[
    \begin{array}{ccc}
    \{\text{sections of $\pi$}\} & \to & \{\text{line subbundles of $\El$}\} \\
    \sigma & \longmapsto & \Ll(\sigma)
    \end{array},
    \]
    which sends minimal sections to line subbundles with maximal degree.

    \item By \cite[Proposition 2.15]{Fong20},
    \[
    \sigma^2 = \deg(\El) - 2\deg(\Ll(\sigma)) = 2(\mu(\El) - \mu(\Ll(\sigma))),
    \]
    where $\Ll(\sigma)\subset \El$ is the line subbundle corresponding to the section $\sigma\subset S$.
    \end{enumerate}
\end{remark}

\begin{definition}
    Let $\pi\colon S\to C$ be a ruled surface over a smooth projective curve $C$. The quantity
    \[
    \seg(S) := \min\{\sigma^2, \text{ $\sigma$ section of $\pi$}\},
    \]
    is an invariant, sometimes called the \emph{Segre invariant}.
\end{definition}

\begin{lemma}\label{lem:stability_dim2}
    Let $\El$ be a rank-$2$ vector bundle over $C$ and $S=\p(\El)$. Then 
    \[\seg(S) 
    \begin{cases}
    $>0$ \text{ if and only if $\El$ is stable,} \\
    $=0$ \text{ if and only if $\El$ is strictly semistable,} \\
    $<0$ \text{ if and only if $\El$ is unstable.}
    \end{cases}\]
\end{lemma}

\begin{proof}
    By \autoref{rem:sections_linebundles}, we have
    \[
    \seg(S) = 2\min\{\mu(\El) - \mu(\Fl)\}, 
    \]
    where the minimum is taken over all line subbundles $\Fl\subset \El$. Equivalently,
    \[
    \frac{\seg(S)}{2} = \mu(\El) - \mu(\deg(\Fl)),
    \]
    where the maximum is taken over all line subbundles $\Fl\subset \El$. Then it follows from \autoref{def:slope} that $\seg(S)$ is positive if and only if $\El$ is stable, zero if and only if $\El$ is strictly semistable, and negative otherwise.
\end{proof}

\begin{lemma}\label{lem:stable C high genus}
Let $n\geq1$ and $\pi\colon X\to C$ a $\p^n$-bundle above a smooth projective curve. Let $\El$ be a vector bundle such that $X=\p(\El)$. 
If $\El$ is stable, then $\Autz(X)$ is trivial. 
\end{lemma}

\begin{proof}
Since $\El$ is stable, we have $\Aut(\El)_C=\G_m$ (see \cite[Corollary 1.2.8]{Huybrechts_Lehn}). We also have a short exact sequence
\[
1\to\Aut(\El)_C/\G_m\to \Aut(X)_C\to \Omega\to 1
\]
where $\Omega$ is a finite group (see \autoref{SES:Grothendieck}). Hence $\Autz(X)_C \simeq \Omega$ is finite. Since $\Autz(X)\subseteq \Aut(X)_C$, it follows that $\Autz(X)$ is trivial. 
\end{proof}

If $\El$ is not stable, then $\Aut(\El)_C$ may strictly contain $\G_m$ and the situation is more complicated. However, $\El$ always contains a subbundle which is invariant by automorphisms:

\begin{lemma}\cite[Lemmas 1.3.5 and 1.5.5]{Huybrechts_Lehn}\label{lem:destabilizing+socle}
    Let $\El$ be a vector bundle over a smooth projective curve $C$. Then the following hold:
    \begin{enumerate}
        \item\label{lem:destabilizing} If $\El$ is unstable, then there exists a unique subbundle $\Fl\subset \El$ such that $\mu(\Fl)\geq\mu(\Gl)$ for every subbundle $\Gl \subset \El$, and $\Gl \subset \Fl$ if $\mu(\Fl)=\mu(\Gl)$. Moreover, $\Fl$ is semistable.

        \item\label{lem:socle} If $\El$ is semistable, then $\El$ contains a unique non-zero maximal polystable bundle $\Fl$ such that $\mu(\Fl)=\mu(\El)$.
    \end{enumerate}
    In both cases, $\Fl$ is invariant under automorphisms of $\El$.
\end{lemma}

\begin{definition}
    In \autoref{lem:destabilizing+socle} \autoref{lem:destabilizing}, $\Fl$ is called the \emph{maximal destabilizing subbundle} of $\El$. In \autoref{lem:destabilizing+socle} \autoref{lem:socle}, $\Fl$ is called the \emph{socle} of $\El$.
\end{definition}

\begin{remark}\label{rmk:destablizing}
Let $\pi\colon X\to C$ be a $\p^2$-bundle above a smooth projective curve. Let $\El$ be a vector bundle of rank-$3$ such that $X=\p(\El)$. 
If $\El$ is unstable, then there exists a unique maximal destabilizing subbundle $\Fl\subseteq\El$ by \autoref{lem:destabilizing+socle}. If $\rk(\Fl)=1$, then $\p(\Fl)$ is a section of $\pi$ and if $\rk(\Fl)=2$, then $\pi|_{\p(\Fl)}\colon\p(\Fl)\to C$ is a $\p^1$-subbundle of $\pi$. 

Moreover, since $\Fl$ is unique, $\p(\Fl)$ is preserved by $\Autz(X)$. Furthermore, by definition of destabilizing subbundle, $\Fl$ is semistable.
\end{remark}

\begin{lemma}\label{lem:unstable C high genus}
Let $\pi\colon X\to S$ be a $\p^2$-bundle above smooth projective curve of genus $g(C)\geq2$. Let $\El$ be a vector bundle of rank $3$ such that $X=\p(\El)$. Suppose that $\El$ is unstable. 
Then $\Autz(X)$ is not a maximal connected algebraic subgroup of $\Bir(X)$. 
\end{lemma}

\begin{proof}
Let $\Fl$ be the maximal destabilizing subbundle of $\El$. Then $\rk(\Fl) \in \{1,2\}$ and we prove that in each case, $\Autz(X)$ is not maximal.

If $\rk(\Fl)=1$, the extraction of the $\Autz(X)$-invariant section $\p(\Fl)$ is an $\Autz(X)$-equvariant link of type I by \autoref{lem:section blow-up}. Then \autoref{lem:F_b_trivial_action_on_C} implies that $\Autz(X)$ is not a maximal connected algebraic subgroup of $\Bir(X)$.

Suppose that $\rk(\Fl)=2$ and let $S:=\p(\Fl)$. Then $\theta:=\pi|_{S}\colon S\to C$ is a $\p^1$-subbundle of $\pi$ and a ruled surface. 
By \autoref{rmk:destablizing}, $\Fl$ is semistable.
If $\Fl$ is stable, $\Autz(S)$ is trivial by \autoref{lem:stable C high genus}. 
Then $\Autz(X)$ fixes any section $s$ of $\theta$. The curve $s$ is also a section of $\pi$ and its blow-up is an $\Autz(X)$-equivariant link of type I by \autoref{lem:section blow-up}. Then \autoref{lem:F_b_trivial_action_on_C} implies that $\Autz(X)$ is not maximal.
If $\Fl$ is strictly semistable, we look at the two cases $S\nsimeq C\times\p^1$ and $S\simeq C\times\p^1$.

Suppose that $S\nsimeq C\times\p^1$. Since $\Fl$ is strictly semistable, the minimal self-intersection $\sigma(S)$ of any section of $\theta$ satisfies $\sigma(S)\leq0$ by \autoref{lem:stability_dim2}. 
Since $S\nsimeq C\times\p^1$, there are only finitely many sections of $\theta$ of self-intersection $\sigma(S)$ \cite[Proposition 2.18]{Fong20}. In particular, they are all preserved by $\Autz(X)$. The blow-up of such a section yields an $\Autz(X)$-equivariant link of type I by \autoref{lem:section blow-up}. Then
\autoref{lem:F_b_trivial_action_on_C} implies that $\Autz(X)$ is not maximal. 

Suppose that $S\simeq C\times \p^1$. Then we can assume that $\Fl\simeq\Ol_C\oplus\Ol_C$.
There exists a line bundle $\Ml$ and a short exact sequence
\[
0\to \Fl \to \El \to \Ml \to 0.
\]
The filtration $0\subset \Fl \subset \El$ is the Harder-Narasimhan filtration of $\El$ (see \cite[Definition 1.3.2 and Theorem 1.3.4]{Huybrechts_Lehn}). In particular, 
$\deg(\Ml)<0=\mu(\Fl)$. There exist trivializations of $\El$ such that the transition matrices are
\[
\begin{pmatrix}
1&0&\star \\ 0&1&\star \\ 0&0&\mu_{kl}
\end{pmatrix},
\]
where $\mu_{kl}\in \Ol_C(U_{kl})^*$ are the cocycles of $\Ml$. 
Then for any $p\in C$, the projective line 
\[
f_p=\pi^{-1}(p) \cap S = (p,[x:y:0])
\]
is $\Autz(X)$-invariant. Let $\phi\colon X\dashrightarrow X'$ be the birational map, locally given by
\[
\begin{array}{ccc}
U_l \times \p^2 & \dashrightarrow & U_l\times \p^2\\
(x,[u:v:w]) & \longmapsto & (x,[\xi(x)u:\xi(x)v:w]),
\end{array}
\]
where $U_l$ is a trivializing open subset containing $p$, and $\xi\in \k(C)^*$ satisfying $\mathrm{div}(\xi)_{\vert U_l}=p$, and $\phi$ is locally the identity on any chart not containing $p$. Then notice that this birational map is the blow-up of $f_p$ followed by the contraction of the strict transform of $\pi^{-1}(p)$ onto the point $(p,[0:0:1])$. In particular, $\phi$ is $\Autz(X)$-equivariant. 

The base locus of $\phi^{-1}$ is the point $(p,[0:0:1])$ and the obtained $\p^2$-bundle $\pi'\colon X'\to C$ has transition matrices of the form
\[
\begin{pmatrix}
1&0&\star \\ 0&1&\star \\ 0&0& \xi \mu_{kl}
\end{pmatrix}.
\]
This implies that $X' \simeq \p(\El')$, where $\El'$ is a rank-$3$ vector bundle which fits into a short exact sequence
\[
0\to \Fl \to \El' \to \Ol_C(-p) \otimes \Ml \to 0.
\]
Since $\phi$ is $\Autz(X)$-equivariant, we have that $\phi \Autz(X)\phi^{-1} \subseteq \Autz(X')$, and this is an equality if and only if $(p,[0:0:1])\in X'$ is $\Autz(X')$-invariant. We claim that $\El'$ is again unstable with $\Fl$ as a maximal destabilizing subbundle. Indeed, $\El'$ is unstable as we have 
$$\mu(\El')=\deg(\El')/3 = (\deg(\Ml)-1)/3 = (\deg(\El)-1)/3=\mu(\El)-\frac{1}{3}<\mu(\Fl)=0. $$ 
 Assume $\Fl'\subset \El'$ is the maximal destabilizing subbundle of $\El'$ and $\mu(\Fl')>\mu(\Fl)$. In the following commutative diagram
\[
\begin{tikzcd}
0 \ar[r] & \Fl \ar[r] & \El \ar[r] & \Ml \ar[r] & 0 \\
&& \Fl' \ar[u,hook]&&,
\end{tikzcd}
\]
the induced morphism $\Fl'\to \Ml$ is zero, as $\mu(\Fl')> \mu(\Ml)$ by \cite[Proposition 1.2.7]{Huybrechts_Lehn}. Therefore, the inclusion $\Fl' \to \El$ factors through an injective morphism $\Fl' \to \Fl$. By loc. cit., as $\mu(\Fl')>\mu(\Fl)$, this morphism is also zero, which is a contradiction. Thus, $\Fl$ is the maximal destabilizing subbundle of $\El$.

Iterating the construction, we obtain an $\Autz(X)$-equivariant birational map $\psi\colon X\dashrightarrow X''$, such that $\psi^{-1}$ has base locus $(p,[0:0:1])\in X''$. Moreover, $X'' \simeq \p(\El'')$ where $\El''$ is a rank-$3$ vector bundle which sits into a short exact sequence
\[
0\to \Fl \to \El'' \to \Ml'' \to 0,
\]
and we can assume $\deg(\Ml'')<<0$ is arbitrary small. By \autoref{key_computation}, $f\in \Gamma(C,\Ml''^{-1})$ can be chosen non-zero by Riemann-Roch theorem. Thus, the point $(p,[0:0:1])\in X''$ is not fixed by $\Autz(X'')$. This shows that $\phi\Autz(X)\phi^{-1}\subsetneq\Autz(X'')$.
\end{proof}

\begin{remark}
    In \autoref{lem:unstable C high genus}, we have used that the maximal destabilizing subbundle $\Fl$ is a non-zero proper subbundle of the unstable bundle $\El$. In particular, $\rk(\Fl)\in \{1,2\}$. In \autoref{lem:semistable_high_genus}, we will proceed similarly as in the proof \autoref{lem:unstable C high genus}, with the assumption that $\El$ is semistable and $\Fl\subset \El$ is its socle. However, $\Fl$ is not necessarily a proper subbundle of $\El$. This fails precisely when $\El$ is itself a polystable bundle and in this case, $\Fl = \El$. Hence, in the proof of \autoref{lem:semistable_high_genus}, we need to consider the cases where $\Fl\subset \El$ is the socle and $\rk(\Fl)\in \{1,2,3\}$.
\end{remark}

\begin{proposition}\label{lem:semistable_high_genus}
Let $C$ be a smooth projective curve of genus $g(C)\geq2$ and let $\pi\colon X\to C$ be a $\p^2$-bundle. If $\Autz(X)$ is a maximal connected algebraic subgroup of $\Bir(X)$, then $X$ is isomorphic to one of the following:
\begin{enumerate}
    \item\label{semistable_high_genus:1} $X\simeq C\times \p^2$,
    \item\label{semistable_high_genus:2} $X\simeq \p(\El \oplus \Ml)$, where $\El$ is a rank-$2$ stable bundle which fits into a short exact sequence
    \[
    0\to \Ol_C \to \El \to \Ll \to 0,
    \]
    where $\Ll$ and $\Ml$ are line bundles such that $\deg(\Ll)=2\deg(\Ml)>0$.
\end{enumerate}
\end{proposition}

\begin{proof}
    Let $\El$ be a rank-$3$ vector bundle over $C$ such that $\p(\El)\simeq X$. 
    If $\El$ is unstable, then $\Autz(X)$ is not maximal by \autoref{lem:unstable C high genus}. From now on, we assume that $\El$ is semistable and we denote by $\Fl$ the socle of $\El$. Recall from \autoref{lem:destabilizing+socle} that $\Fl$ is polystable, $\mu(\Fl)=\mu(\El)$, and $\p(\Fl)\subset X$ is an $\Autz(X)$-invariant projective subbundle. We continue the proof by analyzing the different possibilities for the rank of $\Fl$.

    \begin{enumerate}
        \item\label{case.rk1} If $\rk(\Fl) = 1$, then $\Fl$ corresponds to an $\Autz(X)$-invariant section of $\pi\colon X\to C$. Its blowup yields an birational morphism $\eta\colon X'\to X$, where $X'$ is an $\F_1$-bundle over $C$ (\autoref{lem:section blow-up}). 
        By \autoref{lem:F_b_trivial_action_on_C}, $\Autz(X)$ is not maximal.
        \item If $\rk(\Fl)=2 $, then $S=\p(\Fl)$ is a $\p^1$-subbundle of $\p(\El)$. If $\Fl$ is stable, then $\Autz(S)$ is trivial by \autoref{lem:stable C high genus}, so $\Autz(X)$ fixes any section of $\p(\Fl)$. Then we conclude as in case \autoref{case.rk1}. Else $\Fl$ is not stable, and the minimal self-intersection $\sigma(S)$ of any section of $S$ is at most $0$ by \autoref{lem:stability_dim2}. 
        Then two cases arise. Either $S$ is not isomorphic to $C\times \p^1$, then $S$ admits an $\Autz(S)$-invariant section, and we conclude as in case \autoref{case.rk1}. Or, $S\simeq C\times \p^1$, and we can assume that $\Fl \simeq \Ol_C \oplus \Ol_C$ and that $\pi\colon X\to C$ has transition matrices of the form 
        \[
        \begin{pmatrix}
        1&0&\star \\ 0&1&\star \\ 0&0& \mu_{kl}
        \end{pmatrix},
        \]
        where $\mu_{kl}\in \Ol_C(U_{kl})^*$ are the cocyles of the quotient line bundle $\Ml =\El/\Fl$. In particular, we have 
        \[
        \mu(\El) = \frac{2}{3}\mu(\Fl) + \frac{\deg(\Ml)}{3}.
        \]
        Since $\mu(\El) = \mu(\Fl)$ by definition of the socle, it implies that $\deg(\Ml)=\mu(\Fl )=0$.
        
        For any $p\in C$, the projective line
        \[
        f_p = \pi^{-1}(p) \cap S = \{(p,[x:y:0])\}
        \]
        is $\Autz(X)$-invariant. As in the proof of \autoref{lem:unstable C high genus}, the blowup of $f_p$ followed by the contraction of the strict transform of $\pi^{-1}(p)$ induces an $\Autz(X)$-equivariant birational map $X\dashrightarrow X'$, where $\pi'\colon X'\to C$ is a $\p^2$-bundle with transition matrices 
        \[
        \begin{pmatrix}
        1&0&\star \\ 0&1&\star \\ 0&0& \xi \mu_{kl}
        \end{pmatrix},
        \]
        where $\xi \in \k(C)^*$ satisfies $\mathrm{div}(\xi)_{\vert U_l} = p$. Therefore, $X'\simeq \p(\El')$, where $\El'$ is a rank-$3$ vector bundle which fits into a short exact sequence
        \[
        0 \to \Ol_C \oplus \Ol_C \to \El' \to \Ol_C(-p) \otimes \Ml \to 0.
        \]
        Notice that $\mu(\El')=-1/3 < \mu(\Ol_C \oplus \Ol_C) =0$. Thus, $\El'$ is unstable, so $\Autz(X)$ is not maximal by Lemma \ref{lem:unstable C high genus}.

        \item It remains to consider the case where $\rk(\Fl)=3$, i.e., $\El=\Fl$ is polystable. If $\El$ is stable, we conclude from \autoref{lem:stable C high genus} that $\Autz(X)$ is trivial, and hence not maximal. Else, $\El$ is the direct sum of two proper subbundles: either $\El$ is the sum of a rank-$2$ stable bundle and a line bundle, or the sum of three line bundles.
        If $\El \simeq \El'' \oplus \Ml$, where $\El''$ is stable and $\Ml$ is a line bundle, then we can assume that $\El''$ fits into a short exact sequence
        \[
        0 \to \Ol_C \to \El'' \to \Ll \to 0,
        \]
        where $\Ll$ is a line bundle. Since $0=\mu(\Ol_C)<\mu(\El'') = \deg(\Ll)/2$, it implies that $\deg(\Ll)>0$. Moreover, as $\El$ is polystable, it follows that
        \[
        \mu(\El) =\mu(\El'') = \deg(\Ml),
        \]
        so $\deg(\Ll) =     2\deg(\Ml)>0$. This gives \autoref{semistable_high_genus:2} and we now exclude the remaining case.

        If $\El$ is the sum of three line bundles, then we can assume $\El = \Ol_C \oplus \Ll \oplus \Ml$, where $\Ll,\Ml \in \Pic^0(C)$. Indeed, by definition of polystable, they have the same degree. Then tensoring by the dual of one of them, we can assume that they are as stated. If $\Ll \simeq \Ml \simeq \Ol_C$, then $X = C\times \p^2$ and $\Autz(X)=\PGL_2(\k)^2$ is maximal. Up to tensoring by a line bundle, we can assume that the other possibilities for $\Ll$ and $\Ml$ are the following:
        \begin{enumerate}
            \item $\Ll\simeq \Ol_C$ and $\Ml$ is non-trivial,
            \item $\Ll$ and $\Ml$ are not isomorphic and non-trivial. 
        \end{enumerate}
        In both cases, the section $C\times \{[0:0:1]\}$ corresponding to the line subbundle $\Ml$ is $\Autz(X)$-invariant by \autoref{key_computation}. By \autoref{lem:section blow-up} and \autoref{lem:F_b_trivial_action_on_C}, $\Autz(X)$ is not maximal.
    \end{enumerate}
\end{proof}

\begin{remark}\label{rem:stable+line_bundle}
    If $X=C\times\p^2$, then the orbits of $\Autz(X) \simeq \PGL_3(\k)$ are the fibres of $\pi\colon X\to C$. In particular, $\Autz(X)$ is a maximal connected algebraic subgroup of $\Bir(X)$. 
    However, we don't know whether Case  of \autoref{lem:semistable_high_genus}\autoref{semistable_high_genus:2} actually occurs. If $X\simeq \p(\El \oplus \Ml)$ as stated, the stable subbundles $\El$ and $\Ml$ may not be $\Autz(X)$-invariant. In this case, \autoref{key_computation} does not provide any information whether $\El$ and $\Ml$ are $\Autz(X)$-equivariant or not.
\end{remark}

\section{Proof of \autoref{thm:mainDP} and \autoref{thm:main_superrigidity}} \label{s:proofAB}

\begin{proof}[Proof of \autoref{thm:mainDP}]
If $d\in\{1,2,3,4,6\}$, the claim is from \autoref{thm:K<5}. If $d=8$, then by \autoref{lem:Mfs8 restriction}, either $\Autz(X)\simeq\PGL_2$ acts diagonally on a general fibre of $\pi$ or there is an $\Autz(X)$-equivariant birational map $X\rat Y$ to a Mori Del Pezzo fibration $Y\to B$. In the first case, $\Autz(X)$ is  maximal by \autoref{lem:Mfs8 Aut=PGL2}. 
If $d=9$ and $g(C)=1$, the claim follows from \autoref{thm:DP9 max}. If $d=9$ and $g(C)\geq2$, the claim follows from \autoref{lem:semistable_high_genus}.
\end{proof}

\begin{proof}[Proof of \autoref{thm:main_superrigidity}]
Assume first that $g(C)\geq 2$. If $X\simeq C\times \p^2$, then $\Autz(X) \simeq \PGL_3(\k)$ acts with orbits of dimension two, that are the fibers of the trivial $\p^2$-bundle $C\times \p^2 \to C$. Hence $(X,\pi)$ is superrigid. 

For the other cases, we have $g(C)=1$. If $X\simeq \Al_{3,1}$ or $\Al_{3,2}$, then the superrigidity of $(X,\pi)$ follows from \autoref{prop:A31 and A32}. If $X=C\times \p^2$, then $\Autz(X)\simeq \Autz(C) \times \PGL_3(\k)$ acts with a single orbit; thus, $(X,\pi)$ is also superrigid. In the other cases, the statements on the conjugacy classes follow from \autoref{lem:O+L+M_conjugacy}, \autoref{prop:E'=O+O}, and \autoref{prop:E'=E_20}.
\end{proof}

\begin{proof}[Proof of \autoref{mainprop:MDP5}]
Recall from \autoref{pro:DP5} that there is a short exact sequence
\[
1\to\Z/5\to\Autz(X)\to\Autz(C)\to 1.
\]
Let $\varphi\colon X\rat Y$ be an $\Autz(X)$-equivariant Sarkisov link to a Mori fibre space $\pi'\colon Y\to B$. 
By \autoref{lem:links I}, \autoref{lem: links III and IV}, $\varphi$ is a link of type II and in particular, $\pi'$ is a Mori Del Pezzo fibration. 
By \autoref{lem:type II}, $\varphi$ induces a link of type II $X_{\k(C)}\rat Y_{\k(C)}$. By \cite[Theorem 2.6(ii)($K_X^2=5$)]{iskovskikh_1996}, $Y_{\k(C)}$ is of degree $d\in\{5,8,9\}$. 
The case $d=8$ is impossible by \autoref{lem:dP8 Aut on C}, because $\Autz(X)$ acts non-trivially on $C$.

We have $d=9$ if and only if $\varphi$ extracts an $\Autz(X)$-equivariant curve \cite[Theorem 2.6(ii)($K_X^2=5$)]{iskovskikh_1996}. If this is the case, then by \autoref{prop:Mdp to bundle}, there is an $\Autz(Y)$-equivariant birational map $\psi\colon Y\rat Z$ over $C$ to a $\p^2$-bundle $\pi''\colon Z\to C$. Then $(\psi\varphi)\Autz(X)_C(\psi\varphi)^{-1}\simeq\Z/5$. 
From \autoref{thm:mainDP} we obtain that $\Autz(X)$ is not maximal.  

If $d=5$, then $\Autz(Y)$ is also an elliptic curve and hence $\varphi\Autz(X)\varphi^{-1}=\Autz(Y)$. 

Let $\varphi\colon X\rat Y$ be an $\Autz(X)$-equivariant birational map to a Mori fibre space $\pi'\colon Y\to B$. By \autoref{thm:Enrica}, it is a composition $\varphi=\varphi_n\circ\cdots\circ\varphi_1$ of equivariant Sarkisov links $\varphi_1,\dots,\varphi_n$. We have shown above that $\Autz(X)$ is maximal if and only if $\varphi_1,\dots,\varphi_n$ are of type II and the target is a Mori Del Pezzo fibration of degree $5$. We've also shown that this is the case if and only if $X$ does not contain any $\Autz(X)$-invariant section.  
\end{proof}

\section{From Mori fibre spaces over a surface to projective bundles}\label{s:to_proj_bundles}

The aim of this section is to reduce the case of Mori fibre spaces of dimension $3$ over a surface to projective bundles.
Let $\pi\colon X\to S$ be a Mori fibre space with $\dim X=3$ and $\dim S=2$. In \cite{BFT22} it is called Mori conic bundle. 
The generic fibre of $\pi$ is a smooth curve of genus zero (since $-K_X$ is relatively ample). 

In this chapter, we look at the case when $S$ is an irrational surface.

\subsection{Reduction to standard conic bundles}

\begin{definition}
A surjective morphism $\pi\colon X\to S$ is called {\em standard conic bundle} if
\begin{enumerate}
\item the varieties $X,S$ are smooth projective and $\dim X=\dim S+1$;
\item  The morphism $\pi$ is induced by the inclusion of $X$ (given by an equation of degree $2$) in a $\mathbb{P}^2$-bundle over $S$. The discriminant divisor $\Delta \subset S$ is reduced, and all its components are smooth and intersect in normal crossings (i.e., $\Delta$ is an SNC divisor).
For each $p \in S$, the rank of the $3 \times 3$ matrix corresponding to the quadric equation is $3,2,1$, respectively when $p\notin\Delta$, $p\in\Delta\setminus\mathrm{sing}(\Delta), p\in\mathrm{sing}(\Delta)$.
\item The relative Picard rank is $\rho(X/S) = 1$.
\end{enumerate}
\end{definition}

In view of the above definition, we will identify the divisor $\Delta$ and the union of curves supporting it. 

The following statement applies to Mori fibre threefolds over surfaces.

\begin{proposition}[{\cite[Theorem 3.1.4]{BFT22}}]\label{prop:reduction to standard conic bundles}
Let $S$ be a surface, $X$ a normal variety and $\pi\colon X\to S$ a surjective morphism whose generic fibre is a smooth curve of genus zero and such that $-K_X$ is relatively ample and all fibres are union of curves. 
Then there is an $\Autz(X)$-equivariant commutative diagram
\[
\begin{tikzcd}
\hat X\ar[r,dashed,"\psi"]\ar[d,"\hat\pi"] & X\ar[d,"\pi"]\\
\hat S\ar[r,"\eta"] & S
\end{tikzcd}
\]
where $\psi$ is a birational map, $\eta$ is a birational morphism and $\hat\pi\colon\hat X\to\hat S$ is a standard conic bundle.
 Moreover, if $S$ is smooth, we can take $\eta$ to be an isomorphism.
\end{proposition}

In the next statement, \cite[Lemma 3.1.5]{BFT22} assumes $S$ to be rational, but this is not necessary.

\begin{lemma}[{\cite[Lemma 3.1.5]{BFT22}}]\label{lem:empty discriminant}
Let $S$ be a smooth projective surface, let $\pi \colon X \to S$ be a standard conic bundle, let $\Delta \subset S$ be its discriminant curve and let $K := \k(S)$. Then the following are equivalent:
\begin{enumerate}
    \item $X$ is a $\mathbb{P}^1$-bundle over $S$;
    \item the generic fibre $X_K$ is isomorphic to $\mathbb{P}^1_K$;
    \item $\pi$ has a rational section; and
    \item $\Delta=\emptyset$. 
\end{enumerate}
Moreover, if $\Delta$ is non-empty and reducible, then each rational irreducible component $C$ of $\Delta$ intersects $\overline{\Delta \setminus C}$ in at least two distinct points. If $\Delta$ is non-empty and irreducible, then $g(\Delta) \geq 1$.
\end{lemma}

\begin{lemma}\label{lem:components of Delta irrational 0}
Let $\pi\colon X\to S$ be a standard conic bundle and let $\Delta$ be its discriminant. 
Suppose that there is a birational morphism $\eta\colon S\to S'$ to a ruled surface $\tau\colon S'\to C$ above a curve $C$ of genus $g(C)\geq1$. 
Then the following hold:
\begin{enumerate}
\item Every irreducible component of $\Delta$ not contained in a fibre of $\theta\circ\eta$ is irrational. 
\item Every rational irreducible component of $\Delta$ is a smooth fibre of $\theta\circ\eta$ and intersects an irrational irreducible component.
\end{enumerate}
In particular, if $\Delta$ is non-zero, then it has an irrational irreducible component.
\end{lemma}
\begin{proof}
We suppose that $\Delta$ is non-zero, otherwise the statement holds trivially.
Let $D$ be an irreducible component of $\Delta$. First, suppose that $D$ is not contained in any fibre of $\tau\circ\eta$.
Then $(\tau\circ\eta)|_D\colon D\to C$ is a finite cover. Since $g(C)\geq1$, the Riemann-Hurwitz formula \cite[\href{https://stacks.math.columbia.edu/tag/0C1B}{Tag 0C1B}]{SP24} implies that $g(D)\geq1$.
It follows that, $D$ is rational if and only if it is contained in a fibre $f$ of $\tau\circ\eta$. 

Suppose that $D$ is rational. Then, by \autoref{lem:empty discriminant}, $\Delta$ is reducible and $D$ intersects $\overline{\Delta\setminus D}$ in exactly two points. If $D$ is a smooth fibre, it intersects an irreducible component of $\Delta$ not contained in the fibres of $\tau\circ\eta$. 
If $D$ is not a smooth fibre, then $f$ has two other irreducible components, each of which intersect $D$ and each of which are components of $\Delta$. We iterate the argument for these two components and so on. This leads to an infinite chain of intersecting irreducible components of $f$ that are also components of $\Delta$. This is impossible.
\end{proof}

\begin{lemma}\label{lem:generic fibre singular conic fibration}
Let $\pi\colon X\to S$ be a standard conic bundle. 
Suppose that there is a birational morphism $\eta\colon S\to S'$ to a ruled surface $\tau\colon S'\to C$ above a curve $C$ of genus $g(C)\geq1$. Let $\Delta$ be the discriminant of $\pi$ (which could be empty), let $I\Delta$ be the union of irrational irreducible components of $\Delta$ and let $\Autz(S,\Delta)$ be the subgroup of $\Autz(S)$ that preserves $\Delta$. 
Then the following hold:
\begin{enumerate}
\item\label{gfscf:0} If $\Delta\neq\varnothing$, then $I\Delta$ is non-zero and intersects a general fibre of $\tau\circ\eta$ in at least two distinct points. 
\item\label{gfscf:1} The morphism $\hat\pi\colon X_{\k(C)}\to S_{\k(C)}\simeq S'_{\k(C)}$ induced by the fibration $\tau\circ\eta\circ\pi$ is a standard conic bundle above the rational curve $S_{\k(C)}$ with discriminant $(I\Delta)_{\k(C)}$, which is empty if and only if $\Delta=\emptyset$.
\item\label{gfscf:3} The surface $X_{\k(C)}$ is a Del Pezzo surface or a Hirzebruch surface.
\item\label{gfscf:2} If $\Delta\neq\emptyset$, then the group $\Autz(S,\Delta)_C$ is a torus of dimension$\leq2$. 
\end{enumerate}
\end{lemma}
\begin{proof}
\autoref{gfscf:0} and \autoref{gfscf:1}: 
Since $X$ and $S$ are smooth and $\mathrm{char}(\k)=0$, also $X_{\k(C)}$ and $S_{\k(C)}$ are smooth.
The surface $S_{\k(C)}$ is rational, because $\tau$ is a ruled surface.
Suppose that $\Delta=\emptyset$. Then $I\Delta=\emptyset$ and $\pi$ is a conic bundle by \autoref{lem:empty discriminant} and then $\hat\pi\colon X_{\k(C)}\to S_{\k(C)}$ is a conic bundle with $\dim X_{\k(C)}=2$ and hence has empty discriminant. Suppose that $\Delta\neq\emptyset$.
By \autoref{lem:components of Delta irrational 0},  $\Delta$ contains at least one irrational irreducible component. 
By \autoref{lem:components of Delta irrational 0}, $I\Delta$ is non-zero and generically transversal to the fibration $\tau\circ\eta$.
Let $F$ be the general fibre of $\tau\circ\eta\circ\pi$. It is a smooth surface and $\eta\circ\pi|_F\colon F\to \eta(\pi(F))$ has generic fibre a plane conic and it has singular fibres, which are exactly above the points $F\cap I\Delta$. It follows that the generic fibre of $\hat\pi$ is a plane conic and that $(I\Delta)_{\k(C)}$ is the discriminant of $\hat\pi$. Since $\dim X_{\k(C)}=2$ and $(I\Delta)_{\k(C)}\neq\emptyset$, it follows that the generic fibre of $\hat\pi$ is irrational.
Finally, we have $\rho(X_{\k(C)}/S_{\k(C)})=1$, because $\rho(X/S)=1$. It follows that $I\Delta$ intersects a general fibre of $\tau\circ\eta$ in at least two distinct points. 

\autoref{gfscf:3} By \autoref{gfscf:1}, $\hat\pi\colon X_{\k(C)}\to S_{\k(C)}$ is a rational $2$-dimensional Mori fibre space. This yields the claim.

\autoref{gfscf:2}
By \autoref{gfscf:0}, $I\Delta$ is non-zero and by \autoref{lem:components of Delta irrational 0} it is generically transversal to the fibres of $\tau\circ\eta$. 
Therefore,  
$\Aut(S,\Delta)_C$ preserves at least two distinct points on every general fibre of $\tau\circ\eta$. 
It follows from \cite{Maruyama} that $\Autz(S,\Delta)_C$ is a torus of dimension$\leq2$. 
\end{proof}

The following two lemmas describe birational maps from a standard conic fibration above a surface to another Mori fibre space and from a Mori Del Pezzo fibration to a Mori fibre space above a surface.
We denote by $\kappa(S)$ the Kodaira dimension of $S$. 

\begin{lemma}\label{lem:irrat generic fibre square}
Let $\pi\colon X\to S$ be a standard conic bundle above a surface $S$ and let $\varphi\colon X\rat Y$ be a birational map to a Mori fibre space $\pi'\colon Y\to B$ above a surface $B$. 
\begin{enumerate}
\item\label{square:2} 
If $\kappa(S)\geq0$, then there is a birational map $\theta\colon S\rat B$ such that $\theta\circ\pi=\pi'\circ\varphi$.
If moreover $\varphi$ is $\Autz(X)$-equivariant, then so is $\theta$. 
\item\label{square:1}
Suppose that $S$ is a ruled surface $\tau\colon S\to C$ above a curve of genus $g(C)\geq1$.   
Then there exists a morphism $\alpha\colon B\to C$ and birational map $\eta\colon S\rat B$ such that $\eta\circ\pi=\pi'\circ\varphi$ and $\alpha\circ\eta=\tau$. If moreover $\varphi$ is $\Autz(X)$-equivariant and $\Autz(X)$ acts non-trivially on $C$, then $\eta$ an isomorphism. 
\end{enumerate}
\end{lemma}

\begin{proof}
\autoref{square:2} Let $D$ be a general fibre of $\pi'$. It is rational, its image $D'\subset X$ by $\varphi^{-1}$ is rational. Since $S$ is not covered by rational curves, $\pi(D')$ is a point. This yields the claim.

\autoref{square:1}
Let $f$ be a general fibre of $\pi'$. It is a smooth rational curve and we call $f'\subset X$ its image by $\varphi^{-1}$. Since $f'$ is rational, $\tau(\pi(f'))$ is a point. Thus, there is a morphism $\alpha\colon B\to C$ such that $\alpha\circ\pi'\circ\varphi=\tau\circ\pi$. 
Then $\alpha$ has connected fibres. 
Let us show that the general fibre of $\alpha$ is rational. 
Since $B$ is a surface and the general fibre of $\pi'$ is a rational curve, the morphism $\pi'$ induces a Mori fibre space $\hat\pi'\colon Y_{\k(C)}\to B_{\k(C)}$. 
 By \autoref{lem:generic fibre singular conic fibration}\autoref{gfscf:3}, $X_{\k(C)}$ is rational.
The map $\hat\varphi\colon X_{\k(C)} \dashrightarrow Y_{\k(C)}$ induced by $\varphi$ is birational. So, $Y_{\k(C)}$ is rational and hence the curve $B_{\k(C)}$ is rational. It follows that the general fibre of $\alpha$ is rational. 
We now construct $\eta$. 
Let $h\subset B$ be a general fibre of $\alpha$ and $F\subset Y$ the fibre by $\pi'$ above $h$. Since $h$ is rational, $F$ is rational. Then $(\varphi^{-1})(F)\subset X$ is a rational surface. Then its image $\pi((\varphi^{-1})(F))=:h'\subset S$ in $S$ is a rational curve. Since $\tau\colon S\to C$ is a ruled surface and $g(C)\geq1$, $h'$ is contained in a fibre of $\tau$.
Since $h$ is general, $h'$ is a general fibre of $\tau$. This defines a birational map $\eta\colon S\rat B$ such that $\eta\circ\pi=\pi'\circ\varphi$ and such that $\alpha\circ\eta=\tau$. 

Suppose that $\varphi$ is $\Autz(X)$-equivariant and that $\Autz(X)$ acts non-trivially on $C$. Then every fibre of $\alpha$ is isomorphic to $\p^1$. Thus $\alpha$ is a ruled surface.
The base-locus of $\alpha$ is $\Autz(X)$-invariant and hence empty. 
It follows that $\eta$ is an isomorphism.  
\end{proof}

\begin{lemma}\label{lem:commutative diag surface-curve}
Let $\pi\colon X\to C$ be a Mori Del Pezzo fibration above a curve of genus $g(C)\geq1$ and $\varphi\colon X\rat X'$ a birational map to a Mori fibration $\pi'\colon X'\to S$ above a surface $S$. Then there is a morphism $\tau\colon S\to C$ with $\tau_*\Ol_S=\Ol_C$ such that $\tau\circ\pi'\circ\varphi=\pi$.  
\end{lemma}

\begin{proof}
We have $\kappa(S)=-\infty$ by \autoref{lem:irrat generic fibre square}\autoref{square:2}.
Since $X$ is not rationally connected, $S$ is not rationally connected by \autoref{thm:rationally connected}. So, there is a birational map $\eta\colon S\rat S'$ to a ruled surface $\theta\colon S'\to D$ for some irrational curve $D$. 
Let $(p,q)\colon Z\to S\times S'$ be a minimal resolution of $\eta$ such that $Z$ is smooth. Then $q$ is a sequence of blow-up of points. Since $S$ is klt \cite[Corollary 4.6]{Fuj99}, its singularities are isolated quotient singularities \cite[Proposition 9.3]{GKKP11}. Thus $p$ contracts only rational curves. Since $D$ is irrational, the curve contracted by $p$ are contained in the fibres of $\theta\circ q$. Therefore, there is a morphism $\tau\colon S\to D$ such that $\tau\circ p=\theta\circ q$. In particular, $\tau_*\Ol_S=\Ol_D$. 

Let $F$ be a fibre of $\pi$. It is a Del Pezzo surface and hence rational, so $\tau(\pi'(\varphi(F)))$ is a point. Therefore, there is a morphism $\alpha\colon C\to D$ such that $\tau\circ\pi'\circ\varphi=\alpha\circ\pi$. In particular, $\alpha$ has connected fibres and is hence an isomorphism.
\end{proof}

Finally, we limit the options for birational maps from a standard conic bundle to a Mori Del Pezzo fibration when the vertical automorphism group is big.

\begin{lemma}\label{lem:no_map_to_lower_degree.2}
    Let $\pi\colon X\to S$ be a Mori fibre space over a surface $S$ with a morphism $\tau\colon S\to C$ to a curve $C$ of genus $g(C)\geq1$ such that $\tau_*\Ol_S=\Ol_C$. 
    Assume that $\Autz(X)$ acts transitively onto $C$ and $\dim\Autz(X)_C>0$. Then there is no $\Autz(X)$-equivariant birational map from $X$ to Mori Del Pezzo fibre space of degree $d\leq 8$. 
\end{lemma} 

\begin{proof}
Let $\varphi\colon X\rat Y$ be a birational map to a Mori Del Pezzo fibration $\pi'\colon Y\to B$.  
By \autoref{lem:irrat generic fibre square}, there is an isomorphism $\alpha\colon C\to B$ such that $\pi'\circ\varphi=\alpha\circ\tau\circ\pi$. 
We can assume that $\alpha$ is the identity map. 
Then $\varphi\Autz(X)_C\varphi^{-1}\subseteq\Autz(Y)_C$ and therefore $\dim\Autz(Y)_C>0$. 
By \autoref{rmk: MdP}\autoref{MdP:2}, there is an injective group homomorphism $\Autz(Y)_C\hookrightarrow\Autz(Y_{\k(C)})$. \autoref{lem:auto of dp surface}\autoref{auto dp surface:1} implies that $d:=K_{Y_{\k(C)}}^2\geq6$. 
\autoref{prop:Mdf 6 finite} implies that $d\neq6$. 
\autoref{rmk: MdP}\eqref{MdP:1} implies that $d\neq7$. 
\autoref{lem:dP8 Aut on C} implies that $d\neq8$.
\end{proof}

\subsection{When the generic fibre is irrational}

\begin{proposition}[{\cite[Proposition 3.2.3]{BFT22}}]\label{prop:generic fibre not P1}
Let $\pi \colon X \to S$ be a morphism of algebraic varieties whose generic fibre is a smooth conic in $\mathbb{P}^2_{k(S)}$, not isomorphic to $\mathbb{P}^1_{k(S)}$. 
If $G$ is an algebraic subgroup of $\mathrm{Aut}^{\circ}(X)_S$, then $G$ is a finite group isomorphic to $(\mathbb{Z}/2\mathbb{Z})^r$ for some $r \in \{0,1,2\}$.
\end{proposition}

\begin{proof}[Proof of \autoref{thm:Mfs K(S)-pos generic curve irrational}]
By \autoref{lem:commutative diag surface-curve}, $X$ is not birational to a Mori Del Pezzo fibration. We now show that $\Autz(X)$ is trivial or an elliptic curve, and in the latter case that it is a maximal connected algebraic subgroup of $\Bir(X)$. 
By \autoref{lem:blanchard} (Blanchard's lemma), the morphism $\pi$ induces an exact sequence
\[
1\to\Autz(X)_S\to\Autz(X)\stackrel{\pi_*}\to\Autz(S)
\]
Since $X_{\k(S)}$ is irrational, \autoref{prop:generic fibre not P1} implies that, $\Autz(X)_S$ is isomorphic to $(\Z/2)^r$ for $r\leq2$.
By Chevalley's theorem, there is a short exact sequences
\[
1 \to L \to \Autz(S) \to A \to 1,
\]
where $L$ is an affine algebraic group and $A$ is an abelian variety. By a result of Rosenlicht (see \cite[Proposition 2.5.1]{BFT22}), and since $\kappa(S)\geq0$, the linear group $L$ is trivial. Therefore, $\Autz(S)$ is an abelian variety and hence $\Autz(X)$ is an abelian variety.
Moreover, for every point $x\in S$, only finitely many elements of $\Autz(S)$ fix $x$ \cite[Proposition 2.2.1]{BSU13}. Therefore, $\dim\Autz(X)=\dim\Autz(S)$ is the dimension of any orbit of $\Autz(S)$ on $S$, which is of dimension at most two. However, since $X_{\k(S)}$ is irrational, $\pi$ has singular fibres along a curve in $S$ by \autoref{prop:reduction to standard conic bundles}. It is preserved by $\pi_*\Autz(X)$. 
Recall that $\Autz(S)$ is an abelian surface if and only if $S$ is an abelian surface, see for instance \cite[Proposition 3.24]{Fong20}.
So, if $\pi_*\Autz(X)$ has dimension two, then $\pi_*\Autz(X)=\Autz(S)\simeq S$ acts on $S$ without invariant curves, which is impossible. So, $\Autz(X)$ is trivial or an elliptic curve.

Let $\varphi\colon X\rat Y$ be an $\Autz(X)$-equivariant birational map to another Mori fibre space $Y\to B$. 
By \autoref{lem:commutative diag surface-curve}, $B$ is a surface and there is an $\pi_*\Autz(X)$-equivariant birational map $\theta\colon S\rat B$ such that $\theta\circ\pi=\pi'\circ\varphi$. 
In particular, the generic fibre $Y_{\k(B)}$ is a smooth irrational curve of genus zero. Hence $\Autz(X)$ and $\Autz(Y)$ are elliptic curves and hence $\varphi\Autz(X)\varphi^{-1}=\Autz(Y)$.
\end{proof}

\begin{lemma}\label{lem:generic fibre DP surface}
Let $\pi\colon X\to S$ be a standard conic bundle whose generic fibre is irrational. 
Suppose that there a birational morphism $\eta\colon S\to S'$ to a ruled surface $\theta\colon S'\to C$ above a curve $C$ of genus $g(C)\geq1$. 
Then the following holds:
\begin{enumerate}
\item $K_{X_{\k(C)}}^2\neq 7,8$;
\item  if $K_{X_{\k(C)}}^2\in\{3,5,6\}$, then $X_{\k(C)}$ is Del Pezzo;
\end{enumerate}
\end{lemma}
\begin{proof}
By \autoref{lem:generic fibre singular conic fibration}, $\pi$ induces a standard conic bundle $\hat\pi\colon X_{\k(C)}\to S_{\k(C)}$ with $\rho(X_{\k(C)})=2$ and non-zero discriminant. 
If $K_{X_{\k(C)}}^2=7$, then $\hat\pi$ has a unique singular fibre above a rational point \cite[Theorem 3(iii)]{Isko_1979}. Therefore, $\rho(X_{\k(C)}/S_{\k(C)})\geq3$, a contradiction. 
If $K_{X_{\k(C)}}^2=8$, then $\hat\pi$ has no singular fibres \cite[Theorem 3(ii)\&(iii)]{Isko_1979}, against the fact that $\hat\pi$ has non-zero discriminant. 
The rest is \cite[Theorem 5]{Isko_1979}. 
\end{proof}

Let $\pi\colon X\to S$ be a standard conic bundle above a surface $S$ with a birational morphism $\eta\colon S\to S'$ to a ruled surface $\tau\colon S'\to C$ above an irrational curve. Let $\Delta$ be the discriminant of $\pi$ and $\Autz(S,\Delta)$ the group of automorphims in $\Autz(S)$ preserving $\Delta$.
Blanchard's lemma (\autoref{prop:generic fibre not P1}) and the morphisms $\pi$ and $\tau\circ\eta$ induce the commutative diagram of exact sequences
\begin{equation}\tag{sesrs}\label{eq:sesrs}
\begin{tikzcd}
&& 1\ar[d] && 1\ar[d]\\
&& \Autz(X)_C\ar[d]\ar[rr] && \Autz(S,\Delta)_C\ar[d]\\
1\ar[r] & (\Z/2)^r\simeq \Autz(X)_S\ar[r]  & \Autz(X)\ar[rr]\ar[dr] && \Autz(S,\Delta)\ar[ld]\\
&&&\Autz(C)&
\end{tikzcd}
\end{equation}
for some $i=0,1,2$.

\begin{lemma}\label{lem:to ruled surface}
Let $\pi\colon X\to S$ be a standard conic bundle whose generic fibre is irrational. Suppose there is a birational morphism $\eta\colon S\to S'$ to a ruled surface $\tau\colon S'\to C$ above a curve $C$ of genus $g(C)\geq1$. 
Suppose that $\Autz(X)$ acts trivially on $C$. 
Then $\Autz(X)$ is not a maximal connected algebraic subgroup of $\Bir(X)$. 
\end{lemma}

\begin{proof}
We denote by $\Delta$ the discriminant of $\pi$.
Consider the commutative diagram of short exact sequences \autoref{eq:sesrs}. 
By \autoref{lem:generic fibre singular conic fibration}\autoref{gfscf:2}, $\Autz(S,\Delta)=\Aut(S,\Delta)_C$ is a torus. 
By \autoref{lem:generic fibre singular conic fibration}\autoref{gfscf:0}, $\Aut(S,\Delta)$ preserves $a\geq2$ points on a general fibre of $\tau\circ\eta$. 
If $a\geq3$, then $\Autz(S,\Delta)$ is trivial and hence $\Autz(X)$ is trivial by \autoref{eq:sesrs}. Then it is not maximal. Suppose that $a=2$. 
Then $X_{\k(C)}$ is a Del Pezzo surface of degree $6$ by \autoref{lem:generic fibre DP surface}. 
The morphism $\hat\pi\colon X_{\k(C)}\to S_{\k(C)}$ induced by $\pi$ is a standard conic bundle by \autoref{lem:generic fibre singular conic fibration}, so $\rho(X_{\k(C)})=2$. 
So, there is a birational contraction $\varphi\colon X_{\k(C)}\to\mathcal Y$ to a Del Pezzo surface of degree $8$. 
It induces is a Mori Del Pezzo fibration $\pi'\colon Y\to C$ with $Y_{\k(C)}\simeq\mathcal Y$ and 
an $\Autz(X)_C$-equivariant birational map $\varphi'\colon X\rat Y$ such that $\pi'\circ\varphi'=\pi$. 
Recall from \autoref{rmk: MdP}\autoref{MdP:2} that there is an injective group homomorphism $\Autz(X)=\Autz(X)_C\hookrightarrow\Aut(X_{\k(C)})$. The birational contraction $\varphi$ is $\Aut(X_{\k(C)})$-equivariant, (because it contracts the unique extremal ray corresponding to a birational map). Therefore, $\varphi'$ is $\Autz(X)_C$-equivariant. 
The group $\Autz(X)=\Autz(X)_C$ preserves the sections corresponding to the base-points of $\varphi^{-1}$. 
\autoref{lem:dP8 Aut on C} implies that $\Autz(X)\subsetneq\Autz(Y)$.
\end{proof}

\begin{remark}\label{rmk:igf nontrivial action}
\autoref{lem:to ruled surface} applies in particular when one of the following holds:
\begin{enumerate}
\item $g(C)\geq2$;
\item $g(C)=1$ and $\eta$ is not an isomorphism;
\item $g(C)=1$ and the discriminant $\Delta$ of $\pi$ has a rational irreducible component (\autoref{lem:components of Delta irrational 0});
\item $g(C)=1$ and $\Delta$ has an irreducible component $D$ such that $(\tau\circ\eta)|_D\colon D\to C$ is ramified;
\item $g(C)=1$ and $\Delta$ has two irreducible components that intersect.
\end{enumerate}
\end{remark}

\begin{lemma}\label{lem:scf 2 options}
Let $\pi\colon X\to S$ be a standard conic bundle whose generic fibre is irrational. Suppose that $S$ is a ruled surface $\tau\colon S\to C$ above a curve of genus $g(C)=1$. 
The following hold:
    \begin{enumerate}
    \item\label{scf 2 options:2} If $X_{\k(C)}$ is Del Pezzo of degree $K_{X_{\k(C)}}^2=3,5,6$, then there is a birational map $X\rat Y$ to a Mori Del Pezzo fibre space $\pi'\colon Y\to C$. 
    \item\label{scf 2 options:1} Else, every birational map $X\rat Y$ to a Mori fibre space $\pi'\colon Y\to B$ satisfies that $B$ is a surface.
    \end{enumerate}
\end{lemma}

\begin{proof}
\autoref{scf 2 options:2} 
If $X_{\k(C)}$ is Del Pezzo of degree $3,5,6$, then there is a birational morphism $X_{\k(C)}\to\mathcal Y$ to a Del Pezzo surface of degree $K_{\mathcal Y}^2<K_{X_{\k(C)}}^2$ \cite[Theorem 4]{Isko_1979}. It induces a birational map $X\rat Y$ to a Mori Del Pezzo fibre space $Y\to C$ with $Y_{\k(C)}\simeq \mathcal Y$.

\autoref{scf 2 options:1} 
Suppose there is a birational map $\varphi\colon X\rat Y$ to a Mori Del Pezzo fibre space $\pi'\colon Y\to B$. By \autoref{lem:irrat generic fibre square}, there is an isomorphism $\alpha\colon C\to B$ such that $\pi'\circ\varphi=\alpha\circ\pi$. In particular, there is a birational map $\hat\varphi\colon X_{\k(C)}\rat Y_{\k(C)}$. \cite[Theorem 2.6]{Isko_1979} implies that $X_{\k(C)}$ is a Del Pezzo surface. Let $d=K_{X_{\k(C)}}^2$ be its degree. 
By assumption, we have $d\neq3,5,6$. By \autoref{lem:generic fibre DP surface} we have $d\neq7,8$. 
Then any birational morphism $X_{\k(C)}\to \mathcal Y$ to a Del Pezzo surface $\mathcal Y$ is an isomorphism \cite[Theorem 4]{Isko_1979}. This holds in particular for $\mathcal Y=Y_{\k(C)}$ and contradicts $\rho(Y_{\k(C)})=1$.
\end{proof}

\begin{lemma}\label{lem:igf 9}
Let $\pi\colon X\to S$ be a standard conic bundle whose generic fibre is irrational. Suppose that $S$ is a ruled surface $\tau\colon S\to C$ and that $\Autz(X)$ acts non-trivially on $C$. 
Let $\varphi\colon X\rat Y$ be an $\Autz(X)$-equivariant birational map to a Mori Del Pezzo fibration $\pi'\colon Y\to B$. 
Then $K_{Y_{\k(C)}}^2\in\{4,9\}$. 

Moreover, if $K_{Y_{\k(C)}}^2=9$, then $X_{\k(C)}$ is Del Pezzo of degree $5$ and $\Autz(X)$ is an elliptic curve and is not a maximal connected algebraic subgroup of $\Bir(X)$.
\end{lemma}
\begin{proof}
By \autoref{lem:generic fibre singular conic fibration}, $\pi$ induces a standard conic bundle $\hat\pi\colon X_{\k(C)}\to S_{\k(C)}$. By \autoref{lem:irrat generic fibre square}, $X_{\k(C)}$ is a Del Pezzo surface and it is not a Hirzebruch surface and the map $\varphi$ induces a birational map $X_{\k(C)}\rat Y_{\k(C)}$. 
Since $\rho(Y_{\k(C)})=1$ and $X_{\k(C)}\to S_{\k(C)}$ is a Mori fibre space by \autoref{lem:generic fibre singular conic fibration}, \autoref{lem:rho=1,K<4} (with $G=\{1\}$) implies that $K_{Y_{\k(C)}}^2\geq4$.  
Let $d=K_{Y_{\k(B)}}^2$ be the degree of $\pi'$. 
Since $\Autz(X)$ acts non-trivially on $C$ and since $\varphi$ is $\Autz(X)$-equivariant, we have $d\neq5,6,8$ by \autoref{lem:dP8 Aut on C}, \autoref{pro:DP5} and \autoref{prop:Mdf 6 finite}. So, $d=4$ or $d=9$.

Suppose that $d=9$. 
By \autoref{lem:generic fibre singular conic fibration}, $\pi$ induces a standard conic bundle $\hat\pi\colon X_{\k(C)}\to S_{\k(C)}$. The map $\varphi$ induces a birational map $\hat\varphi\colon X_{\k(C)}\rat Y_{\k(C)}$. 
Since $d=K_{Y_{\k(C)}}^2=9$, the surface $X_{\k(C)}$ is Del Pezzo of degree $K_{X_{\k(C)}}^2\in\{5,6,8\}$ \cite[Theorem 2.6]{iskovskikh_1996}. 
We want to show that $K_{X_{\k(C)}}^2=5$.
If $K_{X_{\k(C)}}^2=8$, then $X_{\k(C)}$ is a Hirzebruch surface, which contradicts $\hat\pi$ having non-zero discriminant by  \autoref{lem:generic fibre singular conic fibration}. 
If $K_{X_{\k(C)}}^2=6$, then there are birational maps $\hat\psi_1\colon X_{\k(C)}\rat\mathcal Z$ and $\hat\psi_2\colon \mathcal Z\rat Y_{\k(C)}$, where $\mathcal Z$ is a Del Pezzo surface $\mathcal Z$ of degree $8$, such that $\hat\varphi=\hat\psi_2\circ \hat\psi_1$ \cite[Theorem 2.6(iii)]{iskovskikh_1996}. 
Then $\hat\psi_1$ corresponds to a birational map $\psi_1\colon X\rat Z$ to a Mori Del Pezzo fibre space $Z\to C$ of degree $8$ with $Z_{\k(C)}\simeq\mathcal Z$. By \autoref{lem:dP8 Aut on C}, $\Autz(X)$ acts trivially on $C$, in contradiction with our hypothesis. 
We have shown that $K_{X_{\k(C)}}^2=5$.

Let $\varphi=\varphi_n\circ\cdots\circ\varphi_1$ be a decomposition of $\varphi$ into $\Autz(X)$-equivariant Sarkisov links $\varphi_i\colon X_i\rat X_{i+1}$ over $C$ \cite{Flo20}, and denote by $\pi_i\colon X_i\to B_i$ the involved Mori fibre spaces. 
Let $j$ be the smallest index so that $\varphi_j$ is not a link between Mori fibre spaces whose generic fibre is a curve. 
By \autoref{lem:irrat generic fibre square}, there is an isomorphism $\eta_j\colon S\to B_j$ such that $\pi_j\circ\varphi_j\circ\cdots\circ\varphi_1=\eta\circ\pi$. 
\[
\begin{tikzcd}
X\ar[d,"\pi"]\ar[rr,dashed,"\varphi_{j-1}\circ\cdots\circ\varphi_1"] && X_j\ar[d,"\pi_j"]\ar[r,dashed,"\varphi_j"] & X_{j+1}\ar[d,"\pi_{j+1}"]\\
S\ar[drr,"\tau"]\ar[rr,"\eta_j\ \simeq"] && B_j\ar[d] & B_{j+1}\ar[ld]\\
&&C&
\end{tikzcd}
\]
Then $\varphi_j$ is a link of type III. 
\autoref{lem:commutative diag surface-curve} applied to $\varphi_j\circ\cdots\circ\varphi_1$ implies that $K_{(X_{j+1})_{\k(C)}}^2=9$. 
By \autoref{lem:generic fibre singular conic fibration}, $\hat\pi_j\colon (X_j)_{\k(B_j)}\simeq X_{\k(B_j)}\to (B_j)_{\k(C)}$ is a standard conic bundle with non-zero discriminant.  
\cite[Theorem 2.6(i)]{iskovskikh_1996} implies that $K_{(X_j)_{\k(C)}}^2=5$. 
Let $D$ be the base-locus of $\varphi_j^{-1}$.
\autoref{lem:links I} implies that  and that $\pi_{j+1}(D)\subset S$ is a curve such that $\tau|_{\pi_{j+1}(D)}$ is a $4:1$-cover. 
Let $\Delta$ be the discriminant of $\pi$.
Thus $\Autz(S,\Delta)_C$ fixes four points on each fibre of $\tau\colon S\to C$. \autoref{lem:generic fibre singular conic fibration} implies that it is trivial. 
Then \autoref{eq:sesrs} and the assumption  that $\Autz(X)$ acts non-trivially on $C$ imply that $\Autz(S,\Delta)$ is an elliptic curve. Then \autoref{eq:sesrs} and \autoref{prop:generic fibre not P1} imply that $\Autz(X_j)$ is an elliptic curve. It follows that $(\varphi_j\circ\cdots\circ\varphi_1)\Autz(X)(\varphi_j\circ\cdots\circ\varphi_1)^{-1}=\Autz(X_j)$. In particular, it suffices to show that $\Autz(X_j)$ is not maximal.  

By \autoref{thm:Tsen}, $X_{\k(C)}$ has a rational point, so $(X_{j+1})_{\k(C)}\simeq\p^2_{\k(C)}$. 
By \autoref{prop:Mdp to bundle}, there is an $\Autz(X_{j+1})$-equivariant birational map $\psi\colon X_{j+1}\rat Z$ over $C$ to a $\p^2$-bundle $\pi''\colon Z\to C$. Then $(\psi\varphi_j)\Autz(X_j)(\psi\varphi_j)^{-1}\subseteq\Autz(Z)$. We now show that either this inclusion is strict or that $\Autz(Z)$ is not maximal. This will imply our claim.

If $\pi''$ is semi-homogeneous, then $\Autz(Z)$ is not maximal by \autoref{prop:trivial_action_on_C}. 
If $\pi''$ is semi-homogeneous, then by \autoref{lem:hom_proj_bundles}, we have  
\begin{enumerate}
\item $Z\simeq\Al_{3,i}$, $i=0,1,2$
\item $Z\simeq\p(\El_{2,0}\oplus\Ml)$ or $Z\simeq \p(\Ol_C\oplus\Ol_C\oplus\Ml)$, 
\item $Z\simeq \p(\Ol_C\oplus\Ll\oplus\Ml)$, where $\Ll,\Ml\in\Pic(C)^0$ not isomorphic
\end{enumerate}

(1) If $Z\simeq \Al_{3,0}$, then $\Autz(Z)$ is not maximal by \autoref{pro:E30}. 
If $Z\simeq\Al_{3,1}$ or $Z\simeq \Al_{3,2}$, then \autoref{prop:A31 and A32} implies that all $\Autz(Z)$-orbits of dimension one are curves in $Z$ such that $\pi''$ induces $9:1$-covers of $C$. Therefore, $(\psi\varphi_j)^{-1}$ is not $\Autz(Z)$-equivariant. 

(2) If $Z\simeq \p(\El_{2,0}\oplus\Ml)$ or $Z\simeq\p(\Ol_C\oplus\Ol_C\oplus\Ml)$ for $\Ml\in\Pic^0(Z)\setminus\{\Ol_C\}$, then $(\psi\varphi_j)^{-1}$ is not $\Autz(Z)$-equivariant by \autoref{lem:degree_0_L_trivial}. If $Z\simeq\p(\El_{2,0}\oplus\Ol_C)$, then $\Autz(Z)$ is not a maximal connected algebraic subgroup of $\Bir(Z)$ by \autoref{lem:E_20+O}.

(3) If $Z\simeq \p(\Ol_C\oplus\Ll\oplus\Ml)$ for $\Ll\nsimeq\Ml$, then $\mathbb{G}_m\subset\Autz(Z)$ by \autoref{lem:O+L+M}, so $(\psi\varphi_j)\Autz(X_j)(\psi\varphi_j)^{-1}\subsetneq\Autz(Z)$.
\end{proof}

\begin{proposition}\label{pro: igfrs}
Let $\pi\colon X\to S$ be a standard conic bundle whose generic fibre is irrational. Suppose that $S$ is a ruled surface $\tau\colon S\to C$ above a curve of genus $g(C)=1$. Suppose that $\Autz(X)$ acts non-trivially on $C$.
Then $\Autz(X)$ is an elliptic curve and the following hold:
\begin{enumerate}
\item\label{igfrs:1} If $K_{X_{\k(C)}}^2\leq4$, then $\Autz(X)$ is a maximal connected algebraic subgroup of $\Bir(X)$.
\item\label{igfrs:3} If $K_{X_{\k(C)}}^2\geq5$, then $\Autz(X)$ is not a maximal connected algebraic subgroup of $\Bir(X)$.  
\end{enumerate}
\end{proposition}

\begin{proof}
By \autoref{lem:generic fibre singular conic fibration}, $\pi$ induces a standard conic bundle $\hat\pi\colon X_{\k(C)}\to S_{\k(C)}$ with non-zero discriminant. 

\autoref{igfrs:3} 
By \autoref{lem:generic fibre DP surface}, we have $K_{X_{\k(C)}}^2\neq7,8$.
If $K_{X_{\k(C)}}^2=5,6$, then $X_{\k(C)}$ is a Del Pezzo surface \cite[Theorem 5]{Isko_1979}. 
\autoref{lem:scf 2 options}\autoref{scf 2 options:2} and \autoref{lem:igf 9} imply that $K_{X_{\k(C)}}^2=5$ and that $\Autz(X)$ is an elliptic curve and not a maximal connected algebraic subgroup of $\Bir(X)$. 

\autoref{igfrs:1} If $K_{X_{\k(C)}}^2=3$, then $X_{\k(C)}$ is a Del Pezzo surface \cite[Theorem 5]{Isko_1979}. Then the claim is from \autoref{pro: DP4}. 

Suppose that $K_{X_{\k(C)}}^2\leq2$ or $K_{X_{\k(C)}}^2=4$. By \autoref{lem:scf 2 options}, every $\Autz(X)$-equivariant birational  map $\varphi\colon X\rat Y$ to a Mori fibre space $\pi\colon Y\to B$ satisfies that $B$ is a surface. By \autoref{lem:irrat generic fibre square}, there is an isomorphism $\eta\colon S\to B$. We replace the Mori fibre space $\pi'\colon Y\to B$ by the Mori fibre space $\pi'':=\alpha^{-1}\circ\pi'\colon Y\to S$. 

We claim that $\Autz(Y)$ preserves the discriminant $\Delta$ of $\pi\colon X\to S$. 
Indeed, by the relative version of \cite{Flo20}, the map $\varphi$ is a composition $\varphi=\varphi_n\circ\cdots\circ\varphi_1$ of Sarkisov links $\varphi_i\colon X_{i-1}\rat X_i$ between Mori fibre spaces $\pi_i\colon X_i\to B_i$, where $\pi_0=\pi$ and $\pi_n=\pi'$, equipped with a morphism $B_i\to S$. Since $\dim S=2\geq\dim B_i$, we have $B_i=S$ and each $\varphi_i$ is a link of type II. 
\[
\begin{tikzcd}
X=X_0\ar[r,dashed,"\varphi_1"]\ar[drr,"\pi_0",swap] & X_1\ar[dr,"\pi_1"]\ar[r,dashed,"\varphi_2"] &\dots&X_{n-1}\ar[r,dashed,"\varphi_n"]\ar[dl,"\pi_{n-1}",swap] & X_n=Y\ar[dll,"\pi_n"]\\
&&S&&
\end{tikzcd}
\]
Let $\Gamma_i\subset X_i$ be the base-locus of $\varphi_i$ and $T_i\subset X_i$ the surface contracted by $\varphi_i$. 
By \cite[Lemma 3.23]{BLZ21}, $\pi_i(\Gamma_i)$ is outside the discriminant curve of $\pi_i$ and that $\pi_i(T_i)=\pi_i(\Gamma_i)$. It follows that there is an open $U\subset S$ such $U\cap \Delta$ is contained in the discriminant of $\pi''\colon Y\to S$. 
It follows that $\Autz(Y)$ preserves $\Delta$. 
Since $K_{X_{\k(C)}}^2=4$ or $K_{X_{\k(C)}}^2\leq2$, \autoref{lem:generic fibre singular conic fibration} implies that $\Delta$ intersects each general fibre of $\tau\colon S\to C$ in at least four points. Therefore, $\Autz(S,\Delta)_C$ is finite.  \autoref{eq:sesrs} implies that $\Autz(X)$ and $\Autz(Y)$ are elliptic curves. 
We have now shown that every $\Autz(X)$-equivariant birational map $X\rat Y$ to another Mori fibre space $Y\to B$ satisfies that $\Autz(Y)$ is an elliptic curve. It follows that $\Autz(X)$ is a maximal connected algebraic subgroup of $\Bir(X)$.
\end{proof}

\begin{proof}[Proof of \autoref{prop:iscf main}]
By \autoref{lem:to ruled surface} and \autoref{rmk:igf nontrivial action}, $\tau\colon S\to C$ is a ruled surface. We have $g(C)=1$ by \autoref{rmk:igf nontrivial action}. 
The rest of the claim is \autoref{pro: igfrs}.
\end{proof}

\subsection{When the generic fibre is rational}

Let $\pi\colon X\to S$ be a standard conic bundle above an irrational surface and with rational generic fibre $X_{\k(S)}$. 
By \autoref{lem:empty discriminant}, $\pi$ is a $\p^1$-bundle over $S$. 
With the following lemma, we can assume $S$ to be a minimal surface. 

\begin{lemma}[{Descent lemma \cite[Lemma 2.3.2]{BFT23}}]\label{lem:descent}
Let $\eta\colon S'\to S$ be a birational morphism between two smooth projective surfaces.
Let $U'\subset S'$ and $U\subset S$ be maximal open sets such that $\eta$ induces an isomorphism $U'\stackrel{\simeq}\to U$ and $\Omega=S'\setminus U'$ is finite and let $\pi'\colon X'\to S'$ be a $\p^1$-bundle. 

Then there exists a $\p^1$-bundle $\pi\colon X\to S$ and a birational map $\psi\colon X'\rat X$ such that $\pi\psi=\eta\pi'$ and such that $\psi$ induces an isomorphism $\pi'^{-1}(U'^{-1})\stackrel{\simeq}\to\pi^{-1}(U)$. Moreover, $\psi$ is unique up to composition with isomorphisms of $\p^1$-bundles at the target and $\psi$ is $\Autz(X')$-equivariant.
\end{lemma}

Let $\pi\colon X\to S$ be a standard conic fibration and suppose that $S$ carries a structure of a ruled surface $\tau\colon S\to C$ over a smooth curve $C$. 
Since $\pi\circ\tau$ has empty discriminant (see \autoref{lem:empty discriminant}), $X_{\k(C)}$ is a smooth surface with a conic bundle structure. In particular, $X_{\k(C)}\simeq\F_{b,\k(C)}$ for some $b\geq0$. It follows that there is an open subset $U\subset C$ such that every fibre $(\tau\circ\pi)^{-1}(p)$, $p\in U$, is isomorphic to $\F_b$.

The next lemma shows that up to equivariant birational maps, we can assume that the integer $b$ does not depend on the point $(\tau\circ\pi)(F)$.

\begin{proposition}[{\cite[Proposition 3.3]{Fong23}, \cite[Proposition 3.2.2]{BFT23}}]\label{lem:jumping fibres}
Let $\pi\colon X\to S$ be a $\p^1$-bundle above a ruled surface $\tau\colon S\to C$ above a curve $C$ of genus $g(C)\geq0$. 
Then there exists $b\geq0$ and a $\F_b$-bundle $\tilde\pi\colon \tilde X\to C$ and an $\Autz(X)$-equivariant map $\varphi\colon X\rat\tilde X$ such that $\tilde\pi\circ\varphi=\tau\circ\pi$.
\end{proposition}

In \autoref{prop:P1-bundles}, we determine which relatively maximal automorphism groups appearing in \cite[Theorem A]{Fong23} is a maximal connected algebraic subgroup. Before that, let us recall that the ruled surface $\tau\colon \Al_{2,1} \to C$ can be obtained as a quotient $(C\times \p^1)/(\Z/2\Z)^2$, and recall the notion of non-trivial 2-divisor introduced in \cite{Fong23}.

\begin{lemma}[{\cite[Lemma 6.2, Definition 6.7]{Fong23}}]\label{lem:2-divisor}
    Let $C$ be an elliptic curve and $m_2\colon C\to C$ be the multiplication by two. Then the following hold:
    \begin{enumerate}
    \item There exists a cartesian square
    \[
    \begin{tikzcd}
    C\times \p^1 \ar[r,"q"] \ar[d,"\tau_1"] & \Al_{2,1} \ar[d,"\tau"] \\
    C \ar[r,"m_2"] & C,
    \end{tikzcd}
    \]
    where $\tau_1$ is the projection onto $C$ and the morphism $q$ is the quotient of $C\times \p^1$ by $(\Z/2\Z)^2$, acting by translations on $C$ as its two-torsion subgroup, and by $x\mapsto \pm x^{\pm 1}$ on $\p^1$.
    \item Let $\sigma\subset \Al_{2,1}$ be a minimal section. Then there exists $D_0\in \Pic(C)$ of degree two such that
    \[
    q^*\sigma \sim \sigma_0 + \tau_1^*(D_0),
    \]
    where $\sigma_0$ is a constant section of $C\times \p^1$.
    \end{enumerate}
    Moreover, we say that a divisor $D\in \Pic(C)$ is a \emph{non-trivial 2-divisor} if $m_2^*(D) - 2\deg(D)D_0$ is not trivial and has degree zero. 
\end{lemma}

\begin{proposition}\label{prop:P1-bundles}
Let $\pi\colon X\to S$ be a $\p^1$-bundle above a ruled surface $\tau\colon S\to C$ above a curve $C$ of genus $g(C)\geq1$. Suppose that every fibre of $\tau\circ\pi$ is isomorphic to $\F_b$. 
Then $\Autz(X)$ is a maximal connected algebraic subgroup of $\Bir(X)$ if and only if one of the following holds:
    \begin{enumerate}[(I)]
    \item\label{P1-bundles:I} If $g(C)=1$, then:
    \begin{enumerate}[(i)]
        \item\label{P1-bundles:i} $X\simeq S'\times \p^1$, where $S'$ is one of the following ruled surfaces:
        \begin{enumerate}[(a)]
            \item $C\times \p^1$,
            \item $\Al_{2,0}$,
            \item $\Al_{2,1}$,
            \item $\p(\Ol_C\oplus \Ll)$, where $\Ll\in \Pic^0(C)$ is non-trivial.
        \end{enumerate}

        \item\label{P1-bundles:ii} $X \simeq \Al_{2,1} \times_C \Al_{2,1}$.

        \item\label{P1-bundles:iii} There exists an $\Autz(X)$-equivariant birational map 
        \[
        \varphi\colon X\dashrightarrow \p(\Ol_C\oplus \Ll) \times_C \Al_{2,1},
        \]
        where $\Ll\in \Pic^0(C)$ is not two-torsion. Via this birational map, the groups $\Autz(X)$ and $\Autz(\p(\Ol_C\oplus \Ll) \times_C \Al_{2,1})$ are conjugate.

        \item\label{P1-bundles:iv} There exists an $\Autz(X)$-equivariant birational map 
        \[
        X\dashrightarrow \p(\Ol_C\oplus \Ll) \times_C \Al_{2,0},
        \]
        where $\Ll\in \Pic^0(C)$ has infinite order. Via this birational map, the groups $\Autz(X)$ and $\Autz(\p(\Ol_C\oplus \Ll) \times_C \Al_{2,0})$ are conjugate.

        \item\label{P1-bundles:v} There exists an $\Autz(X)$-equivariant birational map 
        \[
        X\dashrightarrow \p(\Ol_C\oplus \Ll) \times_C \p(\Ol_C\oplus \Ml) ,
        \]
        where $\Ll,\Ml\in \Pic^0(C)$ and for every $(n,m)\in \mathbb{Z}^2$ coprime, the line bundle $\Ll^{\otimes n} \otimes \Ml^{\otimes m}$ is not trivial. Via this  birational map, the groups $\Autz(X)$ and $\Autz(\p(\Ol_C\oplus \Ll) \times_C \p(\Ol_C\oplus \Ml))$ are conjugate.

        \item\label{P1-bundles:vi} $X\simeq \p(\Ol_{C\times \p^1}\oplus \Ol_{C\times \p^1}(b\sigma +\tau^*(D)))$ is a decomposable $\p^1$-bundle over $S=C\times \p^1$, where $\sigma$ a constant section in $C\times \p^1$, $b>0$, and $D\in \Pic^0(C)$. Moreover, $D$ is non-trivial if $b=1$.
        
        \item\label{P1-bundles:vii} There exists an $\Autz(X)$-equivariant birational map 
        \[
        X\dashrightarrow \p(\Ol_{\Al_{2,1}} \oplus \Ol_{\Al_{2,1}}(2\sigma +\tau^*(D))),
        \]
        where $D$ is a non-trivial $2$-divisor on $C$ (see \autoref{lem:2-divisor}), $\sigma$ is a minimal section of $\Al_{2,1}$. Via this birational map, $\Autz(X)$ and $\Autz(\p(\Ol_{\Al_{2,1}} \oplus \Ol_{\Al_{2,1}}(2\sigma +\tau^*(D))))$ are conjugate.
    \end{enumerate}
Moreover, in the above Cases \ref{P1-bundles:i}, \ref{P1-bundles:ii}, \ref{P1-bundles:iii}, \ref{P1-bundles:vi} provided that $b\neq 1$, and \ref{P1-bundles:vii}, there is no $\Autz(X)$-equivariant birational map to a Mori Del Pezzo fibre space. In Cases \ref{P1-bundles:iv}, \ref{P1-bundles:v}, \ref{P1-bundles:vi} with $b=1$, there exist $\Autz(X)$-equivariant birational maps to $\p^2$-bundles. 
    
    \item\label{P1-bundles:II} If $g(C) \geq 2$, then $X \simeq C\times \p^1 \times \p^1$ and $\pi$ is the trivial $\p^1$-bundle over $S=C\times \p^1$.
    \end{enumerate}
\end{proposition}

\begin{proof}
  Assume that $\Autz(X)$ is a maximal connected algebraic subgroup of $\Bir(X)$. Then $\Autz(X)$ is conjugate to one of the relatively maximal automorphism groups listed in \cite[Theorem A]{Fong23}.
  For each case of loc. cit, it remains to study the equivariant birational map to Mori Del Pezzo fibre space over $C$. In Case \ref{P1-bundles:II}, when $g(C)\geq 2$, it follows from loc. cit. that $X\simeq C\times \p^1 \times \p^1$ is the only $\p^1$-bundle over a ruled surface $S$ such that $\Autz(X)$ is maximal. Moreover, there are no $\Autz(X)$-invariant curves or points in $X$, so there are no $\Autz(X)$-equivariant birational maps to Mori Del Pezzo fibre spaces. This gives Case \ref{P1-bundles:II} of our Proposition. 

  From now on, assume that $g(C)=1$, and we examine the automorphism groups that appear in \cite[Theorem A(I)]{Fong23}. 

 \cite[Theorem A(i)\&(iii)]{Fong23}: by \cite[Corollary D]{Fong23}, the group $\Autz(X)$ is a maximal connected algebraic subgroup of $\Bir(X)$. This gives Cases \ref{P1-bundles:i} and \ref{P1-bundles:iii}. 

 \cite[Theorem A(ii)]{Fong23}: $\Autz(X)$ is not a maximal by \cite[Corollary D]{Fong23}. 

  \cite[Therorem A(iv)\&(viii)]{Fong23} and \cite[Theorem A(vii)]{Fong23} when $b\neq1$: by \cite[Corollary D]{Fong23}, there is no $\Autz(X)$-equivariant birational map to a $\p^2$-bundle. 
  Moreover, $\Autz(X)_C$ has positive dimension by \cite[Proposition B (iv),(vii),(viii)]{Fong23}.
  Hence, by \autoref{lem:no_map_to_lower_degree.2}, there is also no $\Autz(X)$-equivariant birational map to a Mori Del Pezzo fibre space of degree $\leq 8$. This shows Cases \ref{P1-bundles:iii}, \ref{P1-bundles:vi} when $b\neq 1$, and \ref{P1-bundles:vii}.

  \cite[Theorem A(vii)]{Fong23} when $b=1$: then 
  \[
  X\simeq \p(\Ol_{C\times \p^1}\oplus \Ol_{C\times \p^1}(\sigma +\tau^*(D))).
  \]
  Then, by \autoref{lem:F_b-bundles}, there exist trivializations of the $\F_1$-bundle $\tau\circ \pi$ of the form
  \[
  \begin{array}{ccc}
    U_l \times \F_1 & \dashrightarrow & U_k \times \F_1 \\
    (x,[y_0:y_1;z_0:z_1]) & \longmapsto & (x,[y_0:\lambda_{kl}(x)y_1;z_0:z_1]),
  \end{array}
  \] 
  where $\lambda_{kl}$ are the cocycles of the line bundle $\Ol_C(D)$. 
  The contraction $\F_1\to \p^2$, $[y_0:y_1;z_0:z_1] \mapsto [y_0z_0:y_0z_1:y_1]$ induces an $\Autz(X)$-equivariant birational morphism over $C$ from $X$ to the $\p^2$-bundle $X'\simeq \p(\Ol_C \oplus \Ol_C \oplus \Ol_C(D))$, and two cases arise:
  \begin{itemize}
    \item Either $D$ is trivial, then $X' \simeq C\times \p^2$. Then $\Autz(X')\simeq \Autz(C)\times \PGL_3(\k)$ acts on $X'$ with a single orbit, which implies that $\Autz(X)$ is not maximal.
    \item Or, $D$ is not trivial, then $\Autz(X)$ is maximal by \autoref{prop:E'=O+O} and the groups $\Autz(X)$ and $\Autz(X')$ are conjugate via the above contraction.
  \end{itemize}
   This concludes Case \ref{P1-bundles:vi}. 

    \cite[Theorem A(v)]{Fong23}: notice that $\p(\Ol_C\oplus \Ll) \times_C \Al_{2,0}$ and $\p(\Ol_{\Al_{2,0}} \oplus \tau^*(\Ll))$ are isomorphic as $\p^1$-bundle over $\Al_{2,0}$, where $\tau\colon \Al_{2,0}\to C$ is the structure morphism. Hence, by \autoref{thm:DP9 max}\autoref{DP9 max:2} and \autoref{thm:main_superrigidity} \ref{mainDP:1.3ii}, we get Case \ref{P1-bundles:iv}. 

    \cite[Theorem A(vi)]{Fong23}: 
    we proceed similarly as in case \cite[Theorem A(v)]{Fong23} and we use \autoref{thm:main_superrigidity} \ref{mainDP:1.3iii}. This gives Case \ref{P1-bundles:v}.
\end{proof}


\subsection{Proof of \autoref{thm:main}}\label{s:proof_main_thm}

\begin{lemma}\label{lem:Weil}
Let $G$ be a maximal connected algebraic subgroup of $\Bir(C\times\p^2)$. Then there exists a Mori fibre space $\pi\colon X\to B$ with a morphism $\tau\colon B\to C$ such that $\tau_*\Ol_B=\Ol_C$ and $X$ is endowed with a $G$-action and there is a $G$-equivariant birational map $\varphi\colon C\times\p^2\rat X$ and an isomorphism $\alpha\colon C\to C$ such that $\tau\circ\pi\circ\varphi=\alpha\circ pr_C$, where $pr_C$ is the projection onto $C$. 
\end{lemma}

\begin{proof}
Let $G$ be a maximal connected algebraic subgroup of $\Bir(C\times\p^2)$. By \cite[Theorem]{Weil_groups}, there is a $G$-variety $X_1$ and a $G$-equivariant birational map $\varphi_1\colon C\times\p^2\rat X_1$ such that $\varphi_1G\varphi_1^{-1}\subset\Autz(X_1)$. We can assume that $X_1$ is normal, because the $G$-action lifts to the normalisation of $X_1$. 
By \cite[Lemma 8]{Sumihiro75}, there exists a $G$-stable quasi-projective open subset $U\subset X_1$. 
By \cite[Theorem 2]{Brion-models}, there is a projective $G$-equivariant completion $X_2$ of $U$. 
After taking a $G$-equivariant desingularization \cite[Proposition 3.9.1]{Koll-sing}, we can assume that $X_2$ is smooth. 
Since $X_2$ is projective, the group $\Autz(X_2)$ is an algebraic group and hence $\varphi G\varphi^{-1}=\Autz(X_2)$.
We now run the Minimal Model Program (MMP) from $X_2$. By \cite[Remark 2.3]{Flo20}, the MMP is $G$-equivariant. 
Since the first projection $C\times\p^2\to C$ is a Mori fibre space, the MMP starting from $X_2$ terminates in a Mori fibre space $\pi\colon X\to B$. We have obtained a $G$-variety $X$ with a Mori fibration $\pi\colon X\to B$ and a $G$-equivariant birational map $\varphi\colon C\times\p^2\rat X$. 
By \autoref{lem:not Fano}, we have $\dim(B)\geq1$.

Suppose that $B$ is a curve. It is irrational, because $X$ is not rationally connected. Then the claim follows from \autoref{lem:preserves fibration}. 
If $B$ is a surface, we apply \autoref{lem:commutative diag surface-curve}.
\end{proof}

\begin{remark}
In \autoref{lem:Weil}, if $\dim B=1$, then $\tau$ is an isomorphism and $\tau\circ\pi\colon X\to C$ is a Mori Del Pezzo fibration. 
Suppose that $\dim B=2$. By \cite[Corollary 4.6]{Fuj99}, $B$ is klt and hence has isolated singularities \cite[Proposition 9.3]{GKKP11}. Let $\hat\eta\colon\hat B\to B$ be a desingularization. Then $(\tau\circ\hat\eta)_*\Ol_{\hat B}=\Ol_C$. 
Running an MMP over $C$ from $\hat B$ yields a birational morphism $\eta\colon \hat B\to S$ to a ruled surface $\tau'\colon S\to C$ such that $\tau=\tau'\circ\eta$. 
\end{remark}

\begin{proof}[Proof of \autoref{thm:main}]
By \autoref{lem:Weil}, there exists a Mori fibre space $\pi\colon X\to B$ with a morphism $\tau\colon B\to C$ such that $\tau_*\Ol_{B}=\Ol_C$ and $X$ is endowed with a $G$-action and there is a $G$-equivariant birational map $\varphi\colon C\times\p^2\rat X$ and an isomorphism $\alpha\colon C\to C$ such that $\tau\circ\pi\circ\varphi=\alpha\circ pr_C$, where $pr_C$ is the projection onto $C$. Since $\Autz(X)$ is a connected algebraic group, we have $\varphi G\varphi^{-1}=\Autz(X)$. Since $G$ is a maximal connected algebraic subgroup of $\Bir(C\times\p^2)$, which contains at least one non-trivial algebraic group, namely $\Autz(C)\times\PGL_3(\k)$, we have $\Autz(X)\neq\{1\}$.   

Suppose that $B$ is a curve. Then $\pi$ is a Mori Del Pezzo fibration and we can assume that $B=C$. By \autoref{prop:MdP5}, we have $K_{X_{\k(C)}}^2\geq5$. 
By \autoref{rmk: MdP}\autoref{MdP:1}, we have $K_{X_{\k(C)}}^2\neq7$ and by \autoref{pro:DP5} we have $K_{X_{\k(C)}}^2\neq5$.
If $K_{X_{\k(C)}}^2=6$, \autoref{cor:DP6} states that $\Autz(X)$ is a maximal connected algebraic subgroup of $\Bir(X)\simeq\Bir(C\times\p^2)$ and by \autoref{prop:Mdf 6 finite} we have $g(C)=1$, $\Autz(X)_C$ is finite and $X\simeq \Autz(X)\times^{\Autz(X)_C} S_6$, where $S_6$ is a Del Pezzo surface of degree $6$.
If $K_{X_{\k(C)}}^2=8$, then $\Autz(X)\simeq\PGL_2(\k)$ acts diagonally on a general fibre by \autoref{lem:Mfs8 restriction}. 
If $K_{X_{\k(C)}}^2=9$, then by \autoref{prop:Mdp to bundle} we can assume that $\pi$ is a $\p^2$-bundle. 
If $g(C)=1$, then \autoref{prop:trivial_action_on_C} implies that $\Autz(X)$ acts non-trivially on $C$ and the claim is \autoref{thm:DP9 max}.
If $g(C)=2$, then the claim is \autoref{lem:semistable_high_genus}.

Suppose that $B$ is a surface. By \autoref{prop:reduction to standard conic bundles}, we can assume that $\pi$ is a standard conic bundle (in particular, $X,B$ are smooth). In particular, $X_{\k(B)}$ is isomorphic to a plane curve of genus zero. 
If $X_{\k(B)}$ is irrational, then by \autoref{lem:to ruled surface} and \autoref{rmk:igf nontrivial action}, the morphism $\tau\colon B\to C$ is a ruled surface and $g(C)=1$. By \autoref{pro: igfrs}, $\Autz(X)$ is not maximal, contradicting our assumption that it is maximal. 
Therefore, $X_{\k(B)}$ is rational. In particular, $\pi$ is a $\p^1$-bundle above $B$ by \autoref{prop:reduction to standard conic bundles}. 
By \autoref{lem:descent}, we can assume that $\tau\colon B\to C$ is a ruled surface. 
By \autoref{lem:jumping fibres}, we can assume that $\tau\circ\pi$ is an $\F_b$-bundle for some $b\geq0$. 
Then the claim follows from \autoref{prop:P1-bundles}.
\end{proof}

\bibliographystyle{abbrv}
\bibliography{biblio}
\end{document}